\documentclass[noams]{compositio}
\usepackage{amsmath,arydshln,multirow}

\usepackage{hyperref}
\usepackage{arydshln}
\usepackage{cases}
\usepackage{quiver}
\usepackage{amsfonts}
\usepackage{bm}
\usepackage{arydshln}
\usepackage{amsfonts,amssymb,amscd,bbm,mathrsfs,dsfont}
\usepackage{mathrsfs}
\usepackage{pb-diagram}
\usepackage{amssymb}

\usepackage[all]{xy}

\newtheorem{Thm}{THEOREM}[section]

\newtheorem{Prop}[Thm]{PROPOSITION}
\newtheorem{Lem}[Thm]{LEMMA}
\newtheorem{Cor}[Thm]{COROLLARY}

\theoremstyle{definition}
\newtheorem{Def}[Thm]{DEFINITION}
\newtheorem{example}[Thm]{EXAMPLE}
\newtheorem{proposition-definition}[Thm]{Proposition-Definition}

\newtheorem{conj}[Thm]{CONJECTURE}

\theoremstyle{remark}
\newtheorem{Rem}[Thm]{REMARK}

\newtheorem{Convention}[Thm]{NOTATION}

\begin{document}

\title{On the Boston's Unramified Fontaine-Mazur Conjecture}
\author{Yufan Luo}
\email{yufanluo@hotmail.com}

\thanks{The author is supported by the China Scholarship Council (CSC, grant number: 201908440220).}

\begin{abstract}
This paper studies the Unramified Fontaine-Mazur Conjecture for $ p $-adic Galois representations and its generalizations. We prove some basic cases of the conjecture and provide some useful criterions for verifying it. In addition, we propose several different strategies to attack the conjecture and reduce it to some special cases.
We also prove many new results of the conjecture in the two-dimensional case. Furthermore, we also study the unramified Galois deformation rings.  Assuming the Unramified Fontaine-Mazur conjecture, we prove that the generic fiber of the unramified deformation ring is a finite direct product of fields. In particular, the unramified deformation ring has only finitely many $\overline{\mathbb{Q}}_{p}$-valued points. We also give some counterexamples to the so-called dimension conjecture for Galois deformation rings assuming the conjecture.
\end{abstract}

\maketitle

\vspace*{6pt}

\section{Introduction}

\subsection{Background}
Let $ p $ be a prime. The Fontaine-Mazur Conjecture (see \cite[Conjecture 1]{MR1363495}) is one of the central problems in modern algebraic number theory. 

\begin{conj}[Fontaine-Mazur Conjecture]\label{FM}
	Let $ K $ be a number field and $ G_{K} $ the absolute Galois group of $ K $. An irreducible $ p $-adic representation $ \rho:G_{K} \to {\rm GL}_{n}(\mathbb{Q}_{p})$ of $ G_{K} $ is geometric (i.e. it is unramifed outside a finite set of primes of $K$ and its restriction to the decomposition groups $ D_{p} $ is potentially semi-stable) if and only if it comes from algebraic geometry (i.e. if it is isomorphic to a subquotient of an \'etale cohomology group with coeffcients in $ \mathbb{Q}_{p}(r) $ for some Tate twist $ r\in \mathbb{Z} $, of a smooth projective variety over $ K $.)
\end{conj}

When we assume some standard conjectures in algebraic geometry, there is a special case of Conjecture \ref{FM}, the so-called unramified Fontaine-Mazur Conjecture (see \cite[Conjecture 5a]{MR1363495}), as follows.

\begin{conj}[Unramified Fontaine-Mazur Conjecture]\label{UFM}
	If $ K $ is a number field and $ S $ is a finite set of primes of $ K $ not containing any prime above $ p $, then any continuous homomorphism $ \rho:G_{K,S}\to {\rm GL}_{n}(\mathbb{Q}_{p}) $ has finite image where $ G_{K,S} $ is the Galois group of the maximal extension of $ K $ unramified outside $ S $.
\end{conj}

In \cite{MR1363495}, it's remarked that Conjecture \ref{FM} together with the Tate conjecture implies Conjecture \ref{UFM} without proof. Later, Mark Kisin and Sigrid Wortmann proved this remark assuming in addition that the action of $ G_{K,S} $ on the $ p $-adic \'etale cohomology of a smooth projective variety over $ K $ is semi-simple, see \cite{MR1981910}. Thanks to the work of Ben Moonen, now we have the following.

\begin{Thm}\label{hah}
	Assume that the following assumptions hold.
	\begin{enumerate}
		\item Fontaine-Mazur Conjecture \ref{FM} is true;
		\item the Tate Conjecture on algebraic cycles holds;
	\end{enumerate}
	Then Unramified Fontaine-Mazur Conjecture \ref{UFM} holds.
\end{Thm}
\begin{proof}
	It follows from \cite[Proposition 3.2]{MR1981910} and \cite[Theorem 1]{MR3959072}.
\end{proof}

Roughly speaking, the philosophy of Conjecture \ref{UFM} is that the eigenvalues of Frobenius elements must be roots of unity. By Chebotarev's density theorem, we see that there is a dense subset $ D $ of the image $ \text{Im}(\rho) $ of $ \rho $ such that the eigenvalues of all elements in $ D $ are roots of unity. Then the image of such a representation contains an open unipotent subgroup, and hence is finite by class field theory. See Proposition \ref{equi}. For more details, we refer to \cite{MR1981910}.

Conjecture \ref{UFM} for $ S=\emptyset $ implies the following, see \cite[Conjecture 5b]{MR1363495}.

\begin{conj}[Weak Unramified Fontaine-Mazur Conjecture]\label{WUFM}
	Let $ K $ be a number field, and $ G_{K,\emptyset}(p) $ the Galois group of the maximal everywhere unramified pro-$ p $-extensions of $ K $. Then any continuous homomorphism $ \rho:G_{K,\emptyset}(p)\to {\rm GL}_{n}(\mathbb{Q}_{p}) $ has finite image.
\end{conj}

Recall that, in \cite{MR0161852}, Golod and Shafarevich proved that $ G_{K,\emptyset}(p) $ can be infinite. However, Conjecture \ref{WUFM} says that there does not exist an everywhere unramified $ p $-adic representation of the absolute Galois group of a number field $ K $ with infinite image. 

Motivated by considerations concerning the deformations of representations of the global Galois group $ G_{K,S} $, Boston extended Conjecture \ref{UFM} in \cite[Conjecture 2]{MR1681626} as follows. 

\begin{conj}[Boston's generalization of unramified Fontaine-Mazur Conjecture]\label{BUFM}
	Let $ K $ be a number field and $ S $ a finite set of primes of $ K $ not containing any prime above $ p $. Then every continuous homomorphism $ \rho:G_{K,S}\to {\rm GL}_{n}(A) $, where $ A $ is a complete Noetherian local ring with finite residue field of characteristic $ p $, has finite image where $ G_{K,S} $ is the Galois group of the maximal extension of $ K $ unramified outside $ S $.
\end{conj}
The first results towards \ref{BUFM} were given by Frank Calegari and David Geraghty in \cite[Corollary 1.6]{MR3742760}. When $ K $ is a totally real field and $ n=2 $, many cases of Conjecture \ref{BUFM} were proved by Patrick B. Allen and Frank Calegari, following from the result of \cite{MR3581178}, see \cite[Corollary 3]{MR3294389}. Moreover, when $ A $ has positive characteristic, they can even prove Conjecture \ref{BUFM} for any $ n $-dimensional representations for many cases, see \cite[Theorem 1]{MR3294389}. See also Theorem \ref{abllen}. 

\subsection{Outline of the paper}

The goal of Section \ref{chapter1} is to collect all the information we need on profinite groups, $ p $-adic analytic groups, and the structure of local and global Galois groups to detect Boston's generalization of unramified Fontaine-Mazur Conjecture \ref{BUFM}. It is worth mentioning that the solvable case of Conjecture \ref{BUFM} is not a direct corollary of the abelian case, but also relies on the fact that its image in the general linear group over a pro-$p$ ring is always topologically finitely generated. See Corollary \ref{imageisfinitelygenerated}. 

In Section \ref{chapter2}, we study the Galois group $ G_{K,S}^{\text{tame}} $ of maximal tamely ramified extension of a number field $ K $ which is unramified outside a finite set $ S $ of primes of $ K $. In Section \ref{sectionaneuvaliant}, we show that Conjecture \ref{BUFM} is essentially a conjecture about the Galois group $ G_{K,S}^{\text{tame}} $.
In Section \ref{section2.2}, we show that $ G_{K,S}^{\text{tame}} $ is finite when $ K=\mathbb{Q} $, and the ramified primes are very small. In particular, Conjecture \ref{BUFM} holds in this case in a trivial way. In Section \ref{section2.3}, we show that, if Conjecture \ref{BUFM} fails, then each conjugacy class in the image of the continuous homomorphism cannot be too large or too small. See Theorem \ref{Conjugacy classes and countable dense normal subgroups}.
In Section \ref{section2.4}, we prove that the group $ G_{K,S}^{\text{tame}} $ can be topologically generated by a finite number of Frobenius conjugacy classes. The main theorem of this section is Theorem \ref{topologically generated by a finite number of Frobenius conjugacy classes}.

In Section \ref{section3.1}, we study Conjecture \ref{BUFM}, and it is proved in Theorem \ref{virtually solvable case} that Conjecture \ref{BUFM} holds if the image of the continuous homomorphism contains an open solvable subgroup. This is the basic case of Conjecture \ref{BUFM}. From this, we provide a useful criterion for verifying Conjecture \ref{BUFM} in Proposition \ref{equi}. 

In Section \ref{section3.2}, we study the reductions of Conjecture \ref{BUFM}, and we explain how to divide Conjecture \ref{BUFM} into several parts in Theorem \ref{reductionofBUFM}. It turns out that we can verify Conjecture \ref{BUFM} by restricting ourselves to examining continuous representations of global Galois groups over non-Archimedean local fields. The key to the proof is the following Theorem, which is also the first main result of this paper.

\begin{Thm}[THEOREM \ref{reduction}]
	Let $ K $ be a number field and $ S $ a finite set of primes of $ K $ not containing any prime above $ p $. Let $ A $ be a complete Noetherian local ring with finite residue field of characteristic $ p $. Let $ n\geq 1$ be a fixed positive integer. Assume the following:
	\begin{enumerate}
		\item When $\text{char}(A)=0$, suppose that for any non-Archimedean local field $F$ with residue field of characteristic $p$, any continuous representation $\rho:G_{K,S}\to {\rm GL}_{n}(F)$ has finite image.
		\item When $\text{char}(A)>0$, suppose that for any non-Archimedean local field $F$ of characteristic $p$, any continuous representation $\rho:G_{K,S}\to {\rm GL}_{n}(F)$ has finite image.
	\end{enumerate}
	Then any continuous homomorphism $ \rho:G_{K,S}\to {\rm GL}_{n}(A) $ has finite image. 
\end{Thm}
The proof consists of three steps. Firstly, we reduce the problem to the case where $A$ is reduced, since the image of $\rho$ must be a FAb profinite group. Secondly, we can further reduce to the case where $A$ is an integral domain, which follows from $A$ being Noetherian. Finally, we apply some theorems from commutative ring theory to prove the theorem.

In Section \ref{section3.5}, we verify some cases of Conjecture \ref{BUFM} for $ \mathbb{F}_{p}[[T]] $-adic representations. See Theorem \ref{positivecharcase}.

In Section \ref{chapter4}, we study Conjecture \ref{UFM}. 
In Section \ref{section4.3}, we first investigate Conjecture \ref{UFM} using Lazard's theory on $ p $-valued groups. In this direction, our main theorems  are Theorem \ref{simple} and Corollary \ref{gl2}. In section \ref{section4.4}, we show that the pro-$p$ group $G_{\mathbb{Q},\{q_{1},q_{2},q_{3}\}}(p)$ is finite under mild conditions. In Section \ref{chapter5}, we study the two-dimensional case of Conjecture \ref{UFM}. Our second main result of this paper is as follows:

\begin{Thm}[THEOREM \ref{sl21p}]
	Let $ p$ be an odd prime, and $ S=\{q_{1},\cdots,q_{d}\} $ a finite set of primes of $ \mathbb{Q} $ not containing $p$ such that $ q_{i}\not \equiv 1~\text{mod}~p^{2} $ for each $ i\in \{1,\cdots,d\} $. Let $\rho:G_{\mathbb{Q},S}\to {\rm GL}_{2}(\mathbb{Z}_{p})$ be a continuous homomorphism. Suppose that the image $\text{Im}(\overline{\rho})$ of the reduction $\overline{\rho}$ of $\rho$ modulo $p$ is trivial. Then either $\text{Im}(\rho)=\{1\}$ is trivial or $\text{Im}(\rho)$ is the kernel of the natural surjection $ {\rm SL}_{2}(\mathbb{Z}_{p})\twoheadrightarrow {\rm SL}_{2}(\mathbb{F}_{p})$. Assume further that $q_{i}$ is a $p$-th power modulo $q_{j}$ for all $ i,j\in \{1,\cdots,d\} $ with $ i\neq j$. Then $\rho$ is trivial.
	
\end{Thm}

The key to the proof is the presentation of the pro-$p$ group $G_{\mathbb{Q},S}(p)$ which is due to Helmut Koch. 

In Section \ref{chapter6}, we study the universal deformation ring of a continuous absolutely irreducible mod $p$ representation $G_{K,S}\to {\rm GL}_{n}(\mathbb{F}_{p})$ where $S$ is a finite set of primes of $K$ not containing any prime above $p$. After some preparation in Section \ref{section6.1}, we prove the finiteness of unramified deformation rings assuming a special case of Conjecture \ref{BUFM} in Section \ref{section6.2}. Moreover, we prove the following theorem which is the third main result of this paper.

\begin{Thm}(THEOREM \ref{main})
	Let $ K $ be a number field and $ S $ a finite set of primes of $ K $ not containing any prime above $p$. Let $\overline{\rho}:G_{K,S}\to {\rm GL}_{n}(\mathbb{F}_{p})$ be a continuous absolutely irreducible representation. Suppose that the Galois representation associated to any $\overline{\mathbb{Q}}_{p}$-points of the universal deformation ring $R_{\overline{\rho}}$ of $\overline{\rho}$ has finite image. Then the ring $R_{\overline{\rho}}[1/p]=\prod_{x}E_{x}$ is the finite direct product of fields $E_{x}$ where $E_{x}$ is a finite extension of $\mathbb{Q}_{p}$ indexed by $\overline{\mathbb{Q}}_{p}$-points of $R_{\overline{\rho}}$. In particular, there are only finitely many $\overline{\mathbb{Q}}_{p}$-points of $R_{\overline{\rho}}$, i.e. the set $\text{Hom}_{\mathbb{Z}_{p}}(R_{\overline{\rho}},\overline{\mathbb{Q}}_{p})$  of continuous $\mathbb{Z}_{p}$-algebra homomorphisms is finite. Moreover, assume further that $R_{\overline{\rho}}$ is $p$-torsion-free, then $R_{\overline{\rho}}$ is finite over $\mathbb{Z}_{p}$, and the universal deformation $\rho^{\text{univ}}:G_{K,S}\to {\rm GL}_{n}(R_{\overline{\rho}})$ has finite image.
\end{Thm}

This generalizes \cite[Lemma 4.14]{MR3742760}. The proof of the above theorem relies on the theory of generic fibers of deformation rings developed by Kisin, and the fact that the tangent space to any $\overline{\mathbb{Q}}_{p}$-point with finite image will be trivial. 
In Section \ref{section6.3}, we also study deformations of mod $ p $ representations of the global Galois group $ G_{K,S} $ with big image. We obtain the following result.

\begin{Thm}(THEOREM \ref{dimensionalconjecturefails}).
	Suppose that Conjecture \ref{BUFM} holds. Let $ n\geq 2 $ and $ p\geq 7 $. Let $ K $ be a number field and $ S $ a finite set of primes of $ K $ not containing any prime above $p$. If $ \overline{\rho}:G_{K,S}\to {\rm GL}_{n}(\mathbb{F}_{p}) $ is a mod $ p $ representation of $ G_{K,S} $ such that $ \text{Im}(\overline{\rho})\supset {\rm SL}_{n}(\mathbb{F}_{p}) $, then the universal deformation ring $ R_{\overline{\rho}} $ of $ \overline{\rho} $ is a finite ring.
\end{Thm}

Our argument crucially exploits the work of Manoharmayum \cite{MR3336600} on the subgroups of ${\rm GL}_{n}$ over complete local Noetherian rings with large residual image. Note that the above theorem gives counterexamples to the so-called dimension conjecture \ref{dimensionconjecture} for Galois deformation rings assuming Conjecture \ref{BUFM}. For more details, see Corollary \ref{counterexample}. 

\subsection{Notation and conventions}
Throughout this paper,
\begin{itemize}
	\item $ p $ will be a rational prime.
	\item $ \mathbb{Z} $ will be the ring of integers, and $ \mathbb{Q} $ will be the field of rational numbers and $ \overline{\mathbb{Q}} $ will be a fixed algebraic closure of $ \mathbb{Q} $.
	\item $ \mathbb{F}_{p} $ will be a finite field of order $ p $, $ \mathbb{Q}_{p} $ will be the field of $ p $-adic numbers, $ \mathbb{Z}_{p} $ will be the ring of integers of $ \mathbb{Q}_{p} $,
	and $ v_{p} $ will be the usual $ p $-adic valuation of $ \mathbb{Q}_{p} $. We also fix an algebraic closure $ \overline{\mathbb{Q}}_{p} $ of the field $ \mathbb{Q}_{p} $.
	\item If $ K $ is a field, then we denote by $ G_{K} $ the absolute Galois group of $ K $.
	
	\item If $ K $ is a number field, then we denote by $ S_{p}=S_{p}(K) $ the set of primes of $ K $ above $ p $ and $ S_{\infty}=S_{\infty}(K) $ the set of all Archimedean primes of $ K $.
	\item If $ K $ is a number field and $ S $ a finite set of primes of $ K $, then we denote by $ G_{K,S} $ (resp. $ G_{K,S}(p)$) the Galois group of maximal extension (resp. maximal $ p $-extension) of $ K $ in $ \overline{\mathbb{Q}} $ which is unramified outside $ S $. 
	\item If $ F $ is a non-Archimedean local field, then we denote by $ \mathcal{O}_{F} $ the ring of integers of $ F $ and $ \pi_{F} $ a uniformizer of $ \mathcal{O}_{F} $.
	\item If $ \mathbb{F} $ is a finite field, then we denote by $ \mathbb{F}[[T]] $ the ring of formal power series over $ \mathbb{F} $, and $\mathbb{F}((T))$ the field of formal Laurent series in one variable over $\mathbb{F}$.
	\item If $R$ is a commutative ring, then we will denote by $\text{char}(R)$ the characteristic of $R$. Recall that $\text{char}(R)$ is the natural number $n$ such that $n\mathbb{Z}$ is the kernel of the unique ring homomorphism from $\mathbb{Z}$ to $R$.
	\item If $ R $ is a commutative ring and $ n $ a positive integer, then $ M_{n}(R) $ is the space of all $ n\times n $ matrices over $ R $, $ {\rm GL}_{n}(R) $ is the general linear group over $ R $, and $ {\rm SL}_{n}(R) $ is the special linear group over $ R $.
	\item If $ G $ is a group and $ \rho:G\to {\rm GL}_{n}(R) $ is a group homomorphism, then we denote by $ \text{Im}(\rho) $ the image of $ \rho $.
\end{itemize}

\section{Preliminaries}\label{chapter1}

\subsection{The general linear group over a pro-$p$ ring and FAb profinite groups}\label{Section1.1}
By a \emph{pro-$p$ ring} we mean a complete Noetherian local ring with finite residue field of characteristic $ p $. 

\begin{Convention}
	If $ A $ is a pro-$p$ ring with finite residue field $ \mathbb{F} $, then we will denote $ {\rm GL}_{n}^{1}(A) $ to be the kernel of the natural projection $ \eta:{\rm GL}_{n}(A)\to {\rm GL}_{n}(\mathbb{F}) $, and we define
	\[ {\rm SL}_{n}^{1}(A):={\rm SL}_{n}(A)\cap {\rm GL}_{n}^{1}(A). \]
	Moreover, if $ \rho:G\to {\rm GL}_{n}(A) $ is a continuous homomorphism where $ G $ is a profinite group, then we will denote by $ \overline{\rho} $ the composition of morphisms
	\[ \overline{\rho}:G\xrightarrow{\rho} {\rm GL}_{n}(A)\xrightarrow{\eta} {\rm GL}_{n}(\mathbb{F}) ,\]
	and we say that $ \overline{\rho}$ is the \emph{reducion of $ \rho $} modulo $ \mathfrak{m}_{A} $.
\end{Convention}

The following lemma is useful.

\begin{Lem}\label{pro}
	Let $ A $ be a complete Noetherian local ring with finite residue field $ \mathbb{F} $ of characteristic $ p $ and $ n $ a fixed positive integer. Then we have the following:
	\begin{enumerate}
		\item The profinite group $ {\rm GL}_{n}^{1}(A) $ is a pro-$ p $ group.
		\item If $ G $ is a closed subgroup of $ {\rm GL}_{n}(A) $ with order (as a supernatural number) prime to $ p $, then $ G $ is finite. 
		\item If $ G $ is a torsion-free closed subgroup of $ {\rm GL}_{n}(A) $, then it is a pro-$ p $ group.
	\end{enumerate}
	
\end{Lem}
\begin{proof}
	We prove our lemma as follows:
	\begin{enumerate}
		\item It is \cite[Lemma 1.2]{MR1079842}. 
		\item Since $ {\rm GL}_{n}^{1}(A) $ is a pro-$ p $ group by the claim (i), we see that the Sylow pro-$ p $ subgroup $ H $ of $ {\rm GL}_{n}(A) $ is just the preimage under $ \eta $ of any Sylow $ p $-subgroup of the finite group $ {\rm GL}_{n}(\mathbb{F}) $. By Lemma \cite[Corollary 2.3.6]{MR2599132}, any two Sylow pro-$ p $ subgroups are conjugate. Without loss of generality, we choose one particular $ H $. Since $ H $ has finite index in $ {\rm GL}_{n}(A) $, $ [{\rm GL}_{n}(A):H] $ is a natural number. Note that the order of a profinite group is a natural number if and only if the order of the profinite group is finite, cf. \cite[Proposition 2.3.2(a)]{MR2599132}. Thus, if $ G $ is a closed subgroup of $ {\rm GL}_{n}(A) $ with order prime to $ p $, then the order $ G $ must be a natural number and hence it is finite. 
		\item By (ii), we know that for every prime $ \ell $ different from $ p $, the Sylow pro-$ \ell $ subgroup of $ G$ is finite. Since $ G $ is torsion-free, we see that $ G $ must be pro-$ p $. The claim is proved.
	\end{enumerate}
	
\end{proof}

We say that a profinite group $ G $ has a property $ \mathcal{P} $ \emph{virtually} if $ G $ has an open subgroup with property $ \mathcal{P} $. For example, a profinite group is called \emph{virtually solvable} if it has an open solvable subgroup.

A profinite group $G$ is called \emph{FAb} if $U^{\text{ab}}$ is finite for all closed subgroups $U$ of $G$ of finite index where $U^{\text{ab}}$ denotes the topological abelianization of $U$. Observe that any open subgroup (resp. continuous quotient) of a FAb profinite group is FAb.

\begin{Prop}\label{solvablequotient}
	Let $ G $ be a FAb profinite group. Then any topologically finitely generated virtually solvable (continuous) quotient of $ G $ is finite.
\end{Prop}
\begin{proof}
	We prove our claim as follows:
	\begin{itemize}
		\item Step $1$: We show that any topologically finitely generated solvable (continuous) quotient of $ G $ is finite. Since any (continuous) quotient of a FAb profinite group is also FAb, we can assume that $ G $ is a FAb solvable group, and we need to show that $G$ is finite. Let $ G=G^{(0)}\supset G^{(1)}\supset \cdots \supset G^{(i)}\supset \cdots$ be the derived series of $G$. Since $ G $ is solvable, it has finite length. Thus, it suffices to show that all (abstract) quotients $ G^{(i)}/G^{(i+1)} $ are finite. By \cite[Theorem 1.4]{MR2276769}, $G^{(1)}$ is a closed subgroup of $G$, and hence $G/G^{(1)}$ is an abelian (continuous) quotient of $G$. Since $G$ is FAb, we know that $G/G^{(1)}$ is finite, and hence $G^{(1)}$ is an open subgroup of $G$. Since $G$ is topologically finitely generated, so is $G^{(1)}$ by \cite[Prop. 2.5.5]{MR2599132}. Since $G$ is FAb, it follows that all $G^{(i)}/G^{(i+1)}$ are finite by induction. Therefore, $ G $ is finite.
		\item Step $2$: In general, we need to show that if $G$ is a topologically finitely generated, FAb virtually solvable group, then $G$ is finite. By definition, $G$ contains an open solvable subgroup $H$. Since $H$ is an open subgroup of $G$ and $G$ is FAb, we see that $H$ is FAb and topologically finitely generated. By Step $1$, $H$ is finite and so is $G$. We are done.
	\end{itemize}
\end{proof}

\subsection{$ p $-adic analytic groups}\label{padicanalyticgroup}
In this subsection, we review some basic definitions and properties of $ p $-adic analytic groups.

We say that a compact topological group $ G $ is \emph{$ p $-adic analytic} if it is isomorphic to a closed subgroup of $ {\rm GL}_{d}(\mathbb{Z}_{p}) $ for some positive integer $ d $. We say that a pro-$p$ group $G$ is \emph{powerful} if $ p $ is odd and $ G/\overline{G^{p}} $ is abelian, or if $ p=2 $ and $ G/\overline{G^{4}} $ is abelian where $\overline{G^{p}}  $ is the closure of the subgroup of $ G $ generated by the set $ \{g^{p}~|~g\in G\} $ in $ G $. Morever, a powerful group is \emph{uniform} if it is also torsion-free.

A key invariant of a $ p $-adic analytic group is its \emph{dimension} $ \dim(G) $ as a $ p $-adic manifold. Algebraically, one can define $ \dim(G) $ as $ d(U) $, where $ U $ is any uniform open pro-$ p $ subgroup of $ G $ and $ d(U) $ is the minimal cardinality of a topological generating set for $ U $.

\begin{example}\label{dimensionlessthan3}
	If $ G $ is a compact $ p $-adic analytic group, then we have the following:
	\begin{enumerate}
		\item $ \dim(G)=0 $ if and only if $ G $ is finite.
		\item $ \dim(G)=1 $ if and only if $ G $ contains an open procyclic pro-$ p $ group $ H $, i.e. $ H $ is ismorphic to $ \mathbb{Z}_{p} $.
		\item If $ \dim(G)=2 $, then $ G $ contains an open meta-procyclic pro-$ p $ group $ H $, i.e. $ H $ has a procyclic normal subgroup with procyclic quotient. See \cite[Exercise 3.11]{MR1720368}.
	\end{enumerate}
\end{example}

\begin{Lem}\cite[Theorem 9.11 and Theorem 9.14]{MR1720368}\label{closedsubgrouphavingsamediemsnion}
	Let $ G $ be a $ p $-adic analytic group and $ H $ be a closed subgroup of $ G $ such that $ \dim(H)=\dim(G) $. Then $ G $ is an open subgroup of $ G $.
\end{Lem}

\begin{Def}\cite[Section 23]{MR2810332}
	A \emph{$ p $-valued group} is a group $ G $ together with a real valued function $\omega:G\to \mathbb{R}_{>0}\cup \{\infty\}$ on $ G $, which is called \emph{$p$-valuation}, such that the following properties hold for all $ g,h\in G $:
	\begin{enumerate}
		\item $ \omega(g)>\frac{1}{p-1} $,
		\item $\omega(g)=\infty$ if and only if $g=1$,
		\item $ \omega(g^{-1}h)\geq \min\{\omega(g),\omega(h)\} $,
		\item $\omega([g,h])\geq \omega(g)+\omega(h)$,
		\item $ \omega(g^{p})=\omega(g)+1 $
	\end{enumerate}
	where $ [g,h]=g^{-1}h^{-1}gh $.  
\end{Def}

It is said to have \emph{integer values} if $ \omega(x)\in \mathbb{Z} $ for all $ g\in G,~g\neq 1 $. Note that any $p$-valued group is torsion-free.

Lazard uses the valuation $\omega$ to define a topology on $G$ by choosing the subgroups $G_{\nu}=\{g\in G~|~ \omega(g)\geq \nu\}$, with $\nu \in \mathbb{R}_{>0}$, as a fundamental system of neighborhoods of the identity. 

Let $ G $ be a pro-$ p $ group and $ \omega $ be a $ p $-valuation on $ G $ which we assume defines the topology on $ G $. Let $ g\in G $ be any element. Then we have the group homomorphism
\[ c:\mathbb{Z}\to G,~m \mapsto g^{m} ,\]
and it extends uniquely to a continuous group homomorphism $ c:\mathbb{Z}_{p}\to G$ which we always will write as $ g^{x}:=c(x) $.
The $ p $-valuation $ \omega $ on $ G $ has the following property, see \cite[Remark 26.3]{MR2810332}.
\begin{enumerate}
	\item[(vi)] $ \omega(g^{x})=\omega(g)+v_{p}(x) $ for any $ g\in G \backslash \{1\}  $ and any $ x\neq 0 $ in $ \mathbb{Z}_{p} $ where $ v_{p} $ denotes the usual $ p $-adic valuation on $ \mathbb{Q}_{p} $.
\end{enumerate}

\begin{Def}\cite[Section 26]{MR2810332}
	A pro-$p$ group is called \emph{$p$-valuable} if there exists a $p$-valuation $\omega$ on $G$ which defines the topology on $G$ and $G$ has finite rank. (See \cite[Section 26]{MR2810332} for the precise definition of finite rank.) A $ p $-valuable group $ G $ is called \emph{$ p $-saturated} if it has the following property: if $ g\in G $ with $ \omega(g)>\frac{p}{p-1} $, then there exists a $ h\in G $ such that $ g=h^{p} $.
\end{Def}

\begin{Thm}[Lazard]\cite{MR209286}\label{lazard}
	A compact topological group is $p$-adic analytic if and only if it contains an open $p$-saturated pro-$ p $ subgroup.
\end{Thm}

The following example is important in the context of Conjecture \ref{UFM}.

\begin{example}\label{padicvaluedexample}
	Let $ L/\mathbb{Q}_{p} $ be a finite extension and $ v=v_{L} $ its additive valuation normalized by $ v(p)=1 $. Fix a positive integer $ n $, we consider the space $ M_{n}(L) $ of all $ n\times n $ matrices over $ L $. For any nonzero matrix $ A=(a_{ij}) $, we put
	\[ w(A):=\min_{i,j}{v(a_{ij})}, \]
	and $ w(0):=\infty $. Let 
	\begin{align*}\label{ha}
		\mathcal{G}_{n}(L):=&\left\lbrace g\in {\rm GL}_{n}(L)~|~w(g-1)>\dfrac{1}{p-1}\right\rbrace ,
	\end{align*}
	and
	\[ \omega(g):=w(g-1)\qquad \text{for }g\in  \mathcal{G}_{n}(L). \]
	Then one can check that $ \omega $ is a $ p $-valuation on $  \mathcal{G}_{n}(L) $ and $ \omega $ defines the subspace topology on the open subgroup $  \mathcal{G}_{n}(L) $ of $ {\rm GL}_{n}(\mathcal{O}_{L}) $, cf. \cite[Example 23.2]{MR2810332}. Then $  \mathcal{G}_{n}(L) $ is a torsion-free open subgroup of $ {\rm GL}_{n}(\mathcal{O}_{L}) $.
\end{example}
For any real number $ \nu>0 $, we put
\[  \mathcal{G}_{n}(L)_{\nu+}:=\{g\in  \mathcal{G}_{n}(L)~|~\omega(g)>\nu\}. \]

\begin{Convention}\label{definitiona}
	For any integer $ i\geq 1 $, we define
	\[ {\rm GL}_{n}^{i}(\mathcal{O}_{L}):=\ker({\rm GL}_{n}(\mathcal{O}_{L})\to {\rm GL}_{n}(\mathcal{O}_{L}/\pi^{i}_{L}\mathcal{O}_{L}))=1+\pi_{L}^{i}M_{n}(\mathcal{O}_{L}),\]
	and 
	\[ {\rm SL}_{n}^{i}(\mathcal{O}_{L}):={\rm GL}_{n}^{i}(\mathcal{O}_{L})\cap {\rm SL}_{n}(\mathcal{O}_{L}). \]
\end{Convention}
Note that for any real number $ \nu>0 $,
we have
\[ \mathcal{G}_{n}(L)_{\nu+}={\rm GL}_{n}^{\eta(\nu)}(\mathcal{O}_{L}), \]
where $ \eta(\nu) $ is the smallest integer which is larger than $ e_{L/\mathbb{Q}_{p}}\cdot \nu $. For example, when $ L=\mathbb{Q}_{p} $ where $ p\geq 3 $, we have $  \mathcal{G}_{n}(\mathbb{Q}_{p})={\rm GL}_{n}^{1}(\mathbb{Z}_{p}) $ and $  \mathcal{G}_{n}(\mathbb{Q}_{p})_{m+}={\rm GL}_{n}^{m+1}(\mathbb{Z}_{p}) $ for any integer $ m\geq 1 $.

\subsection{Just-infinite profinite groups}\label{section1.3}
An infinite profinite group $G$ is called \emph{just-infinite} if all non-trivial closed normal subgroup of $G$ have finite index.  For example, the pro-$ p $ groups $ {\rm SL}^{1}_{n}(\mathbb{Z}_{p}) $ and $ {\rm SL}_{n}^{1}(\mathbb{F}_{p}[[T]]) $ are just-infinite where $ p>2 $, see \cite{MR1483894}. The following lemma is known. 

\begin{Lem}\label{justinfinitequotient}\cite[Theorem 3]{MR1765119}
	Let $ G $ be a topologically finitely generated infinite profinite group which is a virtually pro-$ p $ group. Then $ G $ admits a just-infinite quotient which is also a virtually pro-$ p $ group.
\end{Lem}

We prove the following theorem:

\begin{Thm}\label{densenormalsubgroup}
	Let $G$ be an infinite profinite group such that it admits a topologically finitely generated just-infinite quotient which is not virtually abelian. Then we have the following:
	\begin{enumerate}
		\item There is no dense normal subgroup $N$ of $G$ such that $N$ has only countably many conjugacy classes.
		\item There is no countable subset $D$ of $G$ such that normal subgroup of $G$ generated by $D$ is dense in $G$ and the conjugacy class $g^{G}$ of $g$ in $G$ is finite (as a set) for any $g\in D$.
	\end{enumerate}
\end{Thm}

To prove Theorem \ref{densenormalsubgroup}, we need the following easy lemma.
\begin{Lem}\label{cardinality}
	Let $ G $ be an abstract group and $ A $ be an infinite subset of $ G $. Then $ A $ and the subgroup $ \left\langle A\right\rangle  $ of $ G $ generated by $ A $ have the same cardinality.
\end{Lem}
\begin{proof}
	Put $ A^{-1}:=\{a^{-1}~|~a\in A\} $, and  denote by $ \text{Pr}_{n}(A) $ the set of all finite products of the form $ a_{1}^{\pm 1}\cdots a_{n}^{\pm 1} $, where $ a_{1},\cdots,a_{n}\in A $ for any integer $ n\geq 1 $. By definition, we have $ \left\langle A\right\rangle=\bigcup_{n\geq 1}^{\infty}\text{Pr}_{n}(A) $. Note that for any $ n\geq 1 $, we have
	\[ |\text{Pr}_{n}(A)|\leq |A\cup A^{-1}|^{n}=|A| \]
	since $ |A| $ is an infinite cardinal. Thus, we have
	\[  |\left\langle A\right\rangle|=\left|  \bigcup_{n\geq 1}^{\infty}\text{Pr}_{n}(A)\right| \leq \sum_{n\geq 1}|\text{Pr}_{n}(A)|\leq \sum_{n\geq 1}|A|=|A|.  \]
	On the other hand, it is clear that $ |A|\leq | \left\langle A\right\rangle| $. We are done.
\end{proof}

\begin{Cor}\label{countable}
	Let $G$ be an abstract group and $A$ a countable subset of $ G$ such that $G=\left\langle A\right\rangle$. Then the group $G$ is countable.
\end{Cor}
\begin{proof}
	If $A$ is finite, then it is clear that $G$ is countable. If $A$ is infinite, then our claim follows from Lemma \ref{cardinality}.
\end{proof}
\begin{proof}[Proof of Theorem \ref{densenormalsubgroup}]
	We prove our theorem as follows:
	\begin{enumerate}
		\item Suppose that a counter-example exists. Let $H$ be a topologically finitely generated just-infinite quotient of $G$ which is not virtually abelian. By our assumption, there exists a dense normal subgroup of $N$ of $H$ such that $N$ has only countably many conjugacy classes. By \cite[Corollary 1.15]{MR2995181}, we know that $H=N$. But this contradicts \cite[Theorem 1.1]{MR3993799}. 
		\item Suppose that a counter-example exists. By our assumption and Corollary \ref{countable}, the normal subgroup $N$ of $G$ generated by $D$ is a countable dense subgroup of $G$. But this contradicts the previous assertion. We are done.
	\end{enumerate}
\end{proof}

\subsection{Some results on local Galois groups and global Galois groups}\label{someresultsonlocalgalosgroup}

\subsubsection{Local Galois groups}
Let $ p $ be a prime.
\begin{Prop}\label{tamelyramified}
	Let $ k $ be a local field with residue field of characteristic $ p $, $ k_{tr} $ the maximal tamely ramified extension of $ k $ and $ \text{Gal}(k_{tr}/k) $ the Galois group of the maximal tamely ramified extension of $ k $. Let $ q $ denote the cardinality of the residue field of $ k $. Then we have the following:
	\begin{enumerate}
		\item The Galois group $ \text{Gal}(k_{tr}/k) $ is isomorphic to the profinite group topologically generated by two elements $ \tau,\sigma $ with the only relation
		\begin{align}\label{relation}
			\sigma \tau \sigma^{-1}=\tau^{q},
		\end{align}
		where the element $ \tau $ is a generator of the inertia group, and $ \sigma $ is a Frobenius lift.
		\item Let $ F $ be a non-Archimedean local field of characteristic zero and $ \rho:\text{Gal}(k_{tr}/k)\to {\rm GL}_{n}(F) $ a continuous representation. Then the roots of the characteristic polynomial of $ \rho(\tau) $ are roots of unity. In particular, $ \rho(\tau) $ has finite order if and only if it's semisimple (i.e. diagonalizable over the algebraic closure of $ F $).
		\item  Let $ F $ be a non-Archimedean local field of positive characteristic and $ \rho:\text{Gal}(k_{tr}/k)\to {\rm GL}_{n}(F) $ a continuous representation. Then $ \rho(\tau) $ has finite order.
	\end{enumerate}
\end{Prop}
\begin{proof}
	The claim (i) is \cite[Theorem 7.5.3]{MR2392026}. If $ F $ is a non-Archimedean local field of characteristic zero or of positive characteristic, then we can assume that $ \rho(\tau) $ is upper triangular after taking finite extension of $ F $. Let $ \lambda_{1},\cdots,\lambda_{n} $ denote the eigenvalues of $ \rho(\tau) $. Using the relation (\ref{relation}), we see that 
	\[ \{\lambda_{1},\cdots,\lambda_{n}\}=\{\lambda_{1}^{q},\cdots,\lambda_{n}^{q}\}. \]
	It follows that $ \lambda_{i}^{q^{n!}-1}=1 $ for all $ i $, and hence $ \lambda_{i} $ are roots of unity. Hence $ \rho(\tau)^{m} $ is unipotent for some positive integer $ m $. Moreover, if $ F $ has positive characteristic $ \ell $, then we see that $ \rho(\tau)^{m} $ is a $ \ell $-element, and hence $ \rho(\tau) $ has finite order. This completes the proof of the claim (ii) and (iii).
	The proposition is proved.
\end{proof}
\begin{Rem}
	In (ii), the image $\rho(\tau)$ of $\rho$ does not necessarily have finite order, see Theorem \ref{notalocalproblem} for a counterexample.
\end{Rem}

By a \textit{$ p $-extension} we mean a Galois extension whose Galois group is a $ p $-group.

\begin{Convention}
	For any prime $ \mathfrak{p} $ of a number field $ K $, we denote by $ G_{\mathfrak{p}} $ (resp. $ G_{\mathfrak{p}}(p) $) the absolute Galois group of the completion $ K_{\mathfrak{p}} $ of $ K $ at $ \mathfrak{p} $ (resp. the Galois group $ \text{Gal}(K_{\mathfrak{p}}(p)/K_{\mathfrak{p}}) $ of the maximal $ p $-extension $ K_{\mathfrak{p}}(p) $ of $ K_{\mathfrak{p}} $ inside $ \overline{\mathbb{Q}}_{p} $). Let $ I_{\mathfrak{p}}\subset G_{\mathfrak{p}} $ (resp. $ I_{\mathfrak{p}}(p)\subset G_{\mathfrak{p}}(p)$) denote the inertia subgroup, and let $ P_{\mathfrak{p}}\subset I_{\mathfrak{p}} $ denote the wild inertia subgroup.
	We denote by $ N(\mathfrak{p}) $ the number of elements of the residue field $ \kappa_{\mathfrak{p}} $ of $ \mathfrak{p} $.
\end{Convention}

\begin{Lem}\cite[Section 8.5]{MR1930372}\label{redundant}
	Let $ K $ be a number field. Then the following type of primes of $ K $ cannot ramify in a $ p $-extension of $ K $:
	\begin{enumerate}
		\item non-Archimedean primes $ \mathfrak{p} $ with $ N(\mathfrak{p})\not \equiv 0,1~\text{mod}~p $,
		\item complex primes,
		\item real primes for $ p\neq 2 $.
	\end{enumerate}
	Moreover, if $L/K $ is a finite $ p $-extension which is ramified at $ \mathfrak{p} $, then $ L/K $ is tamely ramified at $ \mathfrak{p} $ if and only if $ N(\mathfrak{p})\equiv 1~\text{mod}~p $.
\end{Lem}

\begin{Thm}\label{local}
	Let $ K $ be a number field and $ \mathfrak{p} $ a non-Archimedean prime of $ K $ with $ \mathfrak{p}\notin S_{p}(K) $. Suppose that the characteristic of the residue field of $ K_{\mathfrak{p}} $ is $ \ell $. Then 
	\begin{enumerate}
		\item If $ N(\mathfrak{p}) \equiv 1~\text{mod}~p $, then $ G_{\mathfrak{p}}(p) $ is the pro-$ p $ group on two generators $ \sigma,\tau $ subject to the relation $ \sigma\tau\sigma^{-1}=\tau^{q} $ where the element $ \tau $ is a generator of the inertia group, $ \sigma $ is a Frobenius lift and $ q=N(\mathfrak{p}) $.
		\item If $ \rho:G_{\mathfrak{p}}\to {\rm GL}_{n}(\overline{\mathbb{Q}}_{p}) $ is a continuous homomorphism, then $ \rho(P_{\mathfrak{p}}) $ is finite and $ \rho(I_{\mathfrak{p}}) $ contains an open unipotent subgroup.
		\item Let $ F $ be a non-Archimedean local field of characteristic $ p $. If $ \rho:G_{\mathfrak{p}}\to {\rm GL}_{n}(F) $ is a continuous representation, then $ \rho(I_{\mathfrak{p}}) $ has finite image.
	\end{enumerate}
\end{Thm}
\begin{proof}
	Firstly, note that the prime $ \ell $ is different from $ p $ by our assumption. The claim (i) follows from Lemma \ref{redundant} and (i) in Proposition \ref{tamelyramified} by passing to the maximal $ p $-quotient. For the claim (ii), we may assume that $ \text{Im}(\rho)\subset {\rm GL}_{n}(\mathcal{O}_{L}) $ for some finite extension $ L/\mathbb{Q}_{p} $. Since $ P_{\mathfrak{p}} $ is a pro-$ \ell $ group and $ {\rm GL}_{n}^{1}(\mathcal{O}_{L}) $ is a pro-$ p $ group, we see that the composite morphism $ \rho(P_{\mathfrak{p}})\hookrightarrow {\rm GL}_{n}(\mathcal{O}_{L})\twoheadrightarrow {\rm GL}_{n}(\kappa_{L}) $ is injective. In particular, $ \rho(P_{\mathfrak{p}}) $ is finite. After taking finite extension of $ K_{\mathfrak{p}} $, we may assume that $ \rho(P_{\mathfrak{p}}) $ is trivial, and then it must factor through the Galois group of the maximal tamely ramified extension of $ K_{\mathfrak{p}} $. Thus, the claim (ii) follows from (ii) in Proposition \ref{tamelyramified}. Similarly, the claim (iii) follows from (iii) in Proposition \ref{tamelyramified}. We are done.
\end{proof}

\subsubsection{Global Galois groups}

Recall that the \textit{maximal pro-$ p $ quotient} $  G^{(p)} $ of a profinite group $ G $ is the quotient group of $ G $ such that every pro-$ p $ quotient of $ G $ factors through $ G^{(p)} $.

\begin{Def}
	Let $ K $ be a number field and $ S $ be a finite set of primes of $ K $. 
	\begin{enumerate}
		\item A finite extension $ L/K$ of fields is called \emph{$ S $-ramified} if it is unramified outside $ S $. The composite of two such extensions in $ \overline{\mathbb{Q}} $ is again $ S $-ramified, so it makes sense to consider the composite $ K_{S} $ of all the finite $ S $-ramified extensions of $ K $. Because of maximality, $ K_{S}/K $ is Galois. Define $ G_{K,S}:= \text{Gal}(K_{S}/K) $ and we say that $  G_{K,S} $ is the \emph{Galois group of the maximal extension of $ K $ unramified outside $ S $}.
		\item  We will denote by $ G_{K,S}^{\text{tame}} $ the Galois group of the maximal extension of $ K $ in $ K_{S} $ with at most tame ramification at the primes in $ S $ where we assume that the Archimedean primes of $K$ are always tame. It is a quotient of $ G_{K,S}  $.
		\item Let $ K_{S}(p) $ denote the maximal $ p $-extension of $K $ in $ \overline{ \mathbb{Q}} $ which is unramified outside $ S $ (i.e. it is the compositum of all finite $ p $-power degree extensions of $ K $ unramified outside $ S $.). Put $ G_{K,S}(p)=\text{Gal}(K_{S}(p)/K) $ and we say that $ G_{K,S}(p) $ is the \emph{Galois group of the maximal $ p $-extension of $ K $ unramified outside $ S $}.
	\end{enumerate}
\end{Def}	
\begin{Rem}\label{trivial}
	\begin{enumerate}
		\item Note that $ G_{K,S}(p) $ is the maximal pro-$ p $ quotient of $ G_{K,S} $.
		\item One can always remove redundant primes from $ S $, as in Lemma \ref{redundant}, to get a subset $ S_{\min}\subset S $ which satisfies $ G_{K,S}(p)=G_{K,S_{\min}}(p) $. Furthermore, if all primes $ \mathfrak{p}_{i} $ in $ S_{\min} $ satisfies $ N(\mathfrak{p}_{i})\equiv 1~\text{mod}~p $, then $ G_{K,S_{\min}}(p) $ is a quotient of $ G_{K,S}^{\text{tame}} $.
	\end{enumerate}
\end{Rem}

If $ K $ is a number field, then we denote by $ S_{p}=S_{p}(K) $ the set of primes of $ K $ above $ p $ and $ S_{\infty}=S_{\infty}(K) $ the set of all Archimedean primes of $ K $. 
For simplicity, the following easy lemma will be used often without an explicit reference.

\begin{Lem}\label{1}
	Let $ K $ be a number field and $ S $ a finite set of primes of $ K $. If $ U $ is an open subgroup of $ G_{K,S} $, then there is a finite extension $ K' $ of $ K$ such that $ G_{K',S'}=U $ where $ S'$ is the set of primes of $ K' $ lying above the primes in $ S $. In particular, if  $ S\cap S_{p}(K)=\emptyset $, then $ S'\cap S_{p}(K')=\emptyset $.
\end{Lem}

The following theorem collects the known information about the structure of the global Galois group.

\begin{Thm}\label{global structure}
	Let $ K $ be a number field, and $ S $ a finite set of primes of $ K $. Then we have the following:
	\begin{enumerate}
		\item The profinite group $G_{K,S}^{\text{tame}}$ is FAb.
		\item Suppose that $ S\cap S_{p}=\emptyset $. Then the pro-$ p $ group $ G_{K,S}(p) $ is FAb.
		\item The pro-$ p $ group $ G_{K,S}(p) $ is topologically finitely generated.
		\item Suppose that $ S\cap S_{p}=\emptyset $ and the class number of $ K $ is prime to $ p $. Then the minimal cardinality $ d(G_{K,S}(p)) $ of a topological generating set for the pro-$ p $ group $ G_{K,S}(p) $ satisfies $ d(G_{K,S}(p))\leq |S|$. In particular, if $ |S|=\emptyset $, then $ G_{K,S}(p) $ is trivial; if $ |S|=1 $, then $ G_{K,S}(p) $ is a finite cyclic group.
		\item If $ K=\mathbb{Q} $, then we have $ G_{\mathbb{Q},S_{\infty}}=\{1\} $.
		\item If $ K=\mathbb{Q} $ and $ S=\{q_{1},\cdots,q_{d}\}$ with $ q_{i}\equiv 1~\text{mod}~p $ for all $ i=1,\cdots,d $. If $ p=2 $, we assume further that $ S\supset S_{\infty} $. We put $ f_{i}:=v_{p}(q_{i}-1) $ where $ v_{p} $ denotes the usual $ p $-adic valuation of $ \mathbb{Q}_{p} $.
		\begin{enumerate}
			\item The (topological) abelianization $ G_{\mathbb{Q},S}(p)^{\text{ab}} $ of the pro-$ p $ group $ G_{\mathbb{Q},S}(p) $ is isomorphic to $\prod_{i=1}^{d}(\mathbb{Z}/p^{f_{i}}\mathbb{Z}) $.
			\item The pro-$ p $ group $ G_{\mathbb{Q},S}(p) $ has a presentation $ F/R $ where $ F $ is the free pro-$ p $ group on $ x_{1},\cdots,x_{d} $ and $ R=(r_{1},\cdots,r_{d}) $ with 
			\begin{align}\label{relationaaa}
				r_{i}&=x_{i}^{q_{i}-1}[x_{i}^{-1},y_{i}^{-1}],\notag\\
				y_{i}&\equiv \prod_{j\neq i}x_{j}^{L_{ij}},~\text{mod}~F^{p}[F,F],~L_{ij} \in \mathbb{Z},
			\end{align}
			where 
			\begin{enumerate}
				\item $ [x_{i}^{-1},y_{i}^{-1}]:=x_{i}y_{i}x_{i}^{-1}y_{i}^{-1} $ is the commutator and $ F^{p}[F,F] $ is the closed subgroup generated by $ p $-th power of elements of $ F $ and the commutators of $ F $;
				\item the image $ \tau_{i} $ of $ x_{i} $ in $ G_{\mathbb{Q},S}(p) $ is a generator of the (cyclic) intertia group at a fixed prime $ \mathfrak{Q}_{i} $ above $ q_{i} $;
				\item the image $ \sigma_{i} $ of $ y_{i} $ is a lifting of the Frobenius automorphism at $ q_{i} $;
				\item Let $ g_{i} $ denote a primitive root mod $ q_{i} $ for any $ i $. Then we have
				\[ q_{i}\equiv g_{j}^{-L_{ij}}~\text{mod}~q_{j}, \]
				Let $\ell_{ij}$ denote the image of $L_{ij}$ in $\mathbb{Z}/p\mathbb{Z}$. Then $\ell_{ij}=0 $ if and only if $ q_{i} $ is a $ p $-th power modulo $ q_{j} $.
			\end{enumerate}
		\end{enumerate}
	\end{enumerate}
\end{Thm}
\begin{proof}
	By \cite[Section 5.2.2, Chapter II]{MR1941965}, the topological abelianization of the profinite group $G_{F,T}^{\text{tame}}$ is finite for any pair $(F,T)$ where $F$ is a number field and $T$ is a finite set of primes of $F$. Thus, if $U$ is an open subgroup of $G_{K,S}^{\text{tame}}$, then there is a finite extension $ K' $ of $ K $ and a finite set $ S' $ of primes of $ K' $ such that $ G_{K',S'}^{\text{tame}}=U $. It follows that the topological abelianization of $ U $ is finite. Thus, the group $G_{K,S}^{\text{tame}}$ is FAb. This proves the claim (i).
	
	For the claim (ii), note that the pro-$p$ group $G_{K,S}(p)$ is a quotient of $G_{K,S}^{\text{tame}}$ since $S\cap S_{p}=\emptyset$. By the claim (i), we see that $G_{K,S}(p)$ is FAb.
	
	The claims (iii) and (iv) follow from \cite[Theorem 11.5 and Theorem 11.8]{MR1930372}. The claim (v) is \cite[Theorem 2.18, Section III.2]{MR1697859}. The claim (vi) follows from \cite[Section 11.4]{MR1930372}. This completes the proof of our theorem.
\end{proof}
\begin{Rem}\label{remark}
	\begin{enumerate}
		\item It is not known whether or not the profinite group $ G_{K,S} $ is topologically finitely generated. This has been asked by Shafarevich before 1962, see \cite[pp. 283-294]{MR977275}. Moreover, we don't even know if $ G_{K,S}^{tame} $ is topologically finitely generated although this has been conjectured to be the case by Harbater, see \cite[Conjecture 2.1]{MR1299733} and \cite[Proposition 4.1]{MR4177534}. However, we have the following weaker result: the group $G_{K,S}$ is topologically generated by a finite number of conjugacy classes, see \cite[Theorem 10.2.5]{MR2392026}.
		\item When $ S\cap S_{p}=\emptyset $, then group $ G_{K,S}(p) $ can be infinite. For example, if $ K=\mathbb{Q} $, $ p>2 $ and $ |S_{\min}|\geq 4 $, then $ G_{\mathbb{Q},S}(p) $ is infinite (Golod-Shafarevich).
		\item The pro-$ p $ group $ G_{\mathbb{Q},S}(p) $ is still mysterious. The key point is that it is not known how to formulate for $ \sigma_{i} $ in terms of $ \tau_{1},\cdots,\tau_{d} $ in (vi).
		\item Since the quotient $F/(F^{p}[F,F])$ is a $\mathbb{F}_{p}$-vector space, we can write $(\ref{relationaaa})$ as $	y_{i} \equiv \prod_{j\neq i}x_{j}^{\ell_{ij}},\text{mod}~F^{p}[F,F]$ in (vi). 
	\end{enumerate}
\end{Rem}

However, the following corollary shows that we don't need to worry about the question whether or not the global Galois group $ G_{K,S} $ is topologically finitely generated when we consider linear representations of $ G_{K,S}$ over pro-$ p $ rings.

\begin{Cor}\label{imageisfinitelygenerated}
	Let $ K $ be a number field and $ S $ a finite set of primes of $ K $. If $ A $ is a complete Noetherian local ring of finite residue field of characteristic $ p $, and $ \rho:G_{K,S}\to {\rm GL}_{n}(A) $ is a continuous homomorphism, then the image of $ \rho $ is topologically finitely generated.
\end{Cor}
\begin{proof}
	Let $ \overline{\rho} $ denote the reduction of $ \rho $ modulo $ \mathfrak{m}_{A} $. Then we have the following extension of profinite groups:
	\begin{align*}
		1\to \rho(\ker(\overline{\rho}))\to \text{Im}(\rho)\to \text{Im}(\overline{\rho})\to 1,
	\end{align*}
	where $ \text{Im}(\overline{\rho}) $ is a finite group. By Lemma \ref{pro}, we know that $ {\rm GL}_{n}^{1}(A) $ is a pro-$ p $ group. It follows that $ \rho(\ker(\overline{\rho})) $ is a pro-$ p $ subgroup of $ \text{Im}(\rho) $ since $\rho(\ker(\overline{\rho}))$ is a closed subgroup of $ {\rm GL}_{n}^{1}(A) $. Since $ \ker(\overline{\rho}) $ is an open subgroup of $ G_{K,S} $, there is a finite extension $ K' $ of $ K $ and a finite set $ S' $ of primes of $ K' $ with $ S'\cap S_{p}(K')=\emptyset $ such that $ \ker(\overline{\rho})=G_{K',S'} $. Thus, the restriction $ \rho|_{\ker(\overline{\rho})} $ of $ \rho $ on $ \ker(\overline{\rho}) $ will factor through $ G_{K',S'}(p) $.
	By Theorem \ref{global structure}, $ G_{K',S'}(p) $ is topologically finitely generated. It follows that $ \rho(\ker(\overline{\rho})) $ is topologically finitely generated. 
	Since $ \rho(\ker(\overline{\rho})) $ is an open normal subgroup of $ \text{Im}(\rho) $, we obtain that $ \text{Im}(\rho) $ is also topologically finitely generated. Indeed, $ \text{Im}(\rho) $ is topologically generated by the generators of $ \rho(\ker(\overline{\rho})) $ and lifts of the finite group $ \text{Im}(\overline{\rho}) $. This finishes the proof.
\end{proof}

\section{Galois group of the maximal tamely ramified extension of a number field}\label{chapter2}

In this section, we study the Galois group $ G_{K,S}^{\text{tame}} $ of the maximal tamely ramified extension of a number field $ K $ which is unramified outside a finite set $ S $ of primes of $ K $. 
\subsection{An equivalent description of Conjecture \ref{BUFM}}\label{sectionaneuvaliant}
In this subsection, we give an equivalent description of Conjecture \ref{BUFM} as follows:

\begin{conj}\label{tameramifieddis}
	Let $ K $ be a number field and $ S $ a finite set of primes of $ K $. Let $ A $ be a complete Noetherian local ring with finite residue field of characteristic $ p $. Then any continuous homomorphism $\rho:G_{K,S}^{\text{tame}}\to {\rm GL}_{n}(A)$ has finite image.
\end{conj}

\begin{Prop}
	The following are equivalent:
	\begin{enumerate}
		\item Conjecture \ref{BUFM}.
		\item Conjecture \ref{tameramifieddis}.
	\end{enumerate}
\end{Prop}
\begin{proof}
	We prove our claim as follows:
	\begin{itemize}
		\item (i) $ \implies $ (ii): Let $\rho:G_{K,S}^{\text{tame}}\to {\rm GL}_{n}(A)$ be a continuous homomorphism. We need to show that the group $\text{Im}(\rho)$ is finite. By Lemma \ref{pro}, we know that $\text{Im}(\rho)$ contains an open pro-$p$ group. After taking a finite extension of $K$, we can assume that $\text{Im}(\rho)$ is a pro-$p$ group. By Lemma \ref{redundant}, we see that $ \rho(I_{v}) $ is trivial for all prime $ v\in S_{p} $. Thus, after removing redundant primes from $S$, we can assume that $S\cap S_{p}=\emptyset$. Then Conjecture \ref{BUFM} implies that $ \rho $ has finite image. We are done.
		\item (ii) $ \implies $ (i): Let $ \rho:G_{K,S}\to {\rm GL}_{n}(A) $ be a continuous homomorphism as in Conjecture \ref{BUFM} where $ S\cap S_{p}=\emptyset$. By Lemma \ref{pro}, we know that $\text{Im}(\rho)$ contains an open pro-$p$ group. After taking a finite extension of $K$, we can assume that $\text{Im}(\rho)$ is a pro-$p$ group. By Lemma \ref{redundant}, we see that $\rho$ factors through $G_{K,S}^{\text{tame}}$ since $S\cap S_{p}=\emptyset$. Then Conjecture \ref{tameramifieddis} implies that $ \rho $ has finite image. This completes the proof of the proposition.
	\end{itemize}
\end{proof}

The following theorem confirms a basic case of Conjecture \ref{tameramifieddis}.

\begin{Thm}\label{tamelyvirtually solvable case}
	Let $ K $ be a number field and $ S $ a finite set of primes of $ K $. Then any topologically finitely generated virtually solvable (continuous) quotient of $\text{G}_{K,S}^{\text{tame}} $ is finite.
\end{Thm}
\begin{proof}
	It follows from Theorem \ref{global structure} and Proposition \ref{solvablequotient}.
\end{proof}

\subsection{Galois extension of $ \mathbb{Q} $ which is tamely ramified at small primes}\label{section2.2}

In \cite[Theorem 2.6]{MR1299733}, Harbater proved that the Galois group $ G_{\mathbb{Q},\{p\}}^{\text{tame}} $ is cyclic of order $ p-1 $ if $ p<23 $. Inspired by this work, we prove the following theorem:

\begin{Thm}\label{small primes}
	Let $ S=\{p_{1},\cdots,p_{d}\} $ be a finite set of non-Archimedean  primes such that $ \prod_{j=1}^{d}p_{j}<60.1 $. Then the Galois group $ G_{\mathbb{Q},S}^{\text{tame}} $ is finite. In particular, Conjecture \ref{tameramifieddis} holds in this case. 
\end{Thm}

To prove this theorem, we need the following lemmas.

\begin{Lem}\cite[Section 2]{MR1299733}\label{discriminant}
	Let $ K $ be a number field and $ \Delta_{K} $ its (absolute) discriminant. Let $ p $ be a prime dividing $ \Delta_{K} $. Then we have
	\[ v_{p}(\Delta_{K})\leq \sum_{\mathfrak{P}|p}f_{\mathfrak{P}}(e_{\mathfrak{P}}-1+e_{\mathfrak{P}}v_{p}(e_{\mathfrak{P}})),\]
	where $ e_{\mathfrak{P}} $ (resp. $ f_{\mathfrak{P}} $) is the ramification index (resp. the residue extension degree) of a prime ideal $ \mathfrak{P} $ lying over $ p $, and $ v_{p} $ is the usual $ p $-adic valuation of $ \mathbb{Q} $.
\end{Lem}

\begin{Thm}\cite[Theorem 1]{MR441918}\label{formula}
	Let $ K $ be a number field of signature $ (r_{1},r_{2}) $, i.e. $ K $ has $ r_{1} $ real, and $ r_{2} $ pairs of complex conjugate, embeddings. Then we have
	\begin{align*}
		|\Delta_{K}|\geq (60.1)^{r_{1}}(22.2)^{2r_{2}}e^{-254}.
	\end{align*}
\end{Thm}
\begin{Rem}
	Assuming the generalized Riemann Hypothesis, one can get better bounds for $ |\Delta_{K}| $, cf. \cite{MR1061762}.
\end{Rem}

\begin{Lem}\cite[Corollary 2.6.6]{MR2599132}\label{nobound}
	Let $ G $ be a second countable profinite group. Then there is a chain of open normal subgroups $G=G_1\geq G_2\geq \cdots \geq G_{i}\geq \cdots $ of $G$ such that $\bigcap_{i=1}^{\infty} G_i=\{1\}$. In particular, if $ G $ is infinite, then there is no bound on the orders of finite quotients of $ G $.
\end{Lem}

\begin{proof}[Proof of Theorem \ref{small primes}]
	Since $S$ does not contain $S_{\infty}$, the extension $\mathbb{Q}_{S}/\mathbb{Q}$ is totally real, and any finite extension $ K $ of $\mathbb{Q}$ inside $ \mathbb{Q}_{S} $ is totally real. By Lemma \ref{nobound}, it suffices to show that there is a constant $ C $ such that the degree $n=[K:\mathbb{Q}]$ of any finite Galois extension $ K/\mathbb{Q} $ which is tamely ramified at $ S $ must be less than $ C $. Indeed, for such $ K $, by Lemma \ref{discriminant} we have
	\begin{align}
		|\Delta_{K}|^{1/n}\leq& \prod_{j=1}^{d}p_{j}^{1+v_{p_{j}}(e_{p_{j}})-1/e_{p_{j}}} \notag \\
		\leq &\prod_{j=1}^{d}p_{j}^{1-1/n}
	\end{align}
	since $ K/\mathbb{Q} $ is tamely ramified at each $ p_{j} $. On the other hand, by Theorem \ref{formula} we have
	\begin{align}
		|\Delta_{K}|^{1/n}\geq 60.1 \cdot e^{-254/n},
	\end{align}
	since $ K/\mathbb{Q} $ is totally real extension. Thus, we have
	\[ 60.1 \cdot e^{-254/n}\leq (\prod_{j=1}^{d}p_{j})^{1-1/n}. \]
	Since $ \prod_{j=1}^{d}p_{j}<60.1 $, this implies that
	\[ n\leq \dfrac{254-\log(\prod_{j=1}^{d}p_{j})}{\log(60.1/\prod_{j=1}^{d}p_{j})}. \]
	Finally, take $ C:=  \dfrac{254-\log(\prod_{j=1}^{d}p_{j})}{\log(60.1/\prod_{j=1}^{d}p_{j})}$. This completes our proof.
\end{proof}

\subsection{Conjugacy classes and dense normal subgroups}\label{section2.3}

In this subsection, we prove the following theorem.

\begin{Thm}\label{Conjugacy classes and countable dense normal subgroups}
	Let $ K $ be a number field, $ S $ a finite set of primes of $ K $ and $ A $ a complete Noetherian local ring with finite residue field of characteristic $ p $. Let $ \rho:G_{K,S}^{\text{tame}}\to {\rm GL}_{n}(A) $ be a continuous homomorphism. Assume that at least one of the following conditions holds.
	\begin{enumerate}
		\item $ \text{Im}(\rho) $ has an open conjugacy class.
		\item $\text{Im}(\rho)$ contains a dense normal subgroup which has only countably many conjugacy classes.
		\item There exists a countable subset $D$ of $\text{Im}(\rho)$ such that normal subgroup $N$ of $\text{Im}(\rho)$ generated by $D$ is dense in $\text{Im}(\rho)$ and the conjugacy class $g^{\text{Im}(\rho)}$ of $g$ in $\text{Im}(\rho)$ is finite (as a set) for any $g\in D$.
		\item There are only finitely many elements in each conjugacy class $ \rho(\text{Frob}_{v}) $ in $ \text{Im}(\rho) $ for any $ v\notin S $ where $\text{Frob}_{v}$ is the Frobenius conjugacy class of $v$ in $G_{K,S}^{\text{tame}}$.
	\end{enumerate}
	Then $ \rho $ has finite image. 
\end{Thm}

To prove Theorem \ref{Conjugacy classes and countable dense normal subgroups}, we need the following lemma.

\begin{Lem}\label{finiteorderconj}
	Let $G$ be a profinite group. If $g^G$ is an open conjugacy class in $G$ for some $g\in G$, then $g$ has finite order.
\end{Lem}
\begin{proof}
	The proof is the same as the step $ 2.3 $ in the proof of \cite[Theorem 1.1]{MR3993799}. By assumption, $ g^{G} $ contains $ gH $ for some open subgroup $ H $ of $ G $. Then we have
	\[|C_{G/N}(gN)|=\dfrac{|G/N|}{|g^{G}N/N|}\leq \dfrac{|G/N|}{|HN/N|}\leq [G:H]  \]
	for every normal open subgroup $ N $ of $ G $ where $ C_{G/N}(gN) $ denotes the centralizer of $ gN $ in $ G/N $. Since the sizes of $C_{G/N}(gN)$ are uniformly bounded where $N$ ranges over all the open normal subgroups of $G$, we see that $g$ has finite order.
\end{proof}

\begin{Prop}\label{open conjugacy}
	Let $F$ be a non-Archimedean local field. If $H$ is a closed subgroup of ${\rm GL}_n(F)$ such that it has an open conjugacy class, then $H$ is virtually solvable.
\end{Prop}
\begin{proof}
	Assume that $H$ is not virtually solvable. By the topological Tits alternative in \cite[Theorem 1.3]{MR2373146}, we see that $ \text{Im}(\rho) $ contains a dense free subgroup $ D $. But this contradicts Lemma \ref{finiteorderconj}. We are done.
\end{proof}

\begin{proof}[Proof of Theorem \ref{Conjugacy classes and countable dense normal subgroups}]
	By Corollary \ref{imageisfinitelygenerated}, we know that the group $\text{Im}(\rho) $ is topologically finitely generated. Suppose the condition (i) holds. Then $\text{Im}(\rho) $ contains an open conjugacy class. By Theorem \ref{tamelyvirtually solvable case}, we see that any virtually solvable continuous quotient of $\text{Im}(\rho) $ is finite. To show that $\text{Im}(\rho) $ is finite, it's enough to show that any continuous representation $f:\text{Im}(\rho)\to {\rm GL}_n(F) $ of $\text{Im}(\rho) $ has finite image by Theorem \ref{doctoral} where $F$ is a non-Archimedean local field. But such $f$ always has finite image by Proposition \ref{open conjugacy}. Therefore, $\text{Im}(\rho) $ must be finite.
	
	Assume the condition (ii) (resp. (iii)) holds. Suppose that the group $ \text{Im}(\rho) $ is infinite. By Lemma \ref{pro}, we see that $ \text{Im}(\rho) $ is a topologically finitely generated infinite profinite group which is a virtually pro-$ p $ group. By Lemma \ref{justinfinitequotient} and Theorem \ref{tamelyvirtually solvable case}, we obtain that $\text{Im}(\rho)$ admits a topologically finitely generated just-infinite quotient which is not virtually abelian. But this contradicts Theorem \ref{densenormalsubgroup}. We conclude that $\rho$ has finite image. 
	
	Assume the condition (iv) holds. Let $N $ be the normal subgroup of $\text{Im}(\rho)$ generated by the conjugacy classes $\rho(\text{Frob}_{v})$ where $v\notin S$. By Chebotarev density theorem, $N$ is dense in $\text{Im}(\rho)$. Note that there are only countably many Frobenius conjugacy classes in $G_{K,S}^{\text{tame}}$. By our assumption and the previous assertion, we see that $\rho$ has finite image.
	This completes the proof of our theorem.  
\end{proof}

\subsection{Infinite tamely ramified Galois extensions of a number field}\label{section2.4}

In \cite{MR714470}, Ihara proved the following theorem. See also Tsfasman-Vladut \cite{MR1944510}, and \cite[Proposition F]{MR2726590}.

\begin{Thm}\label{ihara}
	Let $ K $ be a number field and $ L/K $ is an infinite Galois extension of $ K $ which is tamely ramified outside a finite set $ S $ of primes of $ K $. Let $ S_{f} $ denote the set of primes of $ K $ that split completely in $ L $. Then there is a finite constant $ c(K) $ depending only on $ K $ such that
	\[ \sum_{\mathfrak{p}\in S_{f}}\frac{\log N(\mathfrak{p})}{N(\mathfrak{p})-1}\leq c(K). \]
\end{Thm}

As mentioned in Remark \ref{remark}, it's conjectured that the Galois group $ G_{K,S}^{\text{tame}} $ is topologically finitely generated but not too many results are known yet. However, we have the following result which is the main theorem in this section. 

\begin{Thm}\label{topologically generated by a finite number of Frobenius conjugacy classes}
	Let $ K $ be a number field of signature $ (r_{1},r_{2}) $, i.e. $ K $ has $ r_{1} $ real, and $ r_{2} $ pairs of complex conjugate, embeddings, and $ S $ a finite set of primes of $ K $. Then there is a finite set $T$ of primes in $K$ such that $S\cap T=\emptyset$ and the group $G_{K,S}^{\text{tame}}$ is topologically generated by the set $\bigcup_{v\in T} \text{Frob}_{v}$ where $\text{Frob}_{v}$ is the Frobenius conjugacy class of $v$ in $G_{K,S}^{\text{tame}}$. Moreover, we have the following:
	\begin{enumerate}
		\item If the profinite group $G_{K,S}^{\text{tame}}$ admits a topologically finitely generated just-infinite quotient, then there exists a $v\in T$ such that the Frobenius conjugacy class $\text{Frob}_{v}$ contains uncountably many elements.
		\item If $|S|-|S\cap S_{\mathbb{C}}|-|S\cap S_{2}(K)|\geq 2+r_{1}+r_{2}+2\sqrt{r_{1}+r_{2}}$ where $S_{\mathbb{C}}$ is the set of complex primes of $K$ and $S_{2}(K)$ is the set of primes of $K$ above $2$, then the profinite group $G_{K,S}^{\text{tame}}$ admits a topologically finitely generated just-infinite quotient. In particular, there exists a $v\in T$ such that the Frobenius conjugacy class $\text{Frob}_{v}$ contains uncountably many elements.
		\item If there are only finitely many elements in each conjugacy class $ \text{Frob}_{v} $ for any $v\in T$, then any continuous homomorphism $\rho:G_{K,S}^{\text{tame}}\to {\rm GL}_{n}(A)$ has finite image where $ A $ is a complete Noetherian local ring with finite residue field of characteristic $ p $. That is, Conjecture \ref{tameramifieddis} holds for such $G_{K,S}^{\text{tame}}$.
	\end{enumerate}
\end{Thm}
\begin{Rem}
	We ask a natural question as follows: If the maximal pro-$p$ quotient $G_{K,S}^{\text{tame}}(p)$ of the group $G_{K,S}^{\text{tame}}$ is finite for every prime $p$, then is it true that $ G_{K,S}^{\text{tame}}$ is finite? Assume the answer is yes. Then the group $G_{K,S}^{\text{tame}}$ always admits a topologically finitely generated just-infinite quotient if it is infinite. 
\end{Rem}

\begin{proof}[Proof of Theorem \ref{topologically generated by a finite number of Frobenius conjugacy classes}]
	The proof of the first part of our theorem is similar to the proof of \cite[Corollary 10.11.15]{MR2392026}. If the group $G_{K,S}^{\text{tame}}$ is finite, then we are done by Chebotarev density theorem. So we assume that $G_{K,S}^{\text{tame}}$ is infinite.
	Since $S$ is a finite set and the sum of $ (\log N(\mathfrak{p}))/(N(\mathfrak{p}-1)) $ over all non-archimedean primes of $ \mathfrak{p} $ of $ K $ is divergent, we can find primes $ \mathfrak{p}_{1},\cdots,\mathfrak{p}_{r} $ such that $ \mathfrak{p}_{i}\notin S $, and 
	\[ \sum_{i=1}^{r}\frac{\log N(\mathfrak{p}_{i})}{N(\mathfrak{p}_{i})-1}>c(K),  \]
	where $ c(K) $ was defined in Theorem \ref{ihara}. Let $ M$ be the maximal tamely ramified extension of $ K $ which is unramified outside $ S $ such that all $ \mathfrak{p}_{i},~i=1,\cdots,r $, completely split. Then the extension $ M/K $ is finite by Theorem \ref{ihara}. It follows that the normal subgroup topologically generated by the decomposition groups $ G_{\mathfrak{p}_{i}}=\left\langle \text{Frob}_{\mathfrak{p}_{i}} \right\rangle  $, $ i=1,\cdots,r $, has finite index in $ G_{K,S}^{\text{tame}} $. By Chebotarev density theorem, we can find some Frobenius conjugacy classes $ \text{Frob}_{\mathfrak{p}_{r+1}},\cdots,\text{Frob}_{\mathfrak{p}_{r+m}} $ in $G_{K,S}^{\text{tame}}$ such that the set of their images in the finite quotient $\text{Gal}(M/K)$ generates the entire group. We conclude that $ G_{K,S}^{\text{tame}} $ is topologically generated by the set $\bigcup_{i=1}^{r+m}\text{Frob}_{\mathfrak{p}_{i}}$. This proves the first part of our theorem.
	
	For the second part of the theorem, if the group $G_{K,S}^{\text{tame}}$ admits a topologically finitely generated just-infinite quotient $H$, then $H$ is not virtually abelian by Theorem \ref{tamelyvirtually solvable case}. Suppose that there are only countably many elements in each conjugacy class $ \text{Frob}_{v} $ for any $v\in T$. Then the normal subgroup $N$ of $G_{K,S}^{\text{tame}}$ generated by the set $\bigcup_{v\in T} \text{Frob}_{v}$ is a countable dense subgroup of $G_{K,S}^{\text{tame}}$ by Corollary \ref{countable}. This contradicts Theorem \ref{densenormalsubgroup}. We conclude that there exists a $v\in T$ such that the Frobenius conjugacy class $\text{Frob}_{v}$ contains uncountably many elements.
	
	For the third part of the theorem, note that the maximal pro-$2$ quotient $G_{K,S}^{\text{tame}}(2)$ of the group $G_{K,S}^{\text{tame}}$ is infinite by our assumption and \cite[Theorem 10.10.1]{MR2392026}. By Theorem \ref{global structure}, we know that the pro-$2$ group $G_{K,S}^{\text{tame}}(2)$ is topologically finitely generated. It follows from Lemma \ref{justinfinitequotient} that $G_{K,S}^{\text{tame}}(2)$ admits a topologically finitely generated just-infinite quotient, and so is $G_{K,S}^{\text{tame}}$.
	
	For the last part of the theorem, let $\rho:G_{K,S}^{\text{tame}}\to {\rm GL}_{n}(A)$ be a continuous homomorphism.
	If there are only finitely many elements in each conjugacy class $ \text{Frob}_{v} $ for any $v\in T$, then the normal subgroup $N$ of $G_{K,S}^{\text{tame}}$ generated by the set $\bigcup_{v\in T} \text{Frob}_{v}$ is a countable dense subgroup of $G_{K,S}^{\text{tame}}$ by Corollary \ref{countable}. It follows that the group $\text{Im}(\rho)$ contains a countable normal dense subgroup. By Theorem \ref{Conjugacy classes and countable dense normal subgroups}, $\text{Im}(\rho)$ is finite. This completes the proof of our theorem.
\end{proof}

\section{The fundamental case of Conjecture \ref{BUFM}}\label{section3.1}

In this section, we study the fundamental case of Conjecture \ref{BUFM}. 
We start with the case of $ \text{Im}(\rho) $ being virtually solvable.

\begin{Thm}\label{virtually solvable case}
	Let $ K $ be a number field and $ S $ a finite set of primes of $ K $ with $ S\cap S_{p}=\emptyset $. Let $ A $ be a complete Noetherian local ring with finite residue field of characteristic $ p $.
	If $ \rho:G_{K,S}\to {\rm GL}_{n}(A) $ is a continuous homomorphism such that $ \text{Im}(\rho) $ is virtually solvable, then $ \rho $ has finite image. In particular, Conjecture \ref{BUFM} holds for $ n=1 $.
\end{Thm}
\begin{proof}
	By Lemma \ref{pro}, we see that $\text{Im}(\rho)$ contains an open pro-$ p $ group subgroup. Note that any subgroup of virtually solvable group is also virtually solvable. After replacing $ K $ by a finite extension, we can assume that $ \text{Im}(\rho) $ is a virtually solvable pro-$ p $ group. Since $ S\cap S_{p}=\emptyset $, we see that $ \rho $ factors through $ G_{K,S}^{\text{tame}} $ by Lemma \ref{redundant}. By Corollary \ref{imageisfinitelygenerated}, we see that $\text{Im}(\rho)$ is topologically finitely generated. Then our claim follows from Theorem \ref{tamelyvirtually solvable case}.
\end{proof}

The following proposition provides a useful criterion for verifying Conjecture \ref{BUFM}.

\begin{Prop}\label{equi}
	Let $ K $ be a number field, $ S $ a finite set of primes of $ K $ with $ S\cap S_{p}=\emptyset $ and $ F $ a non-Archimedean local field of characteristic zero or of characteristic $ p $. Let $ \rho:G_{K,S}\to {\rm GL}_{n}(F) $ be a continous representation of $ G_{K,S} $.  Then the following are equivalent:
	\begin{enumerate}
		\item $ \text{Im}(\rho) $ is finite.
		\item $ \text{Im}(\rho) $ is virtually solvable.
		\item $ \text{Im}(\rho) $ contains a dense subset $ D $ such that the roots of the characteristic polynomial $ P_{x}(t)=\det(1-t\cdot x) $ of $ x $ are roots of unity for any $x\in D$. When $ F $ has characteristic $ p $, it is equivalent to state that $ \text{Im}(\rho) $ contains a dense subset $ D $ such that the characteristic polynomial of $ x $ has coefficients in a finite subfield of $ F $ for any $ x\in D $.
		\item $ \text{Im}(\rho) $ contains an open unipotent subgroup.
		\item The semisimplification $ \rho^{\text{ss}} $ of $ \rho $ has finite image.
		\item The eigenvalues of all $ \rho(\text{Frob}_{v}),v\notin S, $ are roots of unity.
	\end{enumerate}
\end{Prop}

To prove Proposition \ref{equi}, we need the following fact.

\begin{Convention}
	If $F$ is a a non-Archimedean local field, then we put $F_{0}=\mathbb{Q}_{p}$ when $\text{char}(F)=0$ and $F_{0}=\mathbb{F}_{p}((T))$ when $\text{char}(F)=p$.
\end{Convention}

\begin{Lem}\label{bounded}
	Let $n,d$ be two positive integers. Let $r$ be $0$ or $p$. Then there is an integer $ m=m(n,d,r) $ depending on $n,d$ and $r$ such that the order of every torsion element in $ {\rm GL}_{n}(\mathcal{O}_{F}) $ divides $ m $ for any non-Archimedean local field $F$ of characteristic $r$ with degree (over $F_{0}$) less than or equal to $d$.
\end{Lem}
\begin{proof}
	At first, we consider the case of $r=0$. Using Krasner's lemma, we know that there are only finitely many degree $k$ extensions of $\mathbb{Q}_{p}$ (in a fixed algebraic closure of $\mathbb{Q}_{p}$) for any integer $k$, cf. \cite[Section 1.6, Chapter 3]{MR1760253}. Thus, it suffices to show that the order of torsion elements in ${\rm GL}_{n}(\mathcal{O}_{F})$ is bounded by a constant if $F$ is a finite extension of $\mathbb{Q}_{p}$. Indeed, let $F'$ be a finite extension of $F$ which is the compositum of all degree $n$ extensions of $F$. Then any torsion element in ${\rm GL}_{n}(\mathcal{O}_{F})$ is diagonalizable over $F'$. Since there are only finitely many roots of unity in $F'$, we see that the order of torsion elements in ${\rm GL}_{n}(\mathcal{O}_{F})$ is bounded by a constant. 
	
	Secondly, we consider the case of $r=p$. Let $F$ be a finite extension of $\mathbb{F}_{p}((T))$ of degree $d$. Then $\mathcal{O}_{F}=\mathbb{F}_{q}[[T]]$ where $q=p^{d}$. If $g\in {\rm GL}_{n}(\mathcal{O}_{F})$ is an element of finite order, then the order of $x^{p^{n}}$ is prime to $p$. By Lemma \ref{pro}, the group ${\rm GL}_{n}^{1}(\mathcal{O}_{F})$ is a pro-$p$ group, and hence the order of $x^{p^{n}}$ is the same as the order of $\eta(x^{p^{n}})$ in ${\rm GL}_{n}(\mathbb{F}_{q})$ where $\eta:{\rm GL}_{n}(\mathcal{O}_{F})\to {\rm GL}_{n}(\mathbb{F}_{q})$ is the natural surjection. By \cite[Corollary 2]{MR2181417}, the order of $\eta(x^{p^{n}}) $ is less than or equal to $q^{n}-1$, and so is $x^{p^{n}}$. It follows that the order of $x$ is less than or equal to $p^{n}(q^{n}-1)=p^{n}(p^{nd}-1)$. Since $p^{n}(p^{nd}-1)\geq p^{n}(p^{nk}-1)$ for any positive integer $k\leq d$, the order of torsion elements in ${\rm GL}_{n}(\mathcal{O}_{F})$ is bounded by $p^{n}(p^{nd}-1)$ when $F$ is a finite extension of $\mathbb{F}_{p}[[T]]$ of degree less than or equal to $d$. We are done.
\end{proof}

\begin{proof}[Proof of Proposition \ref{equi}]
	Firstly, when $ F $ has characteristic $ p $, one has $ F=\mathbb{F}((T)) $ for some finite field $ \mathbb{F} $ of characteristic $ p $. Then an element in $ F $ is algebraic (over the prime subfield of $ F $) if and only if it lies in the finite subfield $ \mathbb{F} $ of $ F $. The equivalence in the statement $ 3 $ follows. Now, We prove our claim as follows:
	\begin{itemize}
		\item (ii) $ \implies $ (i): It is Theorem \ref{virtually solvable case}.
		\item (iv) $\implies $ (ii): It follows from the fact that any unipotent group is solvable.
		\item (iii) $ \implies $ (iv): By Lemma \ref{bounded}, we know that these roots of unity must have bounded order uniformly as $x $ varies over $ D $. 
		Since the degree of the characteristic polynomial $ P_{x}(t) $ for some $ x\in D $ is fixed and independent of $x\in D  $, there are only finitely many possibilities for the characteristic polynomial $ P_{x}(t) $ as $ x $ varies over $ D $. Since $ D $ is dense in $ \text{Im}(\rho) $, we see that there are also only finitely many possibilities for the characteristic polynomial of any element in $ \text{Im}(\rho) $, and the characteristic polynomial map $ G_{K,S}\to F[t] $ sending any element of $ G_{K,S} $ to the characteristic polynomial of the image of the element is locally constant. It follows that the set $ \{g\in G_{K,S}~|~\det(1-t\rho(g))=(1-t)^{n}\} $ is open in $ G_{K,S} $, and hence contains an open subgroup $ U\subset G_{K,S} $ such that the image $ \rho(U) $ of $ U $ under the representation $ \rho $ is unipotent.
		\item (i) $\implies $ (iii): If $ \text{Im}(\rho) $ is finite, then there exists a positive integer $ m $ such that $ x $ satisfies the polynomial $ t^{m}-1=0 $ for any $ x\in \text{Im}(\rho) $. Hence the minimal polynomial of $ x $ must divide this polynomial for any $ x\in \text{Im}(\rho) $. Thus every eigenvalue of $x $ must be a root of unity for any $ x\in \text{Im}(\rho)$. Put $ D=\text{Im}(\rho)$.
		\item (v) $\implies $ (iii): It follows from the Brauer–Nesbitt theorem (cf. \cite[Section 30.16]{MR0144979})and the argument in (i) $\implies $ (iii).
		\item (vi) $ \implies $ (iii): It follows from Chebotarev density theorem.
		\item (i) $ \implies $ (v): It is trivial.
		\item (i)$ \implies$ (vi): Using the same argument in (i) $ \implies $ (iii), we see it is true.
	\end{itemize}
\end{proof}

The following corollary is known.
\begin{Cor}\label{dimlessthan3}
	Let $ K $ be a number field and $ S $ a finite set of primes of $ K $ with $ S\cap S_{p}=\emptyset $. Let $ m $ be $ 1 $ or $ 2 $. Then there is no $ p $-adic analytic quotient of $ G_{K,S} $ of dimension $ m $.
\end{Cor}
\begin{proof}
	By Example \ref{dimensionlessthan3}, we know that any $ p $-adic analytic group of dimension $ m $ is an infinite virtually solvable group. Then our claim follows from Proposition \ref{equi}.
\end{proof}

If Conjecture \ref{BUFM} fails, then we have the following:
\begin{Cor}
	Let $ K $ be a number field and $ S $ a finite set of primes of $ K $ such that $S\cap S_{p}=\emptyset $. Let $ F $ be a non-Archimedean local field of characteristic $ 0 $ or $ p $, and $ \rho:G_{K,S}\to {\rm GL}_{n}(F) $ be a continuous homomorphism. Suppose that the image $ \text{Im}(\rho) $ of $ \rho $ is infinite. Then $ \text{Im}(\rho) $ contains a dense free subgroup $ D $. 
\end{Cor}
\begin{proof}
	Since $ \text{Im}(\rho) $ is infinite, we see that $ \text{Im}(\rho) $ is not virtually solvable by Proposition \ref{equi}. By the topological Tits alternative in \cite[Theorem 1.3]{MR2373146}, we obtain that $ \text{Im}(\rho) $ contains a dense free subgroup $ D $. By Corollary \ref{imageisfinitelygenerated}, we see that $ \text{Im}(\rho) $ is topologcially finitely generated. This completes the proof.
\end{proof}

Let $ F $ be a field. Recall that a \emph{quasi-unipotent} matrix $ A $ in $ {\rm GL}_{n}(F) $ is a matrix such that some power of $ A $ is unipotent matrix. Note that all the eigenvalues of a quasi-unipotent matrix are roots of unity.
\begin{Cor}
	Let $ K $ be a number field such that it has no non-trivial unramified extensions, $ S $ a finite set of primes of $ K $ with $ S\cap S_{p}=\emptyset $ and $ F $ a non-Archimedean local field of characteristic zero or of characteristic $ p $. If $ \rho:G_{K,S}\to {\rm GL}_{n}(F) $ is a continuous representation such that the quasi-unipotent elements in the image of $ \rho $ form a subgroup, then $ \rho $ has finite image.
\end{Cor}
\begin{proof}
	Since $K$ has no non-trivial unramified extensions, the group $ G_{K,S} $ is topologically generated by inertia groups $ I_{\mathfrak{p}} $ at primes $ \mathfrak{p} $ at $ S $ and all elements in $ \rho(I_{\mathfrak{p}}) $ are quasi-unipotent by Theorem \ref{local}. By our assumption, this implies that $ \text{Im}(\rho) $ contains a dense subgroup consisting of quasi-unipotent elements. Thus, our claim follows from Proposition \ref{equi}.
\end{proof}

\section{Reductions of Conjecture \ref{BUFM}}\label{section3.2}

\subsection{Statements}
In this section, we study the reductions of Conjecture \ref{BUFM}, and we will explain how to divide Conjecture \ref{BUFM} into several parts. For example, we show that, to verify Conjecture \ref{BUFM}, it suffices to consider the continuous representations of global Galois groups over non-Archimedean local fields. 
Our first main result of this thesis is the following:

\begin{Thm}\label{reduction}
	Let $ K $ be a number field and $ S $ a finite set of primes of $ K $ with $ S\cap S_{p}=\emptyset $. Let $ A $ be a complete Noetherian local ring with finite residue field of characteristic $ p $. Let $ n\geq 1$ be a fixed positive integer. Assume the following:
	\begin{enumerate}
		\item When $\text{char}(A)=0$, suppose that for any non-Archimedean local field $F$ with residue field of characteristic $p$, any continuous representation $\rho:G_{K,S}\to {\rm GL}_{n}(F)$ has finite image.
		\item When $\text{char}(A)>0$, suppose that for any non-Archimedean local field $F$ of characteristic $p$, any continuous representation $\rho:G_{K,S}\to {\rm GL}_{n}(F)$ has finite image.
	\end{enumerate}
	Then any continuous homomorphism $ \rho:G_{K,S}\to {\rm GL}_{n}(A) $ has finite image. 
\end{Thm}
\begin{Rem}
	\begin{enumerate}
		\item By definition, the characteristic $\text{char}(A)$ of $A$ is $0$ or $p^r$ for some positive integer $r$.
		\item When the reduction $ \overline{\rho} $ of $\rho $ is absolutely irreducible, Allen and Calegari proved a similar result in \cite[Proposition 10]{MR3294389}. 
		\item The theorem also holds for the profinite group $G_{K,S}^{\text{tame}}$ where $S$ is any finite set of primes of $K$.
	\end{enumerate}
\end{Rem}

A special case of Conjecture \ref{UFM} is the following:

\begin{conj}\label{theimageofineritagroupisfinite}
	Let $ K $ be a number field and $ S $ a finite set of primes of $ K $ such that $ S\cap S_{p}=\emptyset $. Suppose that $ \rho:G_{K,S}\to {\rm GL}_{n}(\mathbb{Q}_{p}) $ is a continuous representation. Then the image $ \rho(I_{\mathfrak{q}}) $ of $ I_{\mathfrak{q}} $ under $ \rho $ is finite for any $ \mathfrak{q}\in S $ where $I_{\mathfrak{q}}  $ is the inertia subgroup of $ G_{K,S} $ at $ \mathfrak{q} $ which is defined up to conjugacy.
\end{conj}

An $ \mathbb{F}_p[[T]] $-adic version of Conjecture \ref{UFM} is the following:
\begin{conj}\label{UFMforcharp}
	Let $ K $ be a number field and $ S $ a finite set of primes of $ K $ such that $ S\cap S_{p}=\emptyset $. Let $ F $ be a non-Archimedean local field of characteristic $ p $. Then any continuous representation $ \rho:G_{K,S}\to {\rm GL}_{n}(F) $ has finite image.
\end{conj}

Also, we state an $ \mathbb{F}_p[[T]] $-adic version of Conjecture \ref{WUFM} as follows:
\begin{conj}\label{WUFMforcharp}
	Let $ K $ be a number field and let $ F $ be a non-Archimedean local field of characteristic $ p $. Then any continuous representation $ \rho:G_{K,\emptyset}(p)\to {\rm GL}_{n}(F) $ of the pro-$ p $ group $ G_{K,\emptyset}(p) $ has finite image.
\end{conj}

Applying Theorem \ref{reduction}, we prove that
\begin{Thm}\label{reductionofBUFM}
	\begin{enumerate}
		\item 	Conjecture \ref{UFMforcharp} is equivalent to Conjecture \ref{WUFMforcharp}.
		\item The following are equivalent:
		\begin{enumerate}
			\item Conjecture \ref{BUFM} holds.
			\item Both Conjecture \ref{UFM} and \ref{WUFMforcharp} hold.
			\item Conjecture \ref{WUFM}, Conjecture \ref{theimageofineritagroupisfinite} and Conjecture \ref{WUFMforcharp} are all true. 
		\end{enumerate}
	\end{enumerate}
\end{Thm}
\begin{Rem}
	The equivalence between (ii)(a) and (ii)(b) of the above theorem was stated in \cite[Section 2]{MR1932455} without giving a proof.
\end{Rem}

\subsection{Proof of Theorem \ref{reduction} and Theorem \ref{reductionofBUFM}}

\begin{Convention}
	Let $ A $ be a complete Noetherian local ring with finite residue field of characteristic $ p $. We say that $A$ is a \emph{pro-$p$ domain} if it is an integral domain. For a pro-$p$ domain $(A,\mathfrak{m}_{A})$, we put $A_{0}:=\mathbb{Z}_{p}$ when $\text{char}(A)=0$ and $A_{0}:=\mathbb{F}_{p}[[T]]$ when $\text{char}(A)=p$. The field of fractions of the integral domain $ A_{0}$ is denoted by $\text{Frac}(A_{0})$.
	For any positive integer $d$, we write $\mathfrak{m}_{A}^{(d)}$ for the set of $d$-tuples of elements of $\mathfrak{m}_{A}$.
\end{Convention}

To prove Theorem \ref{reduction}, we need the following lemmas.

\begin{Lem}[Cohen's structure theorem]\label{conhenstructuretheorem}
	Let $A$ be a pro-$p$ domain. Then $A$ is a finite integral extension of $A_{0}[[T_{1},\cdots,T_{r-1}]]$, where $r$ is the Krull dimension of $A$. Moreover, if $r=1$, then $A$ is a finitely generated free module over $A_{0}$, i.e. $A=\mathcal{O}_{F}$ for some non-Archimedean local field $F$.
\end{Lem}
\begin{proof}
	It is \cite[Theorem 6.42 and Corollary 6.43]{MR1720368}.
\end{proof}

\begin{Lem}\label{evaluation}
	Let $A$ be a pro-$p$ domain. Suppose that $F=\sum  a_{m_{1},\cdots,m_{d}} T_{1}^{m_{1}}\cdots T_{d}^{m_{d}} \in A[[T_{1},\cdots,T_{d}]]$. Then 
	\begin{enumerate}
		\item $F$ can be evaluated at $t=(t_{1},\cdots,t_{d})$ for every $t=(t_{1},\cdots,t_{d})\in \mathfrak{m}_{A}^{(d)}$, i.e. the series $\sum  a_{m_{1},\cdots,m_{d}} t_{1}^{m_{1}}\cdots t_{d}^{m_{d}}$ converges in $A$.
		\item If $F(t)=0$ for any $t=(t_{1},\cdots,t_{d})\in \mathfrak{m}_{A}^{(d)} $, then $F=0$.
	\end{enumerate}
\end{Lem}
\begin{proof}
	This first claim is \cite[Lemma 6.44]{MR1720368}, and the second claim is \cite[Lemma 9]{MR2218704}.
\end{proof}

\begin{Lem}\label{goingup}
	Let $A\subset B$ be commutative rings with $B$ integral over $A$. 
	\begin{enumerate}
		\item Let $\mathfrak{p}$ be a prime ideal of $A$. Then there exists a prime ideal $\mathfrak{q}$ of $B$ such that $\mathfrak{p}=\mathfrak{q}\cap A$.
		\item Suppose also that $B$ is an integral domain. If $I$ is an ideal of $B$ such that $I\cap A=\{0\}$, then $I=\{0\}$. 
	\end{enumerate}
\end{Lem}
\begin{proof}
	The first claim is the well-known going-up theorem, cf. \cite[Theorem 9.3]{MR0879273}. For the second claim, we need to show that if $I$ is a non-zero ideal of $B$, then $I\cap A\neq \{0\}$. If $I\subset A$, then the claim is trivial and hence we can assume that $I$ is not contained in $A$. We fix an element $x\neq 0$ in $I$ such that $x\notin A$. Then we have an equation $x^{n}+a_{n-1}x^{n-1}+\cdots +a_{0}=0$ of $x$ over $A$ where $a_{i}\in A$ and $n\geq 2$. Since $B$ is an integral domain, we can find some $a_{i}\neq 0$. Let $r\geq 0$ be the minimal integer such that $a_{r}\neq 0$. Then we have $r\leq n-1$ and $x^{r}(x^{n-r}+\cdots +a_{r})=0$. Since $B$ is an integral domain and $x\neq 0$, we see that $x^{n-r}+\cdots +a_{r}=0$. This implies that $a_{r}\in (xB)\cap A\subset I\cap A$, i.e. $I\cap A\neq \{0\}$. We are done.
\end{proof}

The following proposition shows that a pro-$p$ domain is determined by its values in local fields.

\begin{Prop}\label{propdomainisdeterminendby}
	Let $A$ be a pro-$p$ domain. Then there exists a positive integer $d$ such that the intersection of kernels of all morphisms $ A\to \mathcal{O}_{F} $ is the zero ideal where $ F $ runs over all finite extensions of $\text{Frac}(A_{0})$ with degree less than or equal to $d$.
\end{Prop}
\begin{proof}
	By Lemma \ref{conhenstructuretheorem}, we see that $A$ is a finite integral extension of $A_{0}[[T_{1},\cdots,T_{r-1}]]$, where $r$ is the Krull dimension of $A$. If $r=1$, then the claim is trivial, and we can assume that $r\geq 2$.
	For any $a=(a_{1},\cdots,a_{r-1})\in \mathfrak{m}_{A_{0}}^{(r-1)}$, by Lemma \ref{evaluation} we have the following well-defined continuous evaluation map:
	\[ s_a: A_{0}[[T_1,..,T_{r-1}]]\to A_{0},~F\mapsto F(a).\]
	Moreover, we have $\bigcap_{a\in \mathfrak{m}_{A_{0}}^{(r-1)}}\ker(s_{a})=\{0\}$.
	
	Set $\mathfrak{p}_{a}:=\ker(s_{a})$. Since $A$ is a finite integral extension of $A_{0}[[T_{1},\cdots,T_{r-1}]]$, there is a prime ideal $\mathfrak{q}_{a}$ of $A$ such that $\mathfrak{q}_{a}\cap A_{0}[[T_{1},\cdots,T_{r-1}]]=\mathfrak{p}_{a} $ by Lemma \ref{goingup}. Set $B_{a}:=A/\mathfrak{q}_{a}$, and define $\tilde{s}_{a}:A\twoheadrightarrow B_{a}$ to be the natural projection. Then we have the obvious injection $A_{0}\to B_{a}$, and we have the following commutative diagram:
	\[ \xymatrix{
		A_{0}[[T_1,..,T_{r-1}]] \ar@{^{(}->}[d] \ar@{->>}[r]^{\qquad s_a} &A_{0}\ar@{^{(}->}[d]\\
		A  \ar@{->>}[r]^{\tilde{s}_a} &B_{a}} \]
	That is, the ring homomorphism $s_a$ extends to a continuous ring epimorphism $\tilde{s}_a: A \to B_{a}$ where $B_{a}$ is a pro-$p$ domain of Krull dimension $1$. Note that we have
	\[ A_{0}[[T_1,..,T_{r-1}]]\cap\left( \bigcap_{a\in \mathfrak{m}_{A_{0}}^{(r-1)}}\mathfrak{q}_{a}\right) =\bigcap_{a\in \mathfrak{m}_{A_{0}}^{(r-1)}}\mathfrak{p}_{a}=\{0\} .\]
	Since $A$ is a finite integral extension of $A_{0}[[T_{1},\cdots,T_{r-1}]]$, by Lemma \ref{goingup} we have
	\[ \bigcap_{a\in \mathfrak{m}_{A_{0}}^{(r-1)}}\ker(\tilde{s}_{a})=\bigcap_{a\in \mathfrak{m}_{A_{0}}^{(r-1)}}\mathfrak{q}_{a}=\{0\}.\]
	
	If $\alpha_{1},\cdots,\alpha_{d}$ is a system of generators of the finite $A_{0}[[T_{1},\cdots,T_{r-1}]]$-module $A$, then $\tilde{s}_{a}(\alpha_{1}),\cdots,\tilde{s}_{a}(\alpha_{d})$ is a system of generators of the finite free $A_{0}$-module $B_{a}$. It follows that each $B_{a}$ is the ring of integers $\mathcal{O}_{F}$ for some finite extension $F$ of $\text{Frac}(A_{0})$ with degree less than or equal to $d$ by Lemma \ref{conhenstructuretheorem}. The theorem is proved.
\end{proof}

Note that if $R_{1},R_{2}$ are commutative rings and there exists a ring homomorphism $R_{1}\to R_{2}$, then the characteristic of $R_{2}$ divides the characteristic of $R_{1}$.

\begin{Prop}\label{true}
	Let $ G $ be a topologically finitely generated profinite group and let $ A $ be a reduced complete Noetherian local ring with finite residue field of characteristic $ p $. Let $ n\geq 1$ be a fixed positive integer. Assume the following:
	\begin{enumerate}
		\item When $\text{char}(A)=0$, suppose that for any non-Archimedean local field $F$ with residue field of characteristic $p$, any continuous representation $\rho:G\to {\rm GL}_{n}(F)$ has finite image.
		\item When $\text{char}(A)>0$, suppose that for any non-Archimedean local field $F$ of characteristic $p$, any continuous representation $\rho:G\to {\rm GL}_{n}(F)$ has finite image.
	\end{enumerate}
	Then any continuous homomorphism $ \rho:G\to {\rm GL}_{n}(A) $ has finite image. 
\end{Prop}
\begin{proof}	
	We prove our proposition as follows:
	\begin{itemize}
		\item Step 1: Assume that $ A $ is an integral domain. Let $ \rho:G\to {\rm GL}_{n}(A) $ be a continuous homomorphism. 
		Since $G$ is topologically finitely generated, so is $\text{Im}(\rho)$. We need to show that $\text{Im}(\rho)$ is finite. By \cite[Theorem 1]{MR1194787}, we know that every topologically finitely generated profinite torsion group is finite. Thus, it suffices to show that each element $g$ in the group $\text{Im}(\rho)$ has finite order. So let $g\in \text{Im}(\rho)$. We need to show that $g$ has finite order.
		
		By Proposition \ref{propdomainisdeterminendby}, there exists a positive integer $d$ such that the intersection of kernels of all morphisms $ A\to \mathcal{O}_{F} $ is the zero ideal where $ F $ runs over all finite extensions of $\text{Frac}(A_{0})$ with degree less than or equal to $d$. By Lemma \ref{bounded}, there is an integer $ N $ such that the order of every torsion element in these $ {\rm GL}_{n}(\mathcal{O}_{F}) $ divides $N $. Since the image of $g$ in these ${\rm GL}_{n}(\mathcal{O}_{F})$ are elements of finite order by our assumption, we see that 
		\[ g^{N}\equiv 1 ~\text{mod}~ \ker(A\to \mathcal{O}_{F}), \]
		for all such $F$, and hence we have $g^{N}=1$. Therefore, the group $ \text{Im}(\rho) $ is finite.
		
		\item Step $ 2 $: In general, let $ A $ be a reduced complete Noetherian local ring, and $ \rho:G\to {\rm GL}_{n}(A) $ be a continuous homomorphism. Consider the natural ring homomorphism
		\[ \varphi: A\to \prod_{\mathfrak{p}\in \text{Minspec}(A)}A/\mathfrak{p}_{i}, \]
		where $ A\in \text{Minspec}(A) $ runs over all minimal prime ideals of $ A $. Since $ A $ is Noetherian and reduced, there are only finitely many minimal prime ideals of $ A $, and the ring homomorphism $ \varphi $ is injective. In particular, the ring homomorphism $ \varphi $ induces the following group homomorphism
		\[ {\rm GL}_{n}(A) \hookrightarrow {\rm GL}_{n}\left(\prod_{\mathfrak{p}\in \text{Minspec}(A)}A/\mathfrak{p}_{i} \right)=\prod_{\mathfrak{p}\in \text{Minspec}(A)}{\rm GL}_{n}(A/\mathfrak{p}_{i}) .\]
		Since each $ A/\mathfrak{p}_{i} $ is a domain, we see that $ \text{Im}(\rho) $ is a finite subgroup of $ \prod_{\mathfrak{p}\in \text{Minspec}(A)}{\rm GL}_{n}(A/\mathfrak{p}_{i}) $ by Step $ 1 $. This completes the proof of our proposition.
	\end{itemize}
\end{proof}

\begin{Lem}\label{nilpotentkernel}
	Let $ A $ be a complete Noetherian local ring with finite residue field of characteristic $ p $, and $A^{\text{red}}$ the quotient of $A$ by the nilradical of $A$. Then the kernel of the group homomorphism $ {\rm GL}_{n}(A)\twoheadrightarrow {\rm GL}_{n}(A^{\text{red}}) $ induced by the ring homomorphism $ A\twoheadrightarrow A^{\text{red}} $ is a nilpotent group.
\end{Lem}
\begin{proof}
	Let $ \text{Nil}(A) $ denote the nilradical of $ A $. Then $ A^{\text{red}}=A/\text{Nil}(A) $, and the kernel of $ {\rm GL}_{n}(A)\twoheadrightarrow {\rm GL}_{n}(A^{\text{red}}) $ is $ I_{n}+M_{n}(\text{Nil}(A)):=\{1+x~|~x\in M_n(\text{Nil}(A))\} $. We need to show that the group $ I_{n}+M_{n}(\text{Nil}(A)) $ is nilpotent.
	
	Since $ A $ is Noetherian, the ideal $ \text{Nil}(A) $ is finitely generated. Let $ a_{1},\cdots,a_{m} $ generate $ \text{Nil}(A) $. For each $ i $, we have $ a_{i}^{r_{i}}=0 $ for some integer $ r_{i} $. By definition, every element $ a\in \text{Nil}(A) $ can be written as 
	\[ c_{1}a_{1}+\cdots +c_{m}a_{m},~c_{i}\in A .\]
	Note that any product of $ (r_{1}+\cdots+ r_{m}) $ elements of $ \text{Nil}(A) $ is equal to $ 0 $. Indeed, for $ b_{j}=c_{1,j}a_{1}+\cdots +c_{m,j}a_{m}\in A,~j=1,\cdots,r_{1}+\cdots+r_{m} $, every term of the product
	\[ \prod_{j=1}^{r_{1}+\cdots+ r_{m}}b_{j}=\prod_{j=1}^{r_{1}+\cdots+ r_{m}}(c_{1,j}a_{1}+\cdots +c_{m,j}a_{m}) \]
	has a factor of the form $ a_{i}^{r_{i}} $ for some $ i $, and hence it is $ 0 $.
	
	It follows that there exists a positive integer $ N $ such that any product of $ N $ elements of $M_{n}(\text{Nil}(A)) $ is equal to $ 0 $. Write $ M $ for $ M_{n}(\text{Nil}(A)) $, and $ M^{i} $ for the set of all sums of product of $ i $ elements of $ M $ where $ i $ is any positive integer. So $ M^{N}=0 $. Put $ U_{i}:=\{1+x~|~x\in M^{i}\} $ for all $ i\geq 1 $. So $ U_{1}=1+M $, and $U_{N}=\{1\}$. Since $M=M_{n}(\text{Nil}(A)) $ is an ideal of the ring $ M_{n}(A) $, $ U_{i} $ is a semigroup with respect to the matrix multiplication, and $ U_{i}\supset U_{i+1} $ for all $ i\geq 1 $.

	If $ 1+x\in U_{r} $ and $ 1+y\in U_{s} $ where $ r,s\geq 1 $, then we have
	\begin{align*}
		[1+x,1+y]=((1+y)(1+x))^{-1}(1+x)(1+y)=(1+y+x+yx)^{-1}(1+x+y+xy).
	\end{align*}
	Put $ u=x+y+xy $ and $ v=y+x+yx $. Using the fact that $M^{N}=0$, we have
	\begin{align*}
		[1+x,1+y]&=(1+v)^{-1}(1+u)  \\
		&=(1-v+v^{2}-\cdots +(-1)^{N-1}v^{N-1})(1+u)\\
		&=1+(1-v+v^{2}-\cdots +(-1)^{N-2}v^{N-2})(u-v)+(-1)^{N-1}v^{N-1}u,
	\end{align*}
	where $ v^{N-1}u=0 $ and $ u-v=xy-yx\in M^{r+s} $. Thus, we see that $ [1+x,1+y] $ lies in $ U_{r+s} $. In other words, $ [U_{r},U_{s}] \subset U_{r+s}$. 
	
	Finally, let $ G=G_{0}\supset G_{1}\supset \cdots G_{i}\supset \cdots $ be the lower central series of $ G:=U_{1} $ where $ G_{i+1}:=[G_{i},G] $. Then we have $ G_{1}=[G,G]=[U_{1},U_{1}]\subset U_{2} $. By induction, we see that
	\[ G_{N-1}=[G_{N-2},G]\subset [U_{N-1},U_{1}]\subset U_{N}=\{1\}. \]
	Thus, $G_{N-1}=\{1\}$ and $ G $ is nilpotent. The lemma is proved.
\end{proof}

\begin{Thm}\label{doctoral}
	Let $ G $ be a topologically finitely generated profinite group such that any virtually solvable continuous quotient of $G$ is finite. Let $ A $ be a complete Noetherian local ring with finite residue field of characteristic $ p $. Let $ n\geq 1$ be a fixed positive integer. Assume the following:
	\begin{enumerate}
		\item When $\text{char}(A)=0$, suppose that for any non-Archimedean local field $F$ with residue field of characteristic $p$, any continuous representation $\rho:G\to {\rm GL}_{n}(F)$ has finite image.
		\item When $\text{char}(A)>0$, suppose that for any non-Archimedean local field $F$ of characteristic $p$, any continuous representation $\rho:G\to {\rm GL}_{n}(F)$ has finite image.
	\end{enumerate}
	Then any continuous homomorphism $ \rho:G\to {\rm GL}_{n}(A) $ has finite image. 
\end{Thm}
\begin{Rem}
	By Proposition \ref{solvablequotient}, any FAb profinite group satisfies the assumption of the theorem.
\end{Rem}

\begin{proof}
	Let $ A $ be a complete Noetherian local ring with finite residue field of characteristic $ p $, and $ \rho:G\to {\rm GL}_{n}(A) $ a continuous homomorphism. We consider the following composition of group homomorphisms:
	\[ \widetilde{\rho}:G\xrightarrow{\rho} {\rm GL}_{n}(A)\xrightarrow{\pi}{\rm GL}_{n}(A^{\text{red}}), \]
	where $ A^{\text{red}} $ is the quotient of $ A $ by the nilradical of $ A $, and the second group homomorphism $ \pi:{\rm GL}_{n}(A)\to {\rm GL}_{n}(A^{\text{red}}) $ is induced by the natural ring homomorphism $ A \twoheadrightarrow A^{\text{red}} $. 
	By our assumption and Proposition \ref{true}, we know that the group $\text{Im}(\widetilde{\rho})$ is finite. It follows that $ \text{Im}(\rho) $ contains an open nilpotent group by Lemma \ref{nilpotentkernel}.  By our assumption, $ \text{Im}(\rho) $ must be finite. This proves the theorem.
\end{proof}

Finally, we prove Theorem \ref{reduction} and Theorem \ref{reductionofBUFM} as follows.

\begin{proof}[Proof of Theorem \ref{reduction}]
	Let $ A $ be a complete Noetherian local ring with finite residue field of characteristic $ p $, and $ \rho:G_{K,S}\to {\rm GL}_{n}(A) $ a continuous homomorphism. By Corollary \ref{imageisfinitelygenerated}, we see that the image of $\rho$ is topologically finitely generated. Then our claim follows from Theorem \ref{doctoral} and Theorem \ref{virtually solvable case}.
\end{proof}

\begin{proof}[Proof of Theorem \ref{reductionofBUFM}]
	We prove our theorem as follows:
	\begin{enumerate}
		\item It is clear that Conjecture \ref{UFMforcharp} implies Conjecture \ref{WUFMforcharp}. 
		Conversely, let $ \rho:G_{K,S}\to {\rm GL}_{n}(F) $ be a continuous homomorphism where $ S=\{\mathfrak{p}_{1},\cdots,\mathfrak{p}_{d}\} $. We may assume that $ \text{Im}(\rho)\subset {\rm GL}_{n}(\mathcal{O}_{F}) $. By Lemma \ref{pro}, we know that $ {\rm GL}_{n}^{1}(\mathcal{O}_{F}) $ is an open pro-$ p $ subgroup of $ {\rm GL}_{n}(\mathcal{O}_{F}) $. After taking a finite extension of $ K $, we can assume that $ \text{Im}(\rho)\subset {\rm GL}_{n}^{1}(\mathcal{O}_{F}) $. By Theorem \ref{local}, we see that $ \rho(I_{\mathfrak{p}_{i}}) $ has finite image where $ I_{\mathfrak{p}_{i}} $ is the inertia subgroup at $ \mathfrak{p}_{i} $ for all $ i $. After taking another finite extension of $ K $, we can assume that $ \rho(I_{\mathfrak{p}_{i}}) $ is trivial for all $ i $. It follows that $ \rho $ factors through $ G_{K,\emptyset}(p) $. Assuming Conjecture \ref{WUFMforcharp}, we see that $ \rho $ has finite image. The claim is proved.
		\item It is clear that $ (a)\implies (b) $ and $ (b)\implies (c) $. To see $ (b)\implies (a) $, assume that both Conjecture \ref{UFM} and Conjecture \ref{WUFMforcharp} hold. By $ 1 $, we see that Conjecture \ref{UFMforcharp} holds. By Theorem \ref{reduction}, we see that Conjecture \ref{BUFM} holds. To see $ (c)\implies (b) $, it is enough to show that Conjecture \ref{UFM} holds assuming Conjecture \ref{WUFM} and \ref{theimageofineritagroupisfinite}. Let $ \rho:G_{K,S}\to {\rm GL}_{n}(\mathbb{Q}_{p}) $ be a continuous homomorphism. By Lemma \ref{pro}, we see that $ \text{Im}(\rho) $ contains an open pro-$ p $ group. Moreover, $ \rho(I_{\mathfrak{q}}) $ is finite for all $ \mathfrak{q}\in S $ by our assumption. After taking a finite extension of $ K $, we can assume that $ \rho $ factors through $ G_{K,\emptyset}(p) $. By our assumption, we see that $ \rho $ has finite image. This finishes the proof of the theorem.
		
	\end{enumerate}
\end{proof}

\section{Some cases of Conjecture \ref{UFMforcharp}}\label{section3.5}

In this section, we verify some cases of Conjecture \ref{UFMforcharp}. Our main theorem in this section is the following:
\begin{Thm}\label{positivecharcase}
	Let $ K $ be a number field and $ S $ a finite set of primes of $ K $ such that $ S\cap S_{p}=\emptyset $. Let $ F $ be a non-Archimedean local field of characteristic $ p$, and $ \rho:G_{K,S}\to {\rm GL}_{n}(F) $ a continuous homomorphism. Assume further that at least one of the following conditions holds.
	\begin{enumerate}
		\item The group $ \text{Im}(\rho) $ is $ p $-adic analytic.
		\item The class number of $ K $ is prime to $ p $, and the group $ \text{Im}(\rho) $ is torsion-free.
	\end{enumerate}
	Then $ \rho $ has finite image.
\end{Thm}

To prove Theorem \ref{positivecharcase}, we need the following theorem.

\begin{Thm}\cite[Proposition 5.6]{MR1765116}\cite[Theorem 1.8]{MR1935507}\label{virtuallyabelian}
	Let $ G $ be a $ p $-adic analytic pro-$ p $ group. Suppose that $ G $ is isomorphic to a closed subgroup of $ {\rm GL}_{n}(F) $ where $ F $ is a non-Archimedean local field of characteristic $ p $. Then $ G $ is virtually abelian.
\end{Thm}

\begin{proof}[Proof of Theorem \ref{positivecharcase}]
	If $ \text{Im}(\rho) $ is a $ p $-adic analytic group, then it contains an open pro-$ p $ group $ H $. By Theorem \ref{virtuallyabelian}, we know that $ H $ is virtually abelian, and so is $  \text{Im}(\rho) $. By Theorem \ref{virtually solvable case}, we obtain that $  \text{Im}(\rho) $ is finite.
	
	If the group $ \text{Im}(\rho) $ is torsion-free, then $ \rho $ factors through $ G_{K,\emptyset}(p) $ by Lemma \ref{pro} and Theorem \ref{local}. Moreover, if the class number of $ K $ is prime to $ p $, then by Theorem \ref{global structure} we have $ G_{K,\emptyset}(p)=\{1\} $, and hence $ \rho $ is trivial. The theorem is proved.
\end{proof}

Finally, we point out that Allen and Calegari have also confirmed many cases of Conjecture \ref{UFMforcharp} as follows:

\begin{Thm}\label{abllen}
	Let $ K $ be a totally real number field and $ S $ a finite set of primes of $ K $ with $ S\cap S_{p}=\emptyset $. Let $ \mathbb{F} $ be a finite field of characteristic $ p $ and $ \mathbb{F}[[T]] $ the ring of formal power series over $ \mathbb{F} $. Let $ \rho:G_{K,S}\to {\rm GL}_{n}(\mathbb{F}[[T]]) $ be a continuous homomorphism. Suppose that $ p>2n^{2}-1 $ and $ \text{Im}(\overline{\rho})\supset {\rm SL}_{n}(\mathbb{F}_{p}) $ where $ \overline{\rho} $ the reduction of $ \rho $ modulo $ T $. Then $ \rho $ has finite image.
\end{Thm}
\begin{proof}
	By \cite[Theorem 1]{MR3294389}, it suffices to show that $ \text{ad}^{0}(\overline{\rho}|_{G_{K(\zeta_{p})}}) $ is absolutely irreducible where $ K(\zeta_{p}) $ is the $ p $-th cyclotomoic field of $ K $, $ \overline{\rho}|_{G_{K(\zeta_{p})}} $ is the restriction of the representation $\overline{\rho} $ on the open subgroup $ G_{K(\zeta_{p})} $ of $ G_{K} $ and $ \text{ad}^{0}(\overline{\rho}|_{G_{K(\zeta_{p})}}) $ is the subrepresentation of the adjoint representation of $ \overline{\rho}|_{G_{K(\zeta_{p})}} $ on trace zero matrices.
	
	We can assume that $n\geq 2$ by Theorem \ref{virtually solvable case}. Since $p>5$ and $p $ does not divide $n$, we know that the action of $ {\rm SL}_{n}(\mathbb{F}_{p}) $ on the $ n\times n $ matrices with trace $ 0 $ over $ \mathbb{F}_{p} $ is absolutely irreducible by \cite[Lemma 3.3]{MR3336600}. Thus, it suffices to show that the image of $ \bar{\rho}|_{G_{K(\zeta_{p})}} $ contains $ {\rm SL}_{n}(\mathbb{F}_{p}) $. Since $p\geq 5$, by \cite[Proposition 2.2]{MR3542491}, we know that the group $ {\rm SL}_{n}(\mathbb{F}_{p}) $ is perfect, i.e. the derived subgroup of $ {\rm SL}_{n}(\mathbb{F}_{p}) $ is itself. Since $ \text{Im}(\overline{\rho})\supset {\rm SL}_{n}(\mathbb{F}_{p}) $ and the extension $ K(\zeta_{p})/K $ is abelian, we see that the image of $ \overline{\rho}|_{G_{K(\zeta_{p})}} $ also contains $ {\rm SL}_{n}(\mathbb{F}_{p}) $. We are done.
\end{proof}

\section{Some cases of Conjecture \ref{UFM}}\label{chapter4}

Following \cite{MR3249411}, we introduce the following notion.  

\begin{Def}
	We say that a pro-$ p $ group $ G $ is a \textit{Fontaine-Mazur group} if any $ p $-adic analytic quotient of $ G $ is finite.
\end{Def}

Note that Conjecture \ref{WUFM} is equivalent to the claim that the pro-$ p $ group $ G_{K,\emptyset}(p) $ is a Fontaine-Mazur group for any number field $ K $. For example, if $ K $ is a number field with class number prime to $ p $, then by Theorem \ref{global structure} we have $ G_{K,\emptyset}(p)=\{1\} $, and hence it is a Fontaine-Mazur group. 

\subsection{When $ \text{Im}(\rho) $ is a $p$-valuable pro-$p$ group}\label{section4.3}

In this subsection, inspired by the work of \cite{MR3249411}, we consider those representations of global Galois group such that their image are $p$-valuable pro-$p$ groups. First of all, note that we can reformulate Conjecture \ref{UFM} as follows:

\begin{conj}\label{UFMusingpadicvaluedgroup}
	Let $ K $ be a number field and $ S $ a finite set of primes of $ K $ such that $ S\cap S_{p}=\emptyset $. If $ (\Gamma,\omega) $ is a $p$-valuable pro-$p$ group, then any continuous homomorphism $ \rho:G_{K,S}(p)\to \Gamma $ is trivial.
\end{conj}

\begin{Prop}
	The following are equivalent:
	\begin{enumerate}
		\item Conjecture \ref{UFM}.
		\item Conjecture \ref{UFMusingpadicvaluedgroup}
	\end{enumerate}
\end{Prop}
\begin{proof}
	We prove our claim as follows:
	\begin{itemize}
		\item (i) $\implies $ (ii): Let $\Gamma$ be a $p$-valuable pro-$p$ group and $\rho:G_{K,S}(p)\to \Gamma$ a continuous homomorphism. By Theorem \ref{lazard}, $\Gamma$ is $p$-adic analytic. That is, $\Gamma$ is isomorphic to a closed subgroup of ${\rm GL}_{n}(\mathbb{Z}_{p})$ for some integer $n\geq 1$. Note that $\Gamma$ is torsion-free. Conjecture \ref{UFM} implies that $\text{Im}(\rho)\subset \Gamma$ is finite, and hence $\text{Im}(\rho) $ is trivial. 
		\item (ii) $\implies $ (i): Let $\rho:G_{K,S}\to {\rm GL}_{n}(\mathbb{Q}_{p})$ be a continuous representation. We can assume that $\text{Im}(\rho)\subset {\rm GL}_{n}(\mathbb{Z}_{p})$. By Theorem \ref{lazard}, $\text{Im}(\rho)$ contains an open $p$-valuable pro-$p$ group $\Gamma$. After taking a finite extension of $K$, we can assume that $\text{Im}(\rho)=\Gamma$. Conjecture \ref{UFMusingpadicvaluedgroup} implies that $\text{Im}(\rho)$ is trivial. We are done. 
	\end{itemize}
\end{proof}

We confirm Conjecture \ref{UFMusingpadicvaluedgroup} in many cases as follows.
\begin{Thm}\label{simple}
	Let $ K $ be a number field such that the pro-$ p $ group $ G_{K,\emptyset}(p) $ is a Fontaine-Mazur group, and $ S=\{\mathfrak{p}_{1},\cdots,\mathfrak{p}_{d}\} $ a finite set of non-Archimedean primes of $ K $ such that $S\cap S_{p}=\emptyset$. If $ (\Gamma,\omega) $ is a $p$-valuable pro-$p$ group such that $ \omega(\gamma)>\max_{1\leq i\leq d} v_{p}(N(\mathfrak{p}_{i})-1) $ for all $ \gamma\in \Gamma $ where $ v_{p} $ denotes the usual $ p $-adic valuation of $ \mathbb{Q}_{p} $, then any continuous homomorphism $\rho: G_{K,S}(p)\to \Gamma $ is trivial.
\end{Thm}
\begin{proof}
	After removing redundant primes from $ S $, we can assume that $ \rho $ is exactly ramified at $ S $. By Lemma \ref{redundant}, we have $ N(\mathfrak{p}_{i})\equiv 1~\text{mod}~p $ for all $ i $.
	If $ S=\emptyset $, then we are done since $ G_{K,\emptyset}(p) $ is a Fontaine-Mazur group and $ \Gamma $ is torsion-free. Thus, we may assume that $ S\neq \emptyset $ (hence $ d\geq 1 $), and $ \rho $ is non-trivial. We want to deduce a contradiction.
	
	For each $ i $ with $ 1\leq i\leq d $, let $ \tau_{i} $ denote a generator of the inertia group at $ \mathfrak{p}_{i} $ and $ \sigma_{i} $ denote a lifting of the Frobenius automorphism at $ \mathfrak{p}_{i} $.
	Put $ \tilde{\tau_{i}}:=\rho(\tau_{i}) $ and $ \tilde{\sigma_{i}}=\rho(\sigma_{i}) $. By Theorem \ref{local}, we have
	\begin{align}\label{relation1}
		\tilde{\tau_{i}}^{N(\mathfrak{p}_{i})-1}=[\tilde{\sigma_{i}}^{-1},\tilde{\tau_{i}}^{-1}]
	\end{align}
	for each $ i $.
	By our assumption, $ \tilde{\tau_{i}}\neq 1 $. We claim that $ \tilde{\sigma_{i}}\neq 1 $ for each $ i $. Indeed, if $ \tilde{\sigma}_{i}=1 $, then we have $ \tilde{\tau_{i}}^{N(\mathfrak{p}_{i})}=\tilde{\tau_{i}} $ by (\ref{relation1}). But $ \text{Im}(\rho)$ is torsion-free which implies that $ \tilde{\tau_{i}}=1 $. It is a contradiction.
	
	In the following, we will use the properties (iv) and (vi) of the valuation $ \omega $ defined in Section \ref{padicanalyticgroup}. For each $ i $, by (\ref{relation1}), we have
	\[ \omega(	\tilde{\tau_{i}}^{N(\mathfrak{p}_{i})-1})=\omega(\tilde{\tau_{i}})+v_{p}(N(\mathfrak{p}_{i})-1)\geq \omega(\tilde{\tau_{i}})+\omega(\tilde{\sigma_{i}}). \]
	This implies that
	\[ \omega(\tilde{\sigma_{i}})\leq v_{p}(N(\mathfrak{p}_{i})-1) .\]
	But by our assumption, we have
	\[ \omega(\tilde{\sigma_{i}})> \max_{1\leq i\leq d} v_{p}(N(\mathfrak{p}_{i})-1)\geq v_{p}(N(\mathfrak{p}_{i})-1), \]
	which is a contradiction. This completes the proof of our theorem.
\end{proof}

When $ p=2 $, it is always true that $ \omega>1/(2-1)=1 $ by definition, and hence we have the following result. 

\begin{Cor}
	Let $ K $ be a number field such that $ G_{K,\emptyset}(2) $ is a Fontaine-Mazur group, and $ S=\{\mathfrak{p}_{1},\cdots,\mathfrak{p}_{d}\} $ a finite set of primes of $ K $ such that $ N(\mathfrak{p}_{i})\equiv 3~\text{mod}~4 $ for all $ i $. Then for any $2$-valuable pro-$2$ group $ \Gamma $, any continuous homomorphism $ \rho:G_{K,S}(2)\to \Gamma $ is trivial.
\end{Cor}

Furthermore, if $ L/\mathbb{Q}_{p} $ is a finite extension and $ m_{0} $ is the smallest integer which is larger than
\[ e_{L/\mathbb{Q}_{p}}\cdot  (\max_{1\leq i\leq d} v_{p}(N(\mathfrak{p}_{i})-1)), \]
then the $p$-valuable pro-$p$ group
\[ \Gamma:={\rm GL}_{n}^{m_{0}}(\mathcal{O}_{L})=\mathcal{G}_{n}(L)_{(\max_{1\leq i\leq d} v_{p}(N(\mathfrak{p}_{i})-1))+}\]
satisfies the condition in Theorem \ref{simple}. 

\begin{Convention}\label{linknumber}
	Recall that if $ S=\{q_{1},\cdots,q_{d}\} $ is a finite set of primes of $ \mathbb{Q} $ and we choose a primitive root $ \text{mod}~q_{i} $ for any $ i $, then we have defined the link numbers $ \ell_{ij}\in \mathbb{Z}/p\mathbb{Z} $ for all $i,j\in \{1,\cdots,d\} $ with $ i\neq j $ in Theorem \ref{global structure}. Note that $\ell_{ij}=0$ if and only if $q_{i}$ is a $p$-th power modulo $q_{j}$. 
\end{Convention}

\begin{Thm}\label{lequalto}
	Let $p$ be an odd prime. Let $ S=\{q_{1},\cdots,q_{d}\} $ be a finite set of rational primes of $ \mathbb{Q} $ not containing $p$ such that $ q_{i}\not \equiv 1~\text{mod}~p^{2} $ for all $ i $.  Assume that $ \ell_{ij}=0 $ for all $ i,j\in \{1,\cdots,d\} $ with $ i\neq j$. If $ (\Gamma,\omega) $ is a $p$-valuable pro-$p$ group such that $ \omega(\gamma)>1/2 $ for all $ \gamma\in \Gamma $, then any continuous homomorphism $ \rho:G_{\mathbb{Q},S}(p)\to \Gamma $ is trivial. 
\end{Thm}
\begin{proof}
	By Lemma \ref{redundant}, we can assume that $q_{i}\equiv 1~\text{mod}~p$ for all $i$. As in the proof of Theorem \ref{simple}, we may assume that $ \rho $ is exactly ramified at $ S $ and $ |S|\geq 1 $. Let $ \tau_{i} $ denote a generator of inertia group at $ q_{i} $ and $ \sigma_{i} $ denote a lifting of the Frobenius automorphism at $ q_{i} $ for each $ i $. Put $ \tilde{\tau}_{i}:=\rho(\tau_{i}) $ and $ \tilde{\sigma_{i}}:=\rho(\sigma_{i}) $. By Theorem \ref{global structure}, for each $ i $, we have 
	\begin{align}\label{jiayouaaa}
		\tilde{\tau_{i}}^{q_{i}-1}=[\tilde{\sigma_{i}}^{-1},\tilde{\tau_{i}}^{-1}],
	\end{align}
	and 
	\begin{align}\label{2}
		\tilde{\sigma_{i}}=\left( \prod_{j\neq i}\tilde{\tau}_{j}^{L_{ij}}\right) \cdot \gamma_{i}
	\end{align}
	where $ L_{ij}\in \mathbb{Z} $ satisfying its image in $\mathbb{Z}/p\mathbb{Z}$ is $\ell_{ij}$, and $\gamma_{i}$ is an element in $ \text{Im}(\rho)^{p}[\text{Im}(\rho),\text{Im}(\rho)] $. Since $\ell_{ij}=0$, we see that $ v_{p}(L_{ij})\geq 1 $. In the following, we will use the properties (iii), (iv) and (vi) of the valuation $ \omega $ defined in Section \ref{padicanalyticgroup}. Then we have
	\begin{align}\label{3}
		\omega(\tau_{i})>1/2,~\text{and}~~\omega(\gamma_{i})> 1.
	\end{align}
	
	By our assumption, we have $ v_{p}(q_{i}-1)=1 $ for each $ i $. The relation (\ref{jiayouaaa}) yields
	\[ \omega(\tilde{\sigma_{i}})\leq v_{p}(q_{i}-1)=1 .\]
	Together with (\ref{2}), (\ref{3}),	we have
	\[1= v_{p}(q_{i}-1)\geq \omega(\tilde{\sigma_{i}})=\omega(\prod_{j\neq i}\tilde{\tau}_{j}^{L_{ij}}\cdot \gamma_{i})\geq \min_{j\neq i}\{\omega(\tau_{j})+v_{p}(L_{ij}),\omega(\gamma_{i})\}>1, \]
	which is a contradiction. Therefore, we have $ S=\emptyset $ and $ \rho $ factors through $ G_{\mathbb{Q},\emptyset}=\{1\} $. This completes the proof of our theorem.
\end{proof}

\begin{Rem}
	By definition, we always have $\omega>1/2$ if $p=3$.
\end{Rem}

For example, if $ (\Gamma,\omega) $ has integer values, then $ \Gamma $ satisfies the condition of Theorem \ref{lequalto}. Moreover, we have:
\begin{Cor}\label{gl2}
	Let $ p $ be an odd prime. Let $ S=\{q_{1},\cdots,q_{d}\} $ be a finite set of rational primes of $ \mathbb{Q} $ not containing $p$ such that $ q_{i}\not \equiv 1~\text{mod}~p^{2} $ for all $ i $. Then we have the following:
	\begin{enumerate}
		\item Any continuous homomorphism $ \rho:G_{\mathbb{Q},S}(p)\to {\rm GL}_{n}^{2}(\mathbb{Z}_{p}) $ is trivial.
		\item Assume further that $ \ell_{ij}=0 $ for all $ i,j\in \{1,\cdots,d\}$ with $ i\neq j$. Then any continuous homomorphism $ \rho:G_{\mathbb{Q},S}(p)\to {\rm GL}_{n}^{1}(\mathbb{Z}_{p}) $ is trivial.  
	\end{enumerate}
\end{Cor}
\begin{proof}
	By Lemma \ref{redundant}, we can assume that $q_{i}\equiv 1~\text{mod}~p$ for all $i$.  
	In Example \ref{padicvaluedexample}, we have defined a valuation $ \omega $ on $ {\rm GL}_{n}^{1}(\mathbb{Z}_{p}) $ such that $ g\in {\rm GL}_{n}^{i}(\mathbb{Z}_{p})  $ if and only if $ \omega(g)\geq i$ for any $ g\in  {\rm GL}_{n}^{1}(\mathbb{Z}_{p}) $ and any positive integer $ i $. By our assumption, we have $ \max_{1\leq i\leq d} v_{p}(q_{i}-1)=1 $. Thus, the first claim follows from Theorem \ref{simple}, and the second claim follows from Theorem \ref{lequalto}. 
\end{proof}

\subsection{When $ K=\mathbb{Q} $ and $ |S|\leq 3 $}\label{section4.4}
In this section, we will keep the notation in Convention \ref{linknumber}. Recall that the pro-$ p $ group $ G_{\mathbb{Q},S}(p) $ is a finite cyclic group if $ |S|=1 $ by Theorem \ref{global structure}. For $ |S|=2 $, we have the following:
\begin{Thm}\label{S=2}
	Let $ p $ be an odd prime, and $ S=\{q_{1},q_{2}\} $ a set of primes of $ \mathbb{Q} $ such that $ q_{i}\equiv 1~\text{mod}~p $ for $ i=1,2 $. Assume that $ \ell_{12}\neq 0 $ or $\ell_{21}\neq 0$. Then the pro-$ p $ group $ G_{\mathbb{Q},S}(p) $ is finite.
\end{Thm}
\begin{proof}
	Without loss of generality, assume that $ \ell_{12}\neq 0 $. Let $ \tau_{i} $ denote a generator of the inertia group at $ q_{i} $ and $ \sigma_{i} $ denote a lifting of the Frobenius automorphism at $ q_{i} $ for each $ i $. We write $ H $ for $ G_{\mathbb{Q},S}(p) $. By Theorem \ref{global structure}, we know that the pro-$ p $ group $ H $ is topologically generated by $ \tau_{1},\tau_{2} $, and we have 
	\[ \sigma_{1}\equiv \tau_{2}^{\ell_{12}}~\text{mod}~\Phi(H), \]
	where $ \Phi(H)=H^{p}[H,H] $ is the Frattini subgroup of $ H $. Since $ H $ is topologically generated by $ \tau_{1},\tau_{2} $, we find that the set $ \{\overline{\tau_{1}},\overline{\tau_{2}} \}$ is a basis of the two dimensional $ \mathbb{F}_{p} $-vector space $ H/\Phi(H) $ where $ \overline{\tau_{i}} $ is the image of $ \tau_{i} $ in $ H/\Phi(H) $ for $ i=1,2 $. Since $ \ell_{12}\neq 0 $, we know that the set $ \{\overline{\tau_{1}},\overline{\sigma_{1}} \}$ is also a basis of the $ \mathbb{F}_{p} $-vector space $ H/\Phi(H) $ where $ \overline{\sigma_{1}} $ is the image of $ \sigma_{1} $ in $ H/\Phi(H) $. It follows that the set $ \{\tau_{1},\sigma_{1}\} $ topologically generates $ H $. But the decomposition group at $ q_{1} $ generated by $ \tau_{1},\sigma_{1} $ is isomorphic to a quotient of a semi-direct product $ \mathbb{Z}_{p}\ltimes \mathbb{Z}_{p} $ which is a solvable group by Proposition \ref{tamelyramified}. By Theorem \ref{virtually solvable case}, $ H $ must be finite.
\end{proof}
\begin{Rem}
	\begin{enumerate}
		\item Assume further that $ q_{i}\not \equiv 1~\text{mod}~p^{2} $ for $i=1,2$. Then $ G_{\mathbb{Q},S}(p) $ is isomorphic to the nonabelian group with order $ p^{3} $ and exponent $ p^{2} $, see \cite[Example 11.15]{MR1930372}.
		\item In the general case, it has been shown that $ G_{\mathbb{Q},S}(p) $ can be infinite when $ |S|=2 $ in \cite[Theorem 5.9]{MR1901171}.
	\end{enumerate}
\end{Rem}

Next, we consider the case of $ |S|=3 $. First, we need an important lemma as follows:

\begin{Lem}\label{sl21}
	Let $ p $ be an odd prime. Let $ G $ be a powerful pro-$ p $ group defined by $ 3 $ generators and $ 3 $ relations such that $ G/[G,G]\simeq (\mathbb{Z}/p\mathbb{Z})^{3} $. Then either $ G $ is finite or $ G $ is isomorphic to $ {\rm SL}_{2}^{1}(\mathbb{Z}_{p}) $.
\end{Lem}
\begin{proof}
	Since the pro-$ p $ group $ G $ is powerful and $ G/[G,G] $ is an elementary abelian $ p $-group, we see that $ G^{p}=[G,G] $. Then our claim follows from \cite[Theorem 1.1 and Corollary 1]{andozhskii1974series}.
\end{proof}

Based on the work of Labute in \cite{MR3249411}, we prove the following theorem. 
\begin{Thm}\label{powerfulisfinite}
	Let $ p $ be an odd prime, and $ S=\{q_{1},q_{2},q_{3}\} $ a finite set of primes of $ \mathbb{Q} $ such that $ q_{i}\equiv 1~\text{mod}~p $ but $ q_{i}\not \equiv 1~\text{mod}~p^{2} $ for each $ i\in \{1,2,3\} $. Suppose that the pro-$ p $ group $ G_{\mathbb{Q},S}(p) $ is powerful. If the group $ G_{\mathbb{Q},S}(p) $ is infinite, then it is isomorphic to ${\rm SL}_{2}^{1}(\mathbb{Z}_{p})$. Moreover, in this case, we must have $ \ell_{ij}\neq 0 $ for all $ i,j $ and $ \ell_{13}/c_{1}=-\ell_{23}/c_{2},~\ell_{21}/c_{2}=-\ell_{31}/c_{3},~\ell_{12}/c_{1}=-\ell_{32}/c_{3} $ where $ c_{i}:=(q_{i}-1)/p $ for each $ i $.
\end{Thm}
\begin{proof}
	By Theorem \ref{global structure}, we see that the pro-$ p $ group $ G_{\mathbb{Q},S}(p) $ is defined by $ 3 $ generators and $ 3 $ relations, and the abelianization $ G/[G,G] $ of $G $ is isomorphic to $(\mathbb{Z}/p\mathbb{Z})^{3} $ since $  q_{i}\not \equiv 1~\text{mod}~p^{2} $ for all $ i $. By Lemma \ref{sl21}, we see that either $ G_{\mathbb{Q},S}(p) $ is finite or $ G_{\mathbb{Q},S}(p) $ is isomorphic to $ {\rm SL}_{2}^{1}(\mathbb{Z}_{p}) $. Thus, if $G_{\mathbb{Q},S}(p)$ is infinite, then it is isomorphic to ${\rm SL}_{2}^{1}(\mathbb{Z}_{p})$. Our second claim follows from \cite[Theorem 1.5 and 1.7]{MR3249411}. The theorem is proved.
\end{proof}

Labute gave the following example in \cite{MR3249411}: if $ p=3,~S=\{7,31,229\}$ or $ p=5,~S=\{11,31,1021\} $, then the condition of Theorem \ref{powerfulisfinite} is satisfied, and it is not known whether $G_{\mathbb{Q},S}(p)$ is infinite or not.

\subsection{Conjecture \ref{UFM} in the two-dimensional case}\label{chapter5}
\subsubsection{Statements} In this subsection, we will keep the notation in Convention \ref{linknumber}. Our second main result of this paper is the following theorems.

\begin{Thm}\label{bigimage}
	Let $ p$ be an odd prime. Let $ K $ be a number field, $ S $ a finite set of primes of $ K $ such that $ S\cap S_{p}=\emptyset $, and $ \rho:G_{K,S}\to {\rm GL}_{2}(\mathbb{Z}_{p}) $ a continuous homomorphism. Then either $ \rho $ has finite image or there exists an integer $ i\geq 1 $ such that $ {\rm SL}_{2}^{i}(\mathbb{Z}_{p}) $ is an open subgroup of $ \text{Im}(\rho) $.
\end{Thm}

In other words, if Conjecture \ref{UFM} fails, then the corresponding representation will have large image. Furthermore, inspired by the work of \cite{MR3249411}, we prove:
\begin{Thm}\label{sl21p}
	Let $ p$ be an odd prime, and $ S=\{q_{1},\cdots,q_{d}\} $ a finite set of primes of $ \mathbb{Q} $ not containing $p$ such that $ q_{i}\not \equiv 1~\text{mod}~p^{2} $ for each $ i\in \{1,\cdots,d\} $. Let $\rho:G_{\mathbb{Q},S}\to {\rm GL}_{2}(\mathbb{Z}_{p})$ be a continuous homomorphism. Suppose that the image $\text{Im}(\overline{\rho})$ of the reduction $\overline{\rho}$ of $\rho$ is trivial. Then either $\rho$ is trivial or $\text{Im}(\rho)={\rm SL}_{2}^{1}(\mathbb{Z}_{p})$. Assume further that $ \ell_{ij}=0 $ for all $ i,j\in \{1,\cdots,d\} $ with $ i\neq j$. Then $\rho$ is trivial.
\end{Thm}

The following theorem shows that Conjecture \ref{theimageofineritagroupisfinite} (or Conjecture \ref{UFM}) cannot be simply reduced to a local problem.
\begin{Thm}\label{notalocalproblem}
	Let $p$ be an odd prime and let $q$ be a prime such that $q\equiv 1~\text{mod}~p$. Let $G$ be the pro-$ p $ group topologically generated by two generators $ \sigma,\tau $ subject to the only relation $ \sigma\tau\sigma^{-1}=\tau^{q} $. Then there exists a continous homomorphism $\rho:G\to {\rm SL}_{2}^{1}(\mathbb{Z}_{p})$ such that $\rho(\tau)$ has infinite order. 
\end{Thm}

When $\text{Im}(\rho)\subset {\rm SL}_{2}(\mathbb{Z}_{p})$, we prove the following:

\begin{Thm}\label{SL2image}
	Let $ p>3 $ be a prime. Let $ S=\{q_{1},\cdots,q_{d}\} $ be a finite set of primes of $ \mathbb{Q} $ not containing $ p $ and $ \rho:G_{\mathbb{Q},S}\to {\rm SL}_{2}(\mathbb{Z}_{p}) $ a continuous homomorphism. Suppose that $ q_{i} $ does not divide $ p^{2}-1 $ for every $ i \in \{1,\cdots,d\}$. Then we have the following:
	\begin{enumerate}
		\item Either $ \text{Im}(\rho) $ is finite or $ \text{Im}(\rho) $ is an open subgroup of $ {\rm SL}_{2}(\mathbb{Z}_{p}) $.
		\item $ \rho $ is tamely ramified at each $ q_{i} $. In particular, $\rho(I_{q_{i}})$ is a cyclic group.
		\item Let $i\in \{1,\cdots,d\}$. Assume that $x^{2}\not \equiv q_{i}~\text{mod}~p$ for any $x\in \mathbb{Z}$. Then the group $\rho(I_{q_{i}})$ is finite where $ I_{q_{i}} $ is the inertia group at $ q_{i} $.
		\item Suppose that the group $\rho(I_{q_{i}})$ is finite for some $i$. Then the order of the cyclic group $ \rho(I_{q_{i}}) $ divides $ p-1 $ or $ p+1$.
		\item  Suppose that the group $\rho(I_{q_{i}})$ is finite for some $i$. Then the order of the cyclic group $ \rho(I_{q_{i}}) $ divides $ q_{i}-1 $ or $ q_{i}+1 $.
		\item Suppose that the group $\rho(I_{q_{i}})$ is finite for any $i\in \{1,\cdots,d\}$. Assume that at least one of the following conditions holds. Let $\gcd$ denote the greatest common divisor.
		\begin{enumerate}
			\item One has $ \gcd(p-1,q_{i}-1)=2 $, and the order of the cyclic group $ \rho(I_{q_{i}}) $ divides $ p-1 $ and $ q_{i}-1 $ for all $ i $.
			\item One has $ \gcd(p-1,q_{i}+1)=2 $, and the order of the cyclic group $ \rho(I_{q_{i}}) $ divides $ p-1 $ and $ q_{i}+1 $ for all $ i $.
			\item One has $ \gcd(p+1,q_{i}-1)=2 $, and the order of the cyclic group $ \rho(I_{q_{i}}) $ divides $ p+1 $ and $ q_{i}-1 $ for all $ i $.
			\item One has $ \gcd(p+1,q_{i}+1)=2 $, and the order of the cyclic group $ \rho(I_{q_{i}}) $ divides $ p+1 $ and $ q_{i}+1 $ for all $ i $.
		\end{enumerate}
		Then the order $ |\text{Im}(\rho)| $ of the group $ \text{Im}(\rho) $ satisfies $ |\text{Im}(\rho)|\leq 2 $. 
	\end{enumerate}
\end{Thm}

\subsubsection{Proofs}

\begin{proof}[Proof of Theorem \ref{bigimage}]
	Suppose that $ \text{Im}(\rho) $ is infinite. Then we consider the following extension of profinite groups
	\begin{align*}
		1\to \rho(\ker(\overline{\rho}))\to \text{Im}(\rho)\to \text{Im}(\overline{\rho})\to 1.
	\end{align*}
	By Theorem \ref{virtually solvable case}, we see that $ \rho $ has finite determinant. Since $ p>2 $, we know that $ \mu(\mathbb{Z}_{p})=\mathbb{F}_{p}^{\times} $ where $ \mu(\mathbb{Z}_{p}) $ is the multiplicative group consisting of the roots of unity in $ \mathbb{Z}_{p} $.
	Since the group $ \rho(\ker(\overline{\rho})) $ is pro-$ p $, we see that any matrix in $ \rho(\ker(\overline{\rho}))\subset {\rm GL}_{2}^{1}(\mathbb{Z}_{p}) $ has determinant of order $ p^{m} $ for some positive integer $ m $ and hence has determinant $ 1 $. That is, we have
	\begin{align}\label{lessthandim}
		\rho(\ker(\overline{\rho}))\subset {\rm SL}_{2}^{1}(\mathbb{Z}_{p}).
	\end{align}
	By \cite[Theorem 5.2]{MR1720368}, we see that the dimension of $ {\rm SL}_{2}^{1}(\mathbb{Z}_{p}) $ is $ 3 $ as a $ p $-adic analytic group. By \cite[Exercise 4, Chapter 8]{MR1720368} and (\ref{lessthandim}), we have 
	\[ \dim(\rho(\ker(\overline{\rho})))\leq 3. \]
	Since $ \text{Im}(\rho) $ is infinite and $ \text{Im}(\bar{\rho}) $ is finite, we see that $ \rho(\ker(\overline{\rho})) $ is infinite. This implies that $ \dim(\rho(\ker(\overline{\rho})))=3 $ by Corollary \ref{dimlessthan3}. By Lemma \ref{closedsubgrouphavingsamediemsnion}, we know that $  \rho(\ker(\overline{\rho})) $ is an open subgroup of $ {\rm SL}_{2}^{1}(\mathbb{Z}_{p}) $. By \cite[Exercise 10, Chapter 1]{MR1720368}, we obtain that $  \rho(\ker(\overline{\rho})) $ contains $ {\rm SL}_{2}^{i}(\mathbb{Z}_{p}) $ for some $ i\geq 1 $. It is clear that $ {\rm SL}_{2}^{i}(\mathbb{Z}_{p}) $ is an open subgroup of $ \text{Im}(\rho) $. This completes the proof of our theorem.
\end{proof}

To prove Theorem \ref{sl21p}, we need the following lemma.

\begin{Lem}\label{sl2}\cite[Lemma 3.13]{abdellatif2022fontaine}
	Let $ M_{2}^{0}(\mathbb{F}_{p})$ denote the three-dimensional $\mathbb{F}_{p}$-linear space of $2\times2 $ matrices with trace zero over $ \mathbb{F}_{p} $. Then the Frattini subgroup of ${\rm SL}_{2}^{1}(\mathbb{Z}_{p})$ is ${\rm SL}_{2}^{2}(\mathbb{Z}_{p})$, and the group homomorphism from the group ${\rm SL}_{2}^{1}(\mathbb{Z}_{p})/{\rm SL}_{2}^{2}(\mathbb{Z}_{p})$ to $  M_{2}^{0}(\mathbb{F}_{p})$ given by 
	\[ 1+pA\mapsto \overline{A}, \]
	is an isomorphism where $\overline{A}$ is the image of $A\in M_{2}(\mathbb{Z}_{p})$ in $M_{2}(\mathbb{F}_{p})$.
\end{Lem}

\begin{proof}[Proof of Theorem \ref{sl21p}]
	First of all, we prove our first assertion as follows. By our assumption, we have $\text{Im}(\rho)\subset {\rm GL}_{2}^{1}(\mathbb{Z}_{p})$.
	Since ${\rm GL}_{2}^{1}(\mathbb{Z}_{p})$ is a pro-$p$ group, the group $\text{Im}(\rho)$ is also pro-$p$, and hence $\rho$ factors through the maximal pro-$p$ quotient $G_{\mathbb{Q},S}(p)$ of $G_{\mathbb{Q},S}$. By Lemma \ref{redundant}, we can assume that $q_{i}\equiv 1~\text{mod}~p$ for all $i$.  
	Set $ c_{i}:=(q_{i}-1)/p $ for each $ i $. By Theorem \ref{global structure}, we see that the pro-$ p $ group $ G:=G_{\mathbb{Q},S}(p) $ is defined by $ d $ generators $\tau_{1},\cdots,\tau_{d}$ and $ d $ relations
	\[ r_{i}=\tau_{i}^{pc_{i}}[\tau_{i}^{-1},\sigma_{i}^{-1}], \]
	\[ \sigma_{i}\equiv \prod_{j\neq i}\tau_{j}^{\ell_{ij}},~\text{mod}~G_{2},\] 
	where $G_{i}$ is the lower $p$-central series of $G$, cf. \cite[Def. 1.15]{MR1720368}. It follows that 
	\begin{align}\label{moduloF3}
		r_{i}\equiv \tau_{i}^{pc_{i}}\prod_{j\neq i}[\tau_{i},\tau_{j}]^{\ell_{ij}}~\text{mod}~G_{3}.
	\end{align}
	Since the abelianization $G^{\text{ab}}$ of $G$ is finite and $p>2$, we see that $\text{Im}(\rho)\subseteq {\rm SL}_{2}^{1}(\mathbb{Z}_{p})$. Since the group ${\rm SL}_{2}^{1}(\mathbb{Z}_{p})$ is torsion-free, we know that the group $\text{Im}(\rho) $ is torsion-free. Thus, if the group $\text{Im}(\rho) $ is finite, then it must be trivial. 
	
	Suppose now that $\text{Im}(\rho)$ is infinite. In Example \ref{padicvaluedexample}, we have shown that ${\rm GL}_{2}^{1}(\mathbb{Z}_{p})$ is a $p$-valued group where the valuation $\omega$ on ${\rm GL}_{2}^{1}(\mathbb{Z}_{p})$ is defined by 
	\[ \omega(A)=w(A-I_{2}), \]
	where $w(A)=\min_{i,j}v_{p}(a_{ij})$ for any $A=(a_{ij})\in {\rm GL}_{2}^{1}(\mathbb{Z}_{p})$. Moreover, we have ${\rm GL}_{2}^{i}(\mathbb{Z}_{p})=\{A\in {\rm GL}_{2}^{1}(\mathbb{Z}_{p})~|~\omega(A)\geq i\}$ for any positive integer $i$.
	Using the properties of the valuation $\omega$ on ${\rm GL}_{2}^{1}(\mathbb{Z}_{p})$, we see that the image of any element of $G_{3}$ under $\rho$ is contained in ${\rm GL}_{2}^{3}(\mathbb{Z}_{p})=1+p^{3}\mathbb{Z}_{p}$. Let $\rho(\tau_{i})=1+pA_{i}$ for every $i$. By (\ref{moduloF3}), we have
	\[ 1=\rho(r_{i})\equiv 1+p^{2}\left( c_{i}A_{i}+\sum_{j\neq i}\ell_{ij}[A_{i},A_{j}]\right),~\text{mod}~p^{3}.  \]
	where $[A_{i},A_{j}]:=A_{i}A_{j}-A_{j}A_{i}$.
	Let $\overline{A_{i}}$ denote the image of $A_{i}$ in $M_{2}(\mathbb{F}_{p})$. Then we have
	\begin{align}\label{modp3}
		c_{i}\overline{A_{i}}+\sum_{j\neq i}\ell_{ij}[\overline{A_{i}},\overline{A_{j}}]=0.
	\end{align}
	We claim that the set $\{\overline{A_{i}}\}_{1\leq i\leq d}$ generates $  M_{2}^{0}(\mathbb{F}_{p})$. If not, then the rank of the set $\{\overline{A_{i}}\}_{1\leq i\leq d}$ in the linear space $ M_{2}^{0}(\mathbb{F}_{p})$ is less than $3$, i.e. there exist $j,k$ such that each $\overline{A_{i}}$ is a linear combination $\overline{A_{j}},\overline{A_{k}}$. By (\ref{modp3}), we see that $\overline{A_{j}}$ and $\overline{A_{k}}$ are both multiples of $[\overline{A_{j}},\overline{A_{k}}]$. This implies that either $\overline{A_{k}}=0$ or $\overline{A_{k}}$ is a multiple of $\overline{A_{j}}$. In either case, we always have $ \overline{A_{j}}=\overline{A_{k}}=0$, and hence all $\overline{A_{i}}$ are zero. Since the group $\text{Im}(\rho)$ is topologically generated by $1+pA_{1},\cdots,1+pA_{d}$, we see that $\text{Im}(\rho)\subset {\rm GL}_{2}^{2}(\mathbb{Z}_{p})$. But by Corollary \ref{gl2}, we must have $\text{Im}(\rho)$ is trivial which is a contradiction. We conclude that the set $\{\overline{A_{i}}\}_{1\leq i\leq d}$ generates $  M_{2}^{0}(\mathbb{F}_{p})$. By Lemma \ref{sl2}, we see that $\text{Im}(\rho)={\rm SL}_{2}^{1}(\mathbb{Z}_{p})$. The first assertion of our theorem is proved.
	
	Finally, the second assertion follows from Corollary \ref{gl2}. This completes the proof of our theorem.
	
\end{proof}

To prove Theorem \ref{notalocalproblem}, we need the following simple fact.

\begin{Lem}\label{extend}
	Suppose that a pro-$ p $ group $ G $ has a presentation $ F(x_{1},\cdots,x_{d})/R $ where $ F(x_{1},\cdots,x_{d}) $ is the free pro-$ p $ group on $ x_{1},\cdots,x_{d} $ and $ R=(r_{1},\cdots,r_{s}) $. Let $ H $ be another pro-$ p $ group and $ f:\{x_{1},\cdots,x_{d}\}\to H  $ a funciton. In an obvious way, $ f $ extends to a function $ F(x_{1},\cdots,x_{d})\to H $. That is, if $ g=x_{i_{1}}\cdots x_{i_{r}} $, then we put $ f(g)=f(x_{i_{1}})\cdots f(x_{i_{r}}) $.
	Suppose that $ f(r_{i})=1 $ for all $ i\in \{1,\cdots,s\} $. Then the function $ f$ extends to a continuous homomorphism $ G\to H $.
\end{Lem}
\begin{proof}
	By the universal property of free pro-$ p $ group (cf. \cite[Theorem 4.6]{MR1930372}), the function $ f:F(x_{1},\cdots,x_{d})\to H $ is a continuous homomorphism. Let $ K:=\ker(f) $. Since $ f(r_{i})=1 $ for all $ i $, we have $ \{r_{1},\cdots,r_{s}\}\subset K $. By definition of $ R $, we find that $ R $ is a closed subgroup of $ K $. It follows that the homomorphism $ f $ factors through $ F(x_{1},\cdots,x_{d})/R=G $. In this way, we get a continuous homomorphism $ G\to H $ that extends $ f $.
\end{proof}

\begin{proof}[Proof of Theorem \ref{notalocalproblem}]
	By our assumption, the pro-$p$ group $G$ has a presentation $F(x,y)/(r_{1})$ where $F(x,y)$ is the free pro-$p$ group on $x,y$ and $r_{1}=x^{q-1}[x^{-1},y^{-1}]$. Since $q\equiv 1~\text{mod}~p$ and $p>2$, there is an element $\sqrt{q}\in 1+p\mathbb{Z}_{p}$ such that $ (\sqrt{q})^{2}=q$ by Hensel's lemma. Then we consider the function $\rho:\{x,y\}\to {\rm SL}_{2}^{1}(\mathbb{Z}_{p})$ as follows:
	\[ \rho(x)=\begin{pmatrix}
		1& p \\
		0	& 1
	\end{pmatrix},\qquad \rho(y)=\begin{pmatrix}
		\sqrt{q}& 0 \\
		0& \sqrt{q}^{-1}
	\end{pmatrix} .\]
	A simple calculation shows that $f(r_{1})=1$. Note that $\rho(x)$ has infinite order. Then our claim follows from Lemma \ref{extend}.
\end{proof}

To prove Theorem \ref{SL2image}, we need the following simple lemma.

\begin{Lem}\cite[Section 8, Chapter 2]{MR569209} \label{easylemma}
	The order of the group $ {\rm SL}_{2}(\mathbb{F}_{p}) $ is $ p(p^{2}-1) $. Moreover, if $ x $ is an element of the group $ {\rm SL}_{2}(\mathbb{F}_{p}) $ such that its order is relatively prime to $ p $, then the order of $ x $ divides $ p-1 $ or $ p+1 $.
\end{Lem}

\begin{proof}[Proof of Theorem \ref{SL2image}]
	The claim (i) follows from Theorem \ref{bigimage}. For the claim (ii), note that the composition $ \rho(P_{q_{i}})\hookrightarrow {\rm SL}_{2}(\mathbb{Z}_{p})\twoheadrightarrow {\rm SL}_{2}(\mathbb{F}_{p}) $ of morphisms is injective since the wild inertia subgroup $ P_{q_{i}} $ at $ q_{i} $ is a pro-$ q_{i} $ group and the kernel $ {\rm SL}_{2}^{1}(\mathbb{Z}_{p}) $ of the reduction map $ \eta:{\rm SL}_{2}(\mathbb{Z}_{p})\twoheadrightarrow {\rm SL}_{2}(\mathbb{F}_{p}) $ is a pro-$ p $ group. Since the order of the group $ {\rm SL}_{2}(\mathbb{F}_{p}) $ is $ p(p^{2}-1) $ and $ q_{i} $ does not divide $ p^{2}-1 $, we see that $ \rho(P_{q_{i}}) $ is trivial. That is, $ \rho $ is tamely ramified at each $ q_{i} $. The claim (ii) is proved.
	
	For the claim (iii), note that $\rho$ will factor through $G_{\mathbb{Q},S}^{\text{tame}}$ by the claim (ii). By abuse of notation, we still write $\rho:G_{\mathbb{Q},S}^{\text{tame}}\to {\rm SL}_{2}(\mathbb{Z}_{p})$ for the corresponding representation. Let $ \tau_{i} $ be a generator of the inertia subgroup at $ q_{i} $ of the Galois group $ G_{\mathbb{Q},S}^{\text{tame}} $, and let $ \sigma_{i} $ be a lifting of the Frobenius automorphism at $ q_{i} $. By Proposition \ref{tamelyramified}, we have the relation 
	\begin{align}\label{tamerelation}
		\sigma_{i}\tau_{i}\sigma_{i}^{-1}=\tau_{i}^{q_{i}}.
	\end{align}
	If $ \rho(\tau_{i}) $ has distinct eigenvalues, then $ \rho(\tau) $ is diagonalizable over the algebraic closure $ \overline{\mathbb{Q}}_{p} $ of $ \mathbb{Q}_{p}$.
	By Proposition \ref{tamelyramified}, we know that $ \rho(\tau_{i}) $ has finite order. Thus, we can assume that $ \rho(\tau_{i}) $ has the repeated eigenvalues $ \lambda $ which is equal to $1$ or $-1$.  After conjugation we can assume that $\rho(\tau_{i})= \begin{pmatrix}
		\lambda	&  \beta \\
		0	&  \lambda
	\end{pmatrix} $ where $ \beta \in \mathbb{Z}_{p} $. If $\beta=0$, then $\rho(\tau_{i})^{2}=1$, and we are done. Otherwise, we assume that $\beta\neq 0$. We write $ \rho(\sigma_{i})=\begin{pmatrix}
		a&b  \\
		c& d
	\end{pmatrix} \in {\rm SL}_{2}(\mathbb{Z}_{p}) $. Using the relation $ \rho(\sigma_{i})\rho(\tau_{i})\rho(\sigma_{i})^{-1}=\rho(\tau_{i})^{q_{i}} $, we see that
	\[ \begin{pmatrix}
		a	& b \\
		c	& d
	\end{pmatrix}\begin{pmatrix}
		\lambda & \beta \\
		0& \lambda
	\end{pmatrix}\begin{pmatrix}
		a	& b \\
		c	& d
	\end{pmatrix}^{-1}=\begin{pmatrix}
		\lambda^{q_{i}}& q_{i}\beta\lambda^{q_{i}-1} \\
		0& \lambda^{q_{i}}
	\end{pmatrix} .\]
	That is,
	\[\begin{pmatrix}
		\lambda-ac\beta	&a^{2}\beta  \\
		-c^{2}\beta	&  \lambda+ac\beta
	\end{pmatrix}=\begin{pmatrix}
		\lambda^{q_{i}}& q_i\beta\lambda^{q_{i}-1} \\
		0& \lambda^{q_{i}}
	\end{pmatrix}. \]
	Since $ \beta\neq 0 $, we know that
	\begin{align*}
		c=0,~\lambda^{q_{i}-1}=1,~a=q_{i}d.
	\end{align*}
	Since $\rho(\sigma_{i})\in {\rm SL}_{2}(\mathbb{Z}_{p})$, we have $ad=a\cdot (a/q_{i})=a^{2}/q_{i}=1$. That is, we have $a^{2}=q_{i}$ in $\mathbb{Z}_{p}$. But this contradicts our assumption. We conclude that $\rho(\tau_{i})$ has finite order. The claim (iii) is proved.
	
	For the claim (iv), note that $ {\rm SL}_{2}(\mathbb{Z}_{p}) $ is $ p $-torsion-free. It follows that the order of each $ \rho(I_{q_{i}}) $ is prime to $ p $. Since $ {\rm SL}_{2}^{1}(\mathbb{Z}_{p}) $ is a pro-$ p $ group, we see that $ \eta $ induces an isomorphism $ \rho(I_{q_{i}})\simeq \bar{\rho}(I_{q_{i}}) $. By Lemma \ref{easylemma}, we obtain that the order of $  \overline{\rho}(I_{q_{i}}) $ divides $ p-1 $ or $ p+1 $, and so does $ \rho(I_{q_{i}}) $. The claim (iv) is proved.
	
	For the claim (v), after taking a finite integral extension $ \mathcal{O} $ of $ \mathbb{Z}_{p} $ and a change of basis, we can assume that 
	\[ \rho(I_{q_{i}})=\begin{pmatrix}
		\omega	&0  \\
		0&\omega^{-1} 
	\end{pmatrix},~\omega \in \mathcal{O}, \]
	where $ \omega $ satisfies $ \omega^{p-1}=1 $ or $ \omega^{p+1}=1 $. Let $ \sigma_{i} $ be a lifting of the Frobenius automorphism at $ q_{i} $, and put $ \rho(\sigma_{i})=\begin{pmatrix}
		a	& b \\
		c	& d
	\end{pmatrix} \in {\rm SL}_{2}(\mathbb{Z}_{p})$. By (\ref{tamerelation}), we have 
	\[ \begin{pmatrix}
		a	& b \\
		c	& d
	\end{pmatrix} \begin{pmatrix}
		\omega	&0  \\
		0&\omega^{-1} 
	\end{pmatrix} \begin{pmatrix}
		a	& b \\
		c	& d
	\end{pmatrix}^{-1}=\begin{pmatrix}
		\omega	&0  \\
		0&\omega^{-1} 
	\end{pmatrix}^{q_{i}}. \]
	That is,
	\begin{align}\label{computation}
		\begin{pmatrix}
			ad\omega-bc\omega^{-1}	& ab(-\omega+\omega^{-1}) \\
			cd(\omega-\omega^{-1})	& ad\omega^{-1}-bc\omega
		\end{pmatrix} =\begin{pmatrix}
			\omega^{q_{i}}&0  \\
			0& \omega^{-q_{i}} 
		\end{pmatrix}.
	\end{align}
	It follows that $ (ab-cd)(-\omega+\omega^{-1})=0 $. We divide the discussion into four cases as follows.
	\begin{enumerate}
		\item[Case $ 1 $]: If $ ab-cd\neq 0 $, then $ \omega^{2}=1 $, and hence the order of $ \rho(I_{q_{i}}) $ divides $ 2 $. 
		\item[Case $ 2 $]: If $ ab=cd $ and $ abcd\neq 0 $, then by (\ref{computation}) we have 
		\[(	ad\omega-bc\omega^{-1})(ad\omega^{-1}-bc\omega)=\omega^{q_{i}}\omega^{-q_{i}}=1=(ad-bc)^{2} . \]
		A simple calculation shows that $ abcd(\omega^{2}+\omega^{-2})=abcd\cdot 2 $. Since $ abcd\neq 0 $, we have $ \omega^{2}=1 $, and hence the order of $ \rho(I_{q_{i}}) $ divides $ 2 $.
		\item[Case $ 3 $]: If $ ab=cd $ and $ a=d=0 $, then (\ref{computation}) becomes
		\[ \begin{pmatrix}
			\omega^{-1}& 0 \\
			0& \omega
		\end{pmatrix}=\begin{pmatrix}
			\omega^{q_{i}}&0  \\
			0& \omega^{-q_{i}} 
		\end{pmatrix}. \]
		That is, we have $ \omega^{q_{i}+1}=1 $, and hence the order of $ \rho(I_{q_{i}}) $ divides $ q_{i}+1 $.
		\item[Case $ 4 $]: If $ ab=cd $ and $ b=c=0 $, then (\ref{computation}) becomes
		\[ \begin{pmatrix}
			\omega& 0 \\
			0& \omega^{-1}
		\end{pmatrix}=\begin{pmatrix}
			\omega^{q_{i}}&0  \\
			0& \omega^{-q_{i}} 
		\end{pmatrix}. \]
		That is, we have $ \omega^{q_{i}-1}=1 $, and hence the order of $ \rho(I_{q_{i}}) $ divides $ q_{i}-1 $.
	\end{enumerate}
	
	Since $ q_{i} $ does not divide $ p^{2}-1 $, we see that $ q_{i}>4 $, and hence both $ q_{i}+1 $ and $ q_{i}-1 $ are even. The claim (v) is proved.
	
	For the claim (vi), note that $ -I_{2} $ is the unique element of order $ 2 $ in $ {\rm SL}_{2}(\mathbb{Z}_{p}) $ because $ p>2 $ where $ I_{2} $ is the identity matrix of size $ 2 $. By our assumption, we have $ \rho(I_{q_{i}})\subset \{\pm I_{2}\}\subset {\rm SL}_{2}(\mathbb{Z}_{p}) $ for all $ i $. Since the group $ G_{\mathbb{Q},S} $ is topologically generated by all $ I_{q_{i}} $, we see that $ \text{Im}(\rho)\subset \{\pm I_{2}\} $. Therefore, the order $ |\text{Im}(\rho)| $ of $ \text{Im}(\rho) $ satisfies $ |\text{Im}(\rho)|\leq 2 $. The theorem is proved.
\end{proof}

\section{Universal deformation rings}\label{chapter6}

\subsection{Preliminaries on deformations of representations of profinite groups}\label{section6.1}

In this subsection, we recall some basic theory of deformations of representations of profinite groups. Our main references are \cite{MR3184335}, \cite{MR1860043} and \cite{MR1012172}.
\begin{Convention}\label{notai}
	For the rest of this section,
	\begin{itemize}
		\item $ \mathbb{F} $ a finite field of characteristic $ p $.
		\item $ W(\mathbb{F}) $ the ring of Witt vectors of $ \mathbb{F} $.
		\item $ G $ a profinite group.
		\item $ V_{\mathbb{F}} $ a finite $ \mathbb{F}[G] $-module on which $ G $ acts continuously; set $ d=\text{dim}_{\mathbb{F}}V_{\mathbb{F}} $.
		\item $ \widehat{\mathcal{C}}_{W(\mathbb{F})} $: the category whose objects are complete Noetherian local $W(\mathbb{F}) $-algebras, with a fixed isomorphism of the residue field to $ \mathbb{F} $, and whose morphisms are local $ W(\mathbb{F}) $-algebra homomorphisms. 
		\item If $ A\in \widehat{\mathcal{C}}_{W(\mathbb{F})} $, then we will denote by $\mathfrak{m}_{A}  $ the maximal ideal of $ A $. 
		\item $ \mathcal{C}_{W(\mathbb{F})} $: the full subcategory of $ \widehat{\mathcal{C}}_{W(\mathbb{F})} $ of finite local Artin $ W(\mathbb{F}) $-algebras.
		\item $ \text{ad}=\text{End}_{\mathbb{F}}(V_{\mathbb{F}})\cong V_{\mathbb{F}}\otimes_{\mathbb{F}}V_{\mathbb{F}}^{\ast} $ is the adjoint representation of $ V_{\mathbb{F}} $ where $ V^{\ast} $ is the dual of $ V $; it is again a $ G $-module.
	\end{itemize}
\end{Convention}

\begin{Def}\cite[Section 1.1]{MR3184335}
	Let $ A\in \mathcal{C}_{W(\mathbb{F})} $. A \emph{deformation} of $ V_{\mathbb{F}} $ to $ A $ is a pair $ (V_{A},\iota_{A}) $ such that
	\begin{enumerate}
		\item $ V_{A} $ is an $ A[G] $-module which is finite free over $ A $ and on which $ G $ acts continuously;
		\item $ \iota_{A} $ is a $ G $-equivariant isomorphism $ \iota_{A}:V_{A}\otimes_{A}\mathbb{F}\cong V_{\mathbb{F}} $.
	\end{enumerate}
	Let $\text{(Sets)} $ be the category of sets. One defines functor $ D_{V_{\mathbb{F}}}:\mathcal{C}_{W(\mathbb{F})}\to \text{(Sets)} $ by setting, for $ A\in \mathcal{C}_{W(\mathbb{F})} $,
	\[ D_{V_{\mathbb{F}}}(A):=\{\text{isomorphism classes of deformations of $ V_{\mathbb{F}} $ to $ A $}\}, \]
	and with the obvious extension to morphism.
\end{Def}

We say that a profinite group $ G $ satisfies the \emph{finiteness condition} $ \Phi_{p} $, if for all open subgroup $ G'\subset G $, the $ \mathbb{F}_{p} $-vector space $ \text{Hom}_{\text{cont}}(G',\mathbb{F}_{p}) $ of continuous group homomorphisms is finite-dimensional. For example, the Galois group $ G_{K,S} $ satisfies $ \Phi_{p} $ for any number field $ K $ and any finite set of primes of $ K $.

\begin{Prop}\cite[Proposition 1.3.1 and Theorem 2.4.1]{MR3184335}\label{existenceofdeformationring}
	Assume that $ G $ satisfies condition $ \Phi_{p} $ and $ V_{\mathbb{F}} $ is absolutely irreducible. Then $ D_{V_{\mathbb{F}}} $ is prorepresentable by some $ R_{V_{\mathbb{F}}}\in \widehat{\mathcal{C}}_{W(\mathbb{F})} $. Moreover, the ring $ R_{V_{\mathbb{F}}} $ is topologically generated over $ W(\mathbb{F}) $ by the values $ \text{Tr}(\rho^{\text{univ}}(g)) $ as $ g $ runs over any dense subset of $ G $ where $ \rho^{\text{univ}}:G\to {\rm GL}_{d}(R_{V_{\mathbb{F}}}) $ is the universal deformation and $ \text{Tr}(\rho^{\text{univ}}(g)) $ denotes the trace of $ \rho^{\text{univ}}(g) $.
\end{Prop}
We will say that $ R_{V_{\mathbb{F}}} $ is the \emph{universal deformation ring} of $ V_{\mathbb{F}} $.

\begin{Prop}\cite[Proposition 1.5.1]{MR3184335}\cite[Theorem 4.2]{MR1860043}\label{presentation of universal ring}
	Suppose that $ G $ satisfies condition $ \Phi_{p} $ and $ V_{\mathbb{F}} $ is absolutely irreducible. Then the universal deformation ring $ R_{V_{\mathbb{F}}} $ of $ V_{\mathbb{F}} $ has a presentation 
	\[ W(\mathbb{F})[[X_{1},\cdots,X_{h^{1}(G,\text{ad})}]]/(f_{1},\cdots,f_{h^{2}(G,\text{ad})}) \]
	where $h^{i}(\cdots):=\dim_{\mathbb{F}}H^{i}(\cdots) $ for $ i=1,2 $. Moreover, one has
	\begin{equation}\label{dimen}
		\text{Krulldim}(R_{V_{\mathbb{F}}}/(p))\geq h^{1}(G,\text{ad})-h^{2}(G,\text{ad}),
	\end{equation}
	where $ 	\text{Krulldim}(R_{V_{\mathbb{F}}}/(p)) $ denotes the Krull dimension of the ring $ R_{V_{\mathbb{F}}}/(p) $.
\end{Prop}

\begin{Def}\label{delta}
	We define
	\[ \delta(\text{ad}(V_{\mathbb{F}})):=h^{2}(G,\text{ad})-h^{1}(G,\text{ad}). \]
\end{Def}

In \cite[Lecture 4]{MR1860043}, Gouvea conjectured that:
\begin{conj}[Dimension Conjecture]\label{dimensionconjecture}
	In Proposition \ref{presentation of universal ring},  we always have equality in (\ref{dimen}). That is, we have 
	\[ -\delta(\text{ad}(V_{\mathbb{F}}))=	\text{Krulldim}(R_{V_{\mathbb{F}}}/(p)). \]
\end{conj}
\begin{Rem}
	\begin{enumerate}
		\item It is known that the conjecture fails for finite groups and $ G=G_{K,S} $ for real quadratic fields $ K $ and $ S=\emptyset $, see \cite[Theorem 1.1]{MR2285736}.
		\item 	If $ G=G_{K,S} $ for a number field $ K $ and a finite set of primes $ S $ with $ S\supset S_{p}\cup S_{\infty} $, then the conjecture is still open. 
	\end{enumerate}
\end{Rem}

\subsection{Finiteness of unramified deformation rings}\label{section6.2}
\subsubsection{Statements}
In this subsection, we prove the following theorems.
\begin{Thm}\label{finitenessofunramifieddeforamtionring}
	Suppose that Conjecture \ref{UFMforcharp} holds. Let $ K $ be a number field and $ S $ a finite set of primes of $ K $ such that $ S\cap S_{p}=\emptyset $. If $ \overline{\rho}:G_{K,S}\to {\rm GL}_{n}(\mathbb{F}) $ is a continuous absolutely irreducible representation, then the universal deformation ring $ R_{\overline{\rho}} $ is finite over $ W(\mathbb{F}) $.
\end{Thm}

\begin{Thm}\label{main}
	Let $ K $ be a number field and $ S $ a finite set of primes of $ K $ such that $S\cap S_p=\emptyset$. Let $\overline{\rho}:G_{K,S}\to {\rm GL}_{n}(\mathbb{F})$ be a continuous absolutely irreducible representation. Suppose that the Galois representation associated to any $\overline{\mathbb{Q}}_{p}$-points of the universal deformation ring $R_{\overline{\rho}}$ of $\overline{\rho}$ has finite image. Then the ring $R_{\overline{\rho}}[1/p]=\prod_{x}E_{x}$ is the finite direct product of fields $E_{x}$ where $E_{x}$ is a finite extension of $\mathbb{Q}_{p}$ indexed by $\overline{\mathbb{Q}}_{p}$-points of $R_{\overline{\rho}}$. In particular, there are only finitely many $\overline{\mathbb{Q}}_{p}$-points of $R_{\overline{\rho}}$, i.e. the set $\text{Hom}_{W(\mathbb{F})}(R_{\overline{\rho}},\overline{\mathbb{Q}}_{p})$  of continuous $W(\mathbb{F})$-algebra homomorphisms is finite. Moreover, assume further that $R_{\overline{\rho}}$ is $p$-torsion-free, then $R_{\overline{\rho}}$ is finite over $W(\mathbb{F})$, and the universal deformation $\rho^{\text{univ}}:G_{K,S}\to {\rm GL}_{n}(R_{\overline{\rho}})$ has finite image.
\end{Thm}
\begin{Rem}
	\begin{enumerate}
		\item Conjecture \ref{UFM} predicts that the Galois representation associated to any $\overline{\mathbb{Q}}_{p}$-points of the universal deformation ring $R_{\overline{\rho}}$ of $\overline{\rho}$ always has finite image, and Conjecture \ref{BUFM} predicts that the universal deformation has finite image.
		\item Without the assumption that $R_{\overline{\rho}}$ is $p$-torsion-free, our result above cannot give any information about the ring $R_{\overline{\rho}}$ being finite over $W(\mathbb{F})$, as can be seen from the example $W(\mathbb{F})[[X]]/(pX)$, which satisfies our conclusion but fails to be finite over $W(\mathbb{F})$.
	\end{enumerate}
\end{Rem}

We also consider the generic fiber of the unramified deformation ring associated to a tame Galois representation. In this direction, we prove the following theorem.

\begin{Thm}\label{tame}
	Let $ K $ be a number field and $ S $ a finite set of primes of $ K $ such that $S\cap S_p=\emptyset$. Suppose that $\overline{\rho}:G_{K,S}\to {\rm GL}_{n}(\mathbb{F})$ is a continuous absolutely irreducible representation which is tame, i.e. the order of $\text{Im}(\overline{\rho})$ is prime to $p$. Assume that the generic fiber $R_{\overline{\rho}}[1/p]$ of the universal deformation ring $R_{\overline{\rho}}$ has no nontrivial idempotents. Then $R_{\overline{\rho}}[1/p]=W(\mathbb{F})[1/p]$. In particular, $R_{\overline{\rho}}$ has a unique $\overline{\mathbb{Q}}_{p}$-point, i.e. the trivial lift $\rho_{0}$ of $\overline{\rho}$ and hence the Galois representation associated to it has finite image, i.e. Conjecture \ref{UFM} holds in this case.
\end{Thm}

\subsubsection{Proof of Theorem \ref{finitenessofunramifieddeforamtionring}}
\begin{Lem}\label{nakayama}
	We keep the notation as in Convention \ref{notai}. Let $ A\in \widehat{\mathcal{C}}_{W(\mathbb{F})} $. Suppose that the ring $A/(p)$ is finite. Then $A$ is a finite $W(\mathbb{F})$-algebra.
\end{Lem}
\begin{proof}
	Note that $A$ is separated for the $p$-adic topology. Then our claim follows from \cite[Theorem 8.4]{MR0879273}.
\end{proof}

\begin{proof}[Proof of Theorem \ref{finitenessofunramifieddeforamtionring}]
	By Proposition \ref{existenceofdeformationring}, we see that the ring $ R_{\overline{\rho}} $ is topologically generated over $ W(\mathbb{F}) $ by the traces $ \text{Tr}(\rho^{\text{univ}}(g)),~g\in G_{K,S}, $ where $ \rho^{\text{univ}}:G_{K,S}\to {\rm GL}_{n}(R_{\overline{\rho}}) $ is the universal deformation. It follows that the ring $ R_{\overline{\rho}}/(p) $ is generated over $ \mathbb{F} $ by the traces $ \text{Tr}(\widetilde{\rho^{\text{univ}}}(g)),~g\in G_{K,S},$ where
	\[ \widetilde{\rho^{\text{univ}}}:G_{K,S}\xrightarrow{\rho^{\text{univ}}}{\rm GL}_{n}(R_{\overline{\rho}})\xrightarrow{\pi}{\rm GL}_{n}(R_{\overline{\rho}}/(p)) \]
	is the composition of $ \rho^{\text{univ}} $ and the surjective homomorphism $ \pi: {\rm GL}_{n}(R_{\overline{\rho}}) \to {\rm GL}_{n}(R_{\overline{\rho}}/(p)) $ induced by the natural surjective ring homomorphism $ R_{\overline{\rho}} \twoheadrightarrow R_{\overline{\rho}}/(p) $. Note that the ring $ R_{\overline{\rho}}/(p) $ is the universal deformation ring for characteristic $ p $ deformation of $ R_{\overline{\rho}} $. Assuming Conjecture \ref{UFMforcharp}, we know that $ \widetilde{\rho^{\text{univ}}} $ has finite image by Theorem \ref{reduction}. This implies that $ R_{\overline{\rho}}/(p) $ is finite, and hence $ R_{\overline{\rho}} $ is finite over $ W(\mathbb{F}) $ by Lemma \ref{nakayama}. This completes the proof of our theorem.
\end{proof}

\begin{Rem}
	Allen and Calegari proved the finiteness of unramified deformation rings in many cases without assuming Conjecture \ref{UFMforcharp}. See \cite[Theorem 1]{MR3294389}.
\end{Rem}

\subsubsection{Proof of Theorem \ref{main} and Theorem \ref{tame}}
Let $G$ be a profinite group satisfing the condition $ \Phi_{p} $ and let $\overline{\rho}:G\to {\rm GL}_{n}(\mathbb{F})$ be a continous absolutely irreducible representation. Let $R_{\overline{\rho}}$ denote the universal deformation ring of $\overline{\rho}$ and $\rho^{\text{univ}}:G\to {\rm GL}_{n}(R_{\overline{\rho}})$ denote the universal deformation. In the following, for simplicity, we will write $R$ for $R_{\overline{\rho}}$, and $\mathcal{O}$ for $W(\mathbb{F})$.

\begin{Lem}\cite[10.5.7]{MR238860}\cite[p. 78-79]{de1995crystalline}
	If $R$ is a complete noetherian local $\mathcal{O}$-algebra with finite residue field $\mathbb{F}$, then $R[1/p]$ is a Jacobson ring, and for all maximal ideal $\mathfrak{m}$ of $R[1/p]$, the quotient $R[1/p]/\mathfrak{m}$ is a finite extension of the fraction field $E:=\mathcal{O}[1/p]$ of $\mathcal{O}$. 
\end{Lem}

Fix a $E$-algebra map $x:R[1/p]\twoheadrightarrow R[1/p]/\mathfrak{m}_{x}=E_{x}$ for a finite extension $E_{x}/E$ of fields where $\mathfrak{m}_{x}$ is a maximal ideal of $R[1/p]$. Let
\[\rho_{x}:G\xrightarrow{\rho^{\text{univ}}}{\rm GL}_{n}(R)\to {\rm GL}_{n}(R[1/p])\twoheadrightarrow{\rm GL}_{n}(E_{x})\]
be the specialized representation.

\begin{Thm}\cite[Section 9]{kisin2003overconvergent}
	Let $ R[1/p]_{\mathfrak{m}_{x}}^{\wedge}$ denote completion of the ring $R[1/p]_{\mathfrak{m}_{x}}$ at the maximal ideal $\mathfrak{m}_{x}$. Let $\rho_{x}^{\text{univ}}:G\to {\rm GL}_{n}(R[1/p]_{\mathfrak{m}_{x}}^{\wedge})$ be induced from $\rho^{\text{univ}}$ by the natural map $R\to R[1/p]_{\mathfrak{m}_{x}}^{\wedge}$. Then we have the following commutative diagram
	\[ % https://tikzcd.yichuanshen.de/#N4Igdg9gJgpgziAXAbVABwnAlgFyxMJZABgBpiBdUkANwEMAbAVxiRAHEQBfU9TXfIRQAmclVqMWbbrxAZseAkVHDx9Zq0QgZfBYKIBGMdXVStAMx1z+iociMG1kzSADm3cTCiv4RUOYAnCABbJDIQHAgkAx5-INDEIwioxGFYkECQsOpIpDSKLiA
	\begin{tikzcd}
		G \arrow[rr] \arrow[rrdd] & \rho_{x}^{\text{univ}}& {\rm GL}_{n}(R[1/p]_{\mathfrak{m}_{x}}^{\wedge}) \arrow[dd] \\
		& \rho_{x}&               \\
		&   & {\rm GL}_{n}(E_{x})           
	\end{tikzcd} \]
	and $\rho_{x}^{\text{univ}}$ is the universal continuous deformation of $\rho_{x}$. 
	
	Moreover, $R[1/p]_{\mathfrak{m}_{x}}^{\wedge}$ is the quotient of a power series over $E_{x}$ in $H^{1}(G,\rho_{x}\otimes_{E_{x}}\rho_{x}^{\ast})$ variables, by an ideal generated by at most $H^{2}(G,\rho_{x}\otimes_{E_{x}}\rho_{x}^{\ast}) $ elements where $\rho_{x}^{\ast}$ is the dual of $\rho_{x}$.
\end{Thm}

Below, we will apply this to our case where $G=G_{K,S}$ is the Galois group of the maximal extension of $ K $ unramified outside $ S $, and $ S $ a finite set of primes of $ K $ not containing any prime above $ p $. 

\begin{proof}[Proof of Theorem \ref{main}]
	First of all, we may assume that the ring $R[1/p]$ is non-zero, since otherwise the statement is trivial. For any maximal ideal $\mathfrak{m}_{x}$ of $R[1/p]$, the continuous representation $\rho_{x}:G_{K,S}\to {\rm GL}_{n}(E_{x})$ will factor through a finite quotient $G_{x}$ by our assumption. If $H_{x}$ denotes the kernel of $G\to G_x$, then we have
	\[ H^{1}(G,\rho_{x}\otimes_{E_{x}}\rho_{x}^{\ast})\simeq H^{1}(H_{x},\rho_{x}\otimes_{E_{x}}\rho_{x}^{\ast})^{G_x}=0, \]
	where the last equality follows from the fact that $H^1(H_{x},\mathbb{Q}_{p})=\text{Hom}(H_{x}^{\text{ab}},\mathbb{Z}_{p})=0$, since the Galois group $G_{K,S}$ is FAb by Theorem \ref{global structure}. In other words, we have $R[1/p]_{\mathfrak{m}_{x}}^{\wedge}=E_{x}$ for each $\mathfrak{m}_{x}$. In particular, we see that $R[1/p]_{\mathfrak{m}_{x}}\subset E_{x}$ with Krull dimension
	\[ \text{Krulldim}(R[1/p]_{\mathfrak{m}_{x}})=\text{Krulldim}(R[1/p]_{\mathfrak{m}_{x}}^{\wedge})=\text{Krulldim}(E_{x})=0. \]
	Thus, every $R[1/p]_{\mathfrak{m}_{x}}$ is actually a field. It follows that each maximal ideal of $R[1/p]$ is also a minimal prime ideal, and hence $R[1/p]$ has only finitely many maximal ideals because the ring $R[1/p]$ is noetherian. Since $R[1/p]$ is Jacobson, the minimal prime ideals of $R[1/p]$ are exactly the maximal ideals of $R[1/p]$. Moreover, we have 
	\[ E_{x}=R[1/p]/\mathfrak{m}_{x}\simeq R[1/p]_{\mathfrak{m}_{x}}.  \]
	It follows that $R[1/p]$ is reduced, and we have a natural isomorphism of rings
	\[ R[1/p]\simeq \prod_{x} R[1/p]/\mathfrak{m}_{x}=\prod_{x}E_{x}, \]
	where $x$ runs over all maximal ideals of $R[1/p]$.
	
	Finally, if $R$ is $p$-torsion-free,then $R$ is flat over $\mathcal{O}$. Since the set $\text{Hom}_{\mathcal{O}}(R,\overline{\mathbb{Q}}_{p})$ is finite, we see that $R$ is finite over $\mathcal{O}$. By \cite[Prop. 10]{MR3294389}, we see that the universal deformation $\rho^{\text{univ}}:G_{K,S}\to {\rm GL}_{n}(R)$ has finite image. This completes the proof of our theorem.
\end{proof}

To prove Theorem \ref{tame}, we need the following lemma.
\begin{Lem}\label{easy}
	Let $A$ be a Noetherian commutative ring, and $\mathfrak{m}$ a maximal ideal of $A$. Suppose that $A$ has no nontrivial idempotents and $A_{\mathfrak{m}}$ is a field. Then $\mathfrak{m}$ is the zero ideal.
\end{Lem}
\begin{proof}
	Since $A_{\mathfrak{m}}$ is a field, we have $\mathfrak{m}A_{\mathfrak{m}}=0$. Since $A$ is Noetherian, we can write $\mathfrak{m}=\text{Ann}(t)$ for some $t\in A-\mathfrak{m}$ where $\text{Ann}(t)$ denotes the annihilator of $t$. Then the principal ideal $At$ of $A$ is a simple $A$-module. Since $t\in A-\text{Ann}(t)$, we have $t^2 \neq 0$, and $At^2=At$. It follows that $t=at^2$ for some $a \in A$. Put $e=at$. Then $e^2=e$. By our assumption, $e=at=1$ implies that $t$ is a unit of $A$, and $\mathfrak{m}=\text{Ann}(t)=0 $ is the zero ideal.
\end{proof}

\begin{proof}[Proof of Theorem \ref{tame}]
	Since the order of $\text{Im}(\overline{\rho})$ is prime to $p$, we have the trivial lift $\rho_{0}$ of $\overline{\rho}$ to ${\rm GL}_{n}(\mathcal{O})$ by Schur-Zassenhasu theorem \cite[Prop. 2.3.3]{MR2599132}, i.e. the one with $\text{Im}(\rho_{0})\simeq \text{Im}(\overline{\rho})$. Then we have a $E$-algebra map corresponding to $\rho_{0}$
	\[ x_{0}:R[1/p] \twoheadrightarrow R[1/p]/\mathfrak{m}_{x_{0}}=E. \]
	The same argument in the proof of Theorem \ref{main} shows that $R[1/p]_{\mathfrak{m}_{x_{0}}}^{\wedge}=E$. It follows that $R[1/p]_{\mathfrak{m}_{x_{0}}}$ is a field. By our assumption and Lemma \ref{easy}, we obtain that $R[1/p]=E$. We are done.
\end{proof}

\subsection{Deformations of mod $ p $ representations of $ G_{K,S} $ with full images}\label{section6.3}
In this subsection, we study deformations of mod $ p $ representations $ \overline{\rho} $ of the global Galois group $ G_{K,S} $ with big image where $ S\cap S_{p}=\emptyset $.

\subsubsection{Statements}
Let $ G $ be a profinite group and $ \mathbb{F} $ a finite field of characteristic $ p $. If $ n\geq 2 $, then we say that a representation $ \rho:G\to {\rm GL}_{n}(\mathbb{F}) $ of $ G $ is \emph{full} if its image contains $ {\rm SL}_{n}(\mathbb{F}) $.

Recall that if $ \rho:G\to {\rm GL}_{n}(\overline{\mathbb{Q}}_{p})$ is a $ p $-adic representation of a profinite group $ G $, then there is a finite extension $ L/\mathbb{Q}_{p} $ such that $ \text{Im}(\rho)\subset {\rm GL}_{n}(L) $ and there exists a $ g\in {\rm GL}_{n}(L) $ such that $ g^{-1}\rho g$ has image in $ {\rm GL}_{n}(\mathcal{O}_{L}) $. (However, the choice of $ g $ is not unique.) Then we can define the \emph{reduction} of $ \rho $ to be
\[ \overline{\rho}:G\to {\rm GL}_{n}(\mathcal{O}_{L})\twoheadrightarrow {\rm GL}_{n}(\mathcal{O}_{L}/\mathfrak{m}_{L}). \]
Note that this definition depends on the choice of $ g $ and it is only well-defined up to semi-simplification. The main result in this section is the following theorem. 

\begin{Thm}\label{dimensionalconjecturefails}
	Let $ n\geq 2 $ and $ p\geq 5 $. Suppose $p>5$ for $n=2$. Let $ K $ be a number field and $ S $ a finite set of primes of $ K $ such that $ S\cap S_{p}=\emptyset $. Then we have the following.
	\begin{enumerate}
		\item Assume that Conjecture \ref{UFM} holds. There is no $ p $-adic representation of $ G_{K,S} $ such that its reduction is full.
		\item Assume that both Conjecture \ref{UFM} and Conjecture \ref{UFMforcharp} hold. If $ \overline{\rho}:G_{K,S}\to {\rm GL}_{n}(\mathbb{F}) $ is a full representation of $ G_{K,S} $, then the universal deformation ring $ R_{\overline{\rho}} $ of $ \overline{\rho} $ is a finite ring.
		\item Assume that both Conjecture \ref{UFM} and Conjecture \ref{UFMforcharp} hold. If $ \overline{\rho}:G_{K,S}\to {\rm GL}_{n}(\mathbb{F}) $ is a full representation of $ G_{K,S} $, then we have $ \delta(\text{ad}(\overline{\rho}))>0 $ where $ \delta(-) $ is defined in Definition \ref{delta}.
	\end{enumerate}
\end{Thm}
\begin{Rem}\label{remove}
	\begin{enumerate}
		\item In $ 2 $ and $ 3 $, we can drop the assumption that Conjecture \ref{UFMforcharp} holds if $ p>2n^{2}-1 $, and $ K $ is a totally real number field by \cite[Theorem 1 and Corollary 2]{MR3294389}.
		\item Note that conditions $ S $ is finite and $ S\cap S_{p}=\emptyset $ are indispensable, otherwise there are counterexamples. Indeed, it is proved that there exists a surjective map $ \rho:G_{\mathbb{Q}}\twoheadrightarrow {\rm SL}_{2}(\mathbb{Z}_{7}) $ unramified at $ 7 $ and ramified at infinitely many primes other than $ 7 $ in \cite[Corollary 13]{MR2154368}. Also, there exists a surjective map $ G_{\mathbb{Q}}\twoheadrightarrow {\rm SL}_{2}(\mathbb{Z}_{7})  $ which is ramified at $ 7 $ and finitely ramified but its reduction $  \overline{\rho} $ is unramified at $ 7 $, see \cite[Section 8]{MR1719819}.
	\end{enumerate}
\end{Rem}

\begin{Cor}\label{counterexample}
	Assume that both Conjecture \ref{UFM} and Conjecture \ref{UFMforcharp} hold. If $ p\geq 7 $ is a prime, $ K $ is a number field, $ S $ is a finite set of primes of $ K $ such that $ S\cap S_{p}=\emptyset $ and $ \overline{\rho}:G_{K,S}\to {\rm GL}_{n}(\mathbb{F}) $ is a full representation, then Conjecture \ref{dimensionconjecture} fails for $ \overline{\rho} $.
\end{Cor}
\begin{Rem}
	As in Remark \ref{remove}, if $ p>2n^{2}-1 $, and $ K $ is a totally real number field, then we can drop the assumption that Conjecture \ref{UFMforcharp} holds.
\end{Rem}

\subsubsection{Proofs}
To prove Theorem \ref{dimensionalconjecturefails}, we need the following facts.

\begin{Lem}\cite[Lemma 2, Appendix]{MR2004460}\label{22}
	Suppose a commutative ring $ R $ has a presentation $ W(\mathbb{F})[[X_{1},\cdots,X_{r}]]/(f_{1},\cdots,f_{s}) $. 
	If $ R/(p) $ is finite and if $ r\geq s $, then $ r=s $, and $ R $ is a complete intersection and finite flat over $ W(\mathbb{F}) $.
\end{Lem}

We keep the notation as in Convention \ref{notai}. Let $ A\in \widehat{\mathcal{C}}_{W(\mathbb{F})} $, and we denote by $ W(\mathbb{F})_{A} $ the image of the natural local homomorphism $ W(\mathbb{F})\to A $.

\begin{Thm}\cite[Main theorem]{MR3336600}\label{fullimage}
	Let $ A\in \widehat{\mathcal{C}}_{W(\mathbb{F})} $. Let $ G $ be a closed subgroup of $ {\rm GL}_{n}(A) $. Assume that
	\begin{itemize}
		\item The cardinality of $ \mathbb{F} $ is least $ 4 $. Furthermore, assume that $ \mathbb{F}\neq \mathbb{F}_{5} $ if $ n=2 $ and that $ \mathbb{F}\neq \mathbb{F}_{4} $ if $ n=3 $.
		\item The residual image of $ G $ in $ {\rm GL}_{n}(A/\mathfrak{m}_{A}) $ contains $ {\rm SL}_{n}(\mathbb{F}) $.
	\end{itemize}
	Then $ G $ contains a conjugate of $ {\rm SL}_{n}(W(\mathbb{F})_{A}) $.
\end{Thm}

\begin{proof}[Proof of Theorem \ref{dimensionalconjecturefails}]
	We prove our theorem as follows:
	\begin{enumerate}
		\item Let $ \rho:G_{K,S}\to {\rm GL}_{n}(\mathcal{O}_{L}) $ be a $ p $-adic representation where $ L/\mathbb{Q}_{p} $ is a finite extension. Suppose that the reduction $ \overline{\rho} $ of $ \rho $ is full. Then by Theorem \ref{fullimage} we see that $ \text{Im}(\rho) $ contains a conjugate of $ {\rm SL}_{n}(W(\mathbb{F})) $. In particular, $ \rho $ has infinite image, which contradicts Conjecture \ref{UFM}. Thus, such $ \rho$ does not exist.
		\item Assume that both Conjecture \ref{UFM} and Conjecture \ref{UFMforcharp} hold. By Theorem \ref{reductionofBUFM}, we see that Conjecture \ref{BUFM} holds.
		
		Let $ R_{\overline{\rho}} $ be the universal deformation ring  of $ \overline{\rho} $ and $ \rho^{\text{univ}}:G_{K,S}\to {\rm GL}_{n}(R_{\overline{\rho}}) $ the corresponding universal deformation. Suppose that $ \text{char}(R_{\overline{\rho}})=0 $. Then by Theorem \ref{fullimage} we see that $ \text{Im}(\rho^{\text{univ}}) $ contains a conjugate of $ {\rm SL}_{n}(W(\mathbb{F})_{R_{\overline{\rho}}}) $. Since $\text{char}(R_{\overline{\rho}})=0$, we know that the ring $ W(\mathbb{F})_{R_{\overline{\rho}}} $ is infinite, and hence $ \rho $ has infinite image. But this contradicts Conjecture \ref{BUFM}. Therefore, we obtain that $ \text{char}(R_{\overline{\rho}})>0 $. It follows that $p^{r}R_{\overline{\rho}}=0$ for some positive integer $r$, and the ring $W(\mathbb{F})_{R_{\overline{\rho}}}$ is finite. By Theorem \ref{finitenessofunramifieddeforamtionring}, $ R_{\overline{\rho}} $ is finite over $ W(\mathbb{F}) $ and hence $ R_{\overline{\rho}} $ is a finite ring. 
		\item Let $ R_{\overline{\rho}} $ be the universal deformation ring  of $\overline{\rho} $.
		Suppose that $ \delta(\text{ad}(\overline{\rho}))\leq 0 $. By Theorem \ref{finitenessofunramifieddeforamtionring}, we see that $ R_{\overline{\rho}} $ is finite over $ W(\mathbb{F}) $, and hence $ R_{\overline{\rho}}/(p) $ is finite. Write $ h^{i}=h^{i}(G_{K,S},\text{ad}(\overline{\rho})) $ for $ i=1,2 $. By Proposition \ref{presentation of universal ring}, $ R_{\overline{\rho}} $ has a presentation 
		\[ W(\mathbb{F})[[X_{1},\cdots,X_{h^{1}}]]/(f_{1},\cdots,f_{h^{2}}), \]
		with $ h^{1}\geq h^{2} $ because $ \delta(\text{ad}(\overline{\rho}))\leq 0 $ by our assumption. By Lemma \ref{22}, we see that $ R_{\overline{\rho}} $ is finite flat over $ W(\mathbb{F}) $, which contradicts the previous assertion $ 2 $. Therefore, we must have $  \delta(\text{ad}(\overline{\rho}))> 0 $. This completes the proof of our theorem.
	\end{enumerate}
\end{proof}

\begin{proof}[Proof of Corollary \ref{counterexample}]
	Let $ \overline{\rho}:G_{K,S}\to {\rm GL}_{n}(\mathbb{F}) $ be a full representation and let $ R_{\overline{\rho}} $ be the universal deformation ring  of $ \overline{\rho} $. By Theorem \ref{dimensionalconjecturefails}, we know that $ \delta(\text{ad}(\overline{\rho}))>0 $ and $ \text{Krulldim}(R_{\overline{\rho}}/(p))=0 $. Thus, we have
	\[ -\delta(\text{ad}(\overline{\rho}))<0= \text{Krulldim}(R_{\overline{\rho}}/(p)).\]
	That is, Conjecture \ref{dimensionconjecture} fails.
\end{proof}

\begin{acknowledgements}
Part of this paper is the author’s Ph.D. thesis at Humboldt-Universität zu Berlin. The author deeply thanks his supervisor Prof. Dr. Elmar Große-Klönne for all the helpful discussions. The author would like to thank the external examiners of the thesis for their valuable feedback. The author thanks Prof. Nigel Boston, Prof. John Labute, Prof. Dr. Andrei Jaikin-Zapirain, Prof. Ravi Ramakrishna, and Prof. Patrick Allen for helpful correspondence. 
\end{acknowledgements}

\bibliographystyle{amsalpha}
\bibliography{FM}

\providecommand{\bysame}{\leavevmode\hbox to3em{\hrulefill}\thinspace}
\providecommand{\MR}{\relax\ifhmode\unskip\space\fi MR }
\providecommand{\MRhref}[2]{%
  \href{http://www.ams.org/mathscinet-getitem?mr=#1}{#2}
}
\providecommand{\href}[2]{#2}
\begin{thebibliography}{DdSMS99}

\bibitem[AC14]{MR3294389}
Patrick~B. Allen and Frank Calegari, \emph{Finiteness of unramified deformation
  rings}, Algebra Number Theory \textbf{8} (2014), no.~9, 2263--2272.
  \MR{3294389}

\bibitem[APRT22]{abdellatif2022fontaine}
Ramla Abdellatif, Supriya Pisolkar, Marine Rougnant, and Lara Thomas,
  \emph{From {F}ontaine-{M}azur conjecture to analytic pro-$p$ groups--a
  survey}, arXiv preprint arXiv:2205.03558 (2022).

\bibitem[AT74]{andozhskii1974series}
IV~Andozhskii and Vladimir~Mikhailovich Tsvetkov, \emph{On a series of finite
  closed $p$-groups}, Izvestiya Rossiiskoi Akademii Nauk. Seriya
  Matematicheskaya \textbf{38} (1974), no.~2, 278--290.

\bibitem[B\"02]{MR1932455}
Gebhard B\"{o}ckle, \emph{Finiteness conjectures for {$\Bbb F_l[[T]]$}-analytic
  extensions of number fields}, J. Number Theory \textbf{96} (2002), no.~2,
  257--274. \MR{1932455}

\bibitem[B\"13]{MR3184335}
\bysame, \emph{Deformations of {G}alois representations}, Elliptic curves,
  {H}ilbert modular forms and {G}alois deformations, Adv. Courses Math. CRM
  Barcelona, Birkh\"{a}user/Springer, Basel, 2013, pp.~21--115. \MR{3184335}

\bibitem[BC07]{MR2285736}
Frauke~M. Bleher and Ted Chinburg, \emph{Universal deformation rings need not
  be complete intersections}, Math. Ann. \textbf{337} (2007), no.~4, 739--767.
  \MR{2285736}

\bibitem[BG07]{MR2373146}
E.~Breuillard and T.~Gelander, \emph{A topological {T}its alternative}, Ann. of
  Math. (2) \textbf{166} (2007), no.~2, 427--474. \MR{2373146}

\bibitem[Bos91]{MR1079842}
Nigel Boston, \emph{Explicit deformation of {G}alois representations}, Invent.
  Math. \textbf{103} (1991), no.~1, 181--196. \MR{1079842}

\bibitem[Bos99]{MR1681626}
\bysame, \emph{Some cases of the {F}ontaine-{M}azur conjecture. {II}}, J.
  Number Theory \textbf{75} (1999), no.~2, 161--169. \MR{1681626}

\bibitem[CG18]{MR3742760}
Frank Calegari and David Geraghty, \emph{Modularity lifting beyond the
  {T}aylor-{W}iles method}, Invent. Math. \textbf{211} (2018), no.~1, 297--433.
  \MR{3742760}

\bibitem[CR62]{MR0144979}
Charles~W. Curtis and Irving Reiner, \emph{Representation theory of finite
  groups and associative algebras}, Pure and Applied Mathematics, Vol. XI,
  Interscience Publishers (a division of John Wiley \& Sons, Inc.), New
  York-London, 1962. \MR{0144979}

\bibitem[Dar05]{MR2181417}
M.~R. Darafsheh, \emph{Order of elements in the groups related to the general
  linear group}, Finite Fields Appl. \textbf{11} (2005), no.~4, 738--747.
  \MR{2181417}

\bibitem[DdSMS99]{MR1720368}
J.~D. Dixon, M.~P.~F. du~Sautoy, A.~Mann, and D.~Segal, \emph{Analytic
  pro-{$p$} groups}, second ed., Cambridge Studies in Advanced Mathematics,
  vol.~61, Cambridge University Press, Cambridge, 1999. \MR{1720368}

\bibitem[DJ95]{de1995crystalline}
Arthur~J De~Jong, \emph{Crystalline dieudonn{\'e} module theory via formal and
  rigid geometry}, Publications Math{\'e}matiques de l'IH{\'E}S \textbf{82}
  (1995), 5--96.

\bibitem[EM16]{MR3542491}
Timothy Eardley and Jayanta Manoharmayum, \emph{The inverse deformation
  problem}, Compos. Math. \textbf{152} (2016), no.~8, 1725--1739. \MR{3542491}

\bibitem[FM95]{MR1363495}
Jean-Marc Fontaine and Barry Mazur, \emph{Geometric {G}alois representations},
  Elliptic curves, modular forms, \& {F}ermat's last theorem ({H}ong {K}ong,
  1993), Ser. Number Theory, I, Int. Press, Cambridge, MA, 1995, pp.~41--78.
  \MR{1363495}

\bibitem[Gor80]{MR569209}
Daniel Gorenstein, \emph{Finite groups}, second ed., Chelsea Publishing Co.,
  New York, 1980. \MR{569209}

\bibitem[Gou01]{MR1860043}
Fernando~Q. Gouv\^{e}a, \emph{Deformations of {G}alois representations},
  Arithmetic algebraic geometry ({P}ark {C}ity, {UT}, 1999), IAS/Park City
  Math. Ser., vol.~9, Amer. Math. Soc., Providence, RI, 2001, Appendix 1 by
  Mark Dickinson, Appendix 2 by Tom Weston and Appendix 3 by Matthew Emerton,
  pp.~233--406. \MR{1860043}

\bibitem[Gra03]{MR1941965}
Georges Gras, \emph{Class field theory}, Springer Monographs in Mathematics,
  Springer-Verlag, Berlin, 2003, From theory to practice, Translated from the
  French manuscript by Henri Cohen. \MR{1941965}

\bibitem[Gri00]{MR1765119}
R.~I. Grigorchuk, \emph{Just infinite branch groups}, New horizons in pro-{$p$}
  groups, Progr. Math., vol. 184, Birkh\"{a}user Boston, Boston, MA, 2000,
  pp.~121--179. \MR{1765119}

\bibitem[Gro67]{MR238860}
A.~Grothendieck, \emph{\'{E}l\'{e}ments de g\'{e}om\'{e}trie alg\'{e}brique.
  {IV}. \'{E}tude locale des sch\'{e}mas et des morphismes de sch\'{e}mas
  {IV}}, Inst. Hautes \'{E}tudes Sci. Publ. Math. (1967), no.~32, 361.
  \MR{238860}

\bibitem[Gv64]{MR0161852}
E.~S. Golod and I.~R. \v{S}afarevi\v{c}, \emph{On the class field tower}, Izv.
  Akad. Nauk SSSR Ser. Mat. \textbf{28} (1964), 261--272. \MR{0161852}

\bibitem[Har94]{MR1299733}
David Harbater, \emph{Galois groups with prescribed ramification}, Arithmetic
  geometry ({T}empe, {AZ}, 1993), Contemp. Math., vol. 174, Amer. Math. Soc.,
  Providence, RI, 1994, pp.~35--60. \MR{1299733}

\bibitem[HM02]{MR1901171}
Farshid Hajir and Christian Maire, \emph{Unramified subextensions of ray class
  field towers}, J. Algebra \textbf{249} (2002), no.~2, 528--543. \MR{1901171}

\bibitem[Iha83]{MR714470}
Yasutaka Ihara, \emph{How many primes decompose completely in an infinite
  unramified {G}alois extension of a global field?}, J. Math. Soc. Japan
  \textbf{35} (1983), no.~4, 693--709. \MR{714470}

\bibitem[JZ02]{MR1935507}
Andrei Jaikin-Zapirain, \emph{On linear just infinite pro-{$p$} groups}, J.
  Algebra \textbf{255} (2002), no.~2, 392--404. \MR{1935507}

\bibitem[JZ06]{MR2218704}
A.~Jaikin-Zapirain, \emph{On linearity of finitely generated {$R$}-analytic
  groups}, Math. Z. \textbf{253} (2006), no.~2, 333--345. \MR{2218704}

\bibitem[JZN19]{MR3993799}
Andrei Jaikin-Zapirain and Nikolay Nikolov, \emph{An infinite compact
  {H}ausdorff group has uncountably many conjugacy classes}, Proc. Amer. Math.
  Soc. \textbf{147} (2019), no.~9, 4083--4089. \MR{3993799}

\bibitem[Kha03]{MR2004460}
Chandrashekhar Khare, \emph{On isomorphisms between deformation rings and
  {H}ecke rings}, Invent. Math. \textbf{154} (2003), no.~1, 199--222, With an
  appendix by Gebhard B\"{o}ckle. \MR{2004460}

\bibitem[Kis03]{kisin2003overconvergent}
Mark Kisin, \emph{Overconvergent modular forms and the {F}ontaine-{M}azur
  conjecture}, Inventiones mathematicae \textbf{153} (2003), no.~2, 373--454.

\bibitem[KLGP97]{MR1483894}
G.~Klaas, C.~R. Leedham-Green, and W.~Plesken, \emph{Linear pro-{$p$}-groups of
  finite width}, Lecture Notes in Mathematics, vol. 1674, Springer-Verlag,
  Berlin, 1997. \MR{1483894}

\bibitem[KLR05]{MR2154368}
Chandrashekhar Khare, Michael Larsen, and Ravi Ramakrishna, \emph{Constructing
  semisimple {$p$}-adic {G}alois representations with prescribed properties},
  Amer. J. Math. \textbf{127} (2005), no.~4, 709--734. \MR{2154368}

\bibitem[Koc02]{MR1930372}
Helmut Koch, \emph{Galois theory of {$p$}-extensions}, Springer Monographs in
  Mathematics, Springer-Verlag, Berlin, 2002, With a foreword by I. R.
  Shafarevich, Translated from the 1970 German original by Franz Lemmermeyer,
  With a postscript by the author and Lemmermeyer. \MR{1930372}

\bibitem[KW03]{MR1981910}
Mark Kisin and Sigrid Wortmann, \emph{A note on {A}rtin motives}, Math. Res.
  Lett. \textbf{10} (2003), no.~2-3, 375--389. \MR{1981910}

\bibitem[Lab14]{MR3249411}
John Labute, \emph{Linking numbers and the tame {F}ontaine-{M}azur conjecture},
  Ann. Math. Qu\'{e}. \textbf{38} (2014), no.~1, 61--71. \MR{3249411}

\bibitem[Laz65]{MR209286}
Michel Lazard, \emph{Groupes analytiques {$p$}-adiques}, Inst. Hautes
  \'{E}tudes Sci. Publ. Math. (1965), no.~26, 389--603. \MR{209286}

\bibitem[Leb10]{MR2726590}
Philippe Lebacque, \emph{On {T}sfasman-{V}l\u{a}du\c{t} invariants of infinite
  global fields}, Int. J. Number Theory \textbf{6} (2010), no.~6, 1419--1448.
  \MR{2726590}

\bibitem[Man15]{MR3336600}
Jayanta Manoharmayum, \emph{A structure theorem for subgroups of {$GL_n$} over
  complete local {N}oetherian rings with large residual image}, Proc. Amer.
  Math. Soc. \textbf{143} (2015), no.~7, 2743--2758. \MR{3336600}

\bibitem[Mat86]{MR0879273}
Hideyuki Matsumura, \emph{Commutative ring theory}, Cambridge Studies in
  Advanced Mathematics, vol.~8, Cambridge University Press, Cambridge, 1986,
  Translated from the Japanese by M. Reid. \MR{879273}

\bibitem[Maz89]{MR1012172}
B.~Mazur, \emph{Deforming {G}alois representations}, Galois groups over {${\bf
  Q}$} ({B}erkeley, {CA}, 1987), Math. Sci. Res. Inst. Publ., vol.~16,
  Springer, New York, 1989, pp.~385--437. \MR{1012172}

\bibitem[Moo19]{MR3959072}
Ben Moonen, \emph{A remark on the {T}ate conjecture}, J. Algebraic Geom.
  \textbf{28} (2019), no.~3, 599--603. \MR{3959072}

\bibitem[Neu99]{MR1697859}
J\"{u}rgen Neukirch, \emph{Algebraic number theory}, Grundlehren der
  mathematischen Wissenschaften [Fundamental Principles of Mathematical
  Sciences], vol. 322, Springer-Verlag, Berlin, 1999, Translated from the 1992
  German original and with a note by Norbert Schappacher, With a foreword by G.
  Harder. \MR{1697859}

\bibitem[NS07]{MR2276769}
Nikolay Nikolov and Dan Segal, \emph{On finitely generated profinite groups.
  {I}. {S}trong completeness and uniform bounds}, Ann. of Math. (2)
  \textbf{165} (2007), no.~1, 171--238. \MR{2276769}

\bibitem[NS12]{MR2995181}
\bysame, \emph{Generators and commutators in finite groups; abstract quotients
  of compact groups}, Invent. Math. \textbf{190} (2012), no.~3, 513--602.
  \MR{2995181}

\bibitem[NSW08]{MR2392026}
J\"{u}rgen Neukirch, Alexander Schmidt, and Kay Wingberg, \emph{Cohomology of
  number fields}, second ed., Grundlehren der mathematischen Wissenschaften
  [Fundamental Principles of Mathematical Sciences], vol. 323, Springer-Verlag,
  Berlin, 2008. \MR{2392026}

\bibitem[Odl77]{MR441918}
A.~M. Odlyzko, \emph{Lower bounds for discriminants of number fields. {II}},
  T\^{o}hoku Math. J. \textbf{29} (1977), no.~2, 209--216. \MR{441918}

\bibitem[Odl90]{MR1061762}
\bysame, \emph{Bounds for discriminants and related estimates for class
  numbers, regulators and zeros of zeta functions: a survey of recent results},
  S\'{e}m. Th\'{e}or. Nombres Bordeaux (2) \textbf{2} (1990), no.~1, 119--141.
  \MR{1061762}

\bibitem[Pol21]{MR4177534}
Benjamin Pollak, \emph{Ramification in the inverse {G}alois problem}, J. Number
  Theory \textbf{220} (2021), 34--60. \MR{4177534}

\bibitem[PS16]{MR3581178}
Vincent Pilloni and Beno\^{i}t Stroh, \emph{Surconvergence, ramification et
  modularit\'{e}}, Ast\'{e}risque (2016), no.~382, 195--266. \MR{3581178}

\bibitem[Ram99]{MR1719819}
Ravi Ramakrishna, \emph{Lifting {G}alois representations}, Invent. Math.
  \textbf{138} (1999), no.~3, 537--562. \MR{1719819}

\bibitem[Rob00]{MR1760253}
Alain~M. Robert, \emph{A course in {$p$}-adic analysis}, Graduate Texts in
  Mathematics, vol. 198, Springer-Verlag, New York, 2000. \MR{1760253}

\bibitem[RZ10]{MR2599132}
Luis Ribes and Pavel Zalesskii, \emph{Profinite groups}, second ed., Ergebnisse
  der Mathematik und ihrer Grenzgebiete. 3. Folge. A Series of Modern Surveys
  in Mathematics [Results in Mathematics and Related Areas. 3rd Series. A
  Series of Modern Surveys in Mathematics], vol.~40, Springer-Verlag, Berlin,
  2010. \MR{2599132}

\bibitem[Sch11]{MR2810332}
Peter Schneider, \emph{{$p$}-adic {L}ie groups}, Grundlehren der mathematischen
  Wissenschaften [Fundamental Principles of Mathematical Sciences], vol. 344,
  Springer, Heidelberg, 2011. \MR{2810332}

\bibitem[Sha89]{MR977275}
Igor~R. Shafarevich, \emph{Collected mathematical papers}, Springer-Verlag,
  Berlin, 1989, Translated from the Russian. \MR{977275}

\bibitem[Sha00]{MR1765116}
Aner Shalev, \emph{Lie methods in the theory of pro-{$p$} groups}, New horizons
  in pro-{$p$} groups, Progr. Math., vol. 184, Birkh\"{a}user Boston, Boston,
  MA, 2000, pp.~1--54. \MR{1765116}

\bibitem[TV02]{MR1944510}
M.~A. Tsfasman and S.~G. Vl\u{a}du\c{t}, \emph{Infinite global fields and the
  generalized {B}rauer-{S}iegel theorem}, vol.~2, 2002, Dedicated to Yuri I.
  Manin on the occasion of his 65th birthday, pp.~329--402. \MR{1944510}

\bibitem[Zm92]{MR1194787}
E.~I. Zelcprime~manov, \emph{On periodic compact groups}, Israel J. Math.
  \textbf{77} (1992), no.~1-2, 83--95. \MR{1194787}

\end{thebibliography}

\end{document}